\documentclass[reqno,11pt]{amsart}
\usepackage{amsmath,amssymb,mathrsfs,amsthm,amsfonts}
\usepackage[inline]{enumitem} 
\usepackage[usenames,dvipsnames]{xcolor}
\usepackage{hyperref}
\usepackage[percent]{overpic}
\usepackage{comment}
\usepackage{stmaryrd}
\usepackage{algorithm}
\usepackage{algorithmic}
\usepackage{pgfplots}
\usepackage{enumitem}
\usepackage{caption,subcaption}
\usepackage{tikz}
\usetikzlibrary{calc}
\usetikzlibrary{arrows.meta,backgrounds}
\usetikzlibrary{arrows}
 
\hypersetup{
	colorlinks=true, linkcolor=blue,
	citecolor=ForestGreen
}
\usepackage[paper=letterpaper,margin=1in]{geometry}
\usepackage[multiple]{footmisc}
\DeclareMathAlphabet{\mathpzc}{OT1}{pzc}{m}{it}
\newcommand{\ind}{\mathbf{1}}
\newcommand{\sbm}{\mathsf{SBM}}
\usepackage{bbm}
\usepackage{tikz}
\usetikzlibrary{calc}
\usetikzlibrary{arrows.meta,backgrounds}
\usetikzlibrary{arrows}
\usetikzlibrary{decorations.pathmorphing}

\newtheorem{thm}{Theorem}[section]
\newtheorem{cor}[thm]{Corollary}
\newtheorem{prop}[thm]{Proposition}
\newtheorem{lem}[thm]{Lemma}

\theoremstyle{definition}
\newtheorem{defn}[thm]{Definition}

\newtheorem{rk}[thm]{Remark}

\numberwithin{equation}{section}

\newcommand{\mc}{\mathscr}

\newcommand{\R}{\mathbb{R}}
\newcommand{\V}{\mathcal{L}}
\renewcommand{\hat}{\widehat}
\newcommand{\til}{\widetilde}
\renewcommand{\bar}{\overline}

\renewcommand{\Pr}{\mathbb{P}}
	
\newcommand{\Ex}{\mathbb{E}}

\newcommand{\Con}{\mathrm{C}}
\newcommand{\mE}{\mathbf{E}}

\newcommand{\e}{\varepsilon}
\newcommand{\sig}{\tau}

\title[KPZ equation limit of sticky Brownian motion]{KPZ equation limit of sticky Brownian motion}

\author[S.\ Das]{Sayan Das}
\address{S.\ Das,
	Department of Mathematics, Columbia University,
	\newline\hphantom{\quad \ \ S. Das}
	2990 Broadway, New York, NY 10027 USA
}
\email{sayan.das@columbia.edu}

\author[H.\ Drillick]{Hindy Drillick}
\address{H.\ Drillick,
	Department of Mathematics, Columbia University,
	\newline\hphantom{\quad \ \ H. Drillick}
	2990 Broadway, New York, NY 10027 USA
}
\email{hindy.drillick@columbia.edu}

\author[S.\ Parekh]{Shalin Parekh}
\address{S.\ Parekh,
	Department of Mathematics, University of Maryland,
	\newline\hphantom{\quad \ \ S. Parekh}
	4176 Campus Dr, College Park, MD 20742 USA
}
\email{parekh@umd.edu}

\subjclass[2020]{
	Primary 60K37, 
        82B21,	
        82C22,	
	Secondary 60G70. 
}

\keywords{
	Kardar--Parisi--Zhang equation, stochastic heat equation, Howitt-Warren flow, sticky Brownian motion, local time, extreme value theory.
}

\begin{document}
	
	\begin{abstract} We consider the motion of a particle under a continuum random environment whose distribution is given by the Howitt-Warren flow. In the moderate deviation regime, we establish that the quenched density of the motion of the particle (after appropriate centering and scaling) converges weakly to the $(1+1)$ dimensional stochastic heat equation driven by multiplicative space-time white noise. Our result confirms physics predictions and computations in \cite{ldt,bld} and is the first rigorous instance of such weak convergence in the moderate deviation regime. Our proof relies on a certain Girsanov transform and works for all Howitt-Warren flows with finite and nonzero characteristic measures. Our results capture universality in the sense that the limiting distribution depends on the flow only via the total mass of the characteristic measure. As a corollary of our results, we prove that the fluctuations of the maximum of an $N$-point sticky Brownian motion are given by the KPZ equation plus an independent Gumbel on timescales of order $(\log N)^2.$
		
	\end{abstract}
	
	\maketitle
	{
		\hypersetup{linkcolor=black}
		\setcounter{tocdepth}{1}
		\tableofcontents
	}
	
	\section{Introduction}

\subsection{Preface} 
Diffusion in time-dependent random environments has been a subject of intense investigation recently due to its connection with the KPZ universality class \cite{bc}. It is well known that the quenched density of the position of a particle in the diffusive regime (when its location $= O(\sqrt{\mbox{time}})$) converges to the Gaussian distribution with a second-order correction given by the fixed point of the Edwards-Wilkinson (EW) universality class {\cite[Theorem 3.2]{balazs2006random}, \cite[Theorem 1.6]{yu}, see also \cite{ew1,ew2,ew4,ew3,ew5,joseph2019independent,ew6}}. Meanwhile in the large deviation (LD) regime (when location $= O({\mbox{time}})$), the quenched density {is expected to admit} a large deviation principle with linear speed, and the second-order correction is given by the Tracy-Widom (TW) distribution, the one-point marginal of the fixed point in the Kardar-Parisi-Zhang (KPZ) universality class ({proven rigorously only for a few special models in \cite{bc,mark}}). The goal of this paper is to show that for a large class of such diffusions in the moderate deviation (MD) regime (when location $= O({\mbox{time}}^{3/4})$), the crossover distribution {between these two classes}, namely the \textit{KPZ equation}, arises as a second-order correction.

The (1+1) dimensional KPZ equation is a stochastic partial differential equation (SPDE) given by 
\begin{align}
    \label{def:kpz} \tag{KPZ}
\partial_t\mathcal{H}=\tfrac12\partial_{xx}\mathcal{H}+\tfrac12(\partial_x\mathcal{H})^2+\sigma^{1/2}\cdot \xi, \qquad \mathcal{H}=\mathcal{H}_t(x),
\end{align}
where $\sigma>0$ and $\xi=\xi(t,x)$ is a space-time white noise. The KPZ equation was first introduced in \cite{kpz} as a prototypical model for interfaces of random growth. Since then, the model has been studied intensively in both the mathematics and physics literature.
We refer to \cite{FS10, Qua11, Cor12, QS15, CW17, CS20} for some surveys of the mathematical studies of the KPZ equation. 

As an SPDE, \eqref{def:kpz} is ill-posed due to the presence of the non-linear term $(\partial_x\mathcal{H})^2$. One way to make sense of the equation is to consider $\mathcal{Z}:= e^{\mathcal{H}}$ which formally solves the stochastic heat equation (SHE) with multiplicative noise:
\begin{equation}\label{she} \tag{SHE}\partial_t \mathcal{Z} = \tfrac12 \partial_{xx} \mathcal{Z} + \sigma^{1/2} \cdot \mathcal{Z} \xi, \qquad \mathcal{Z}=\mathcal{Z}_t(x).\end{equation}
The SHE is known to be well-posed and has a well-developed solution theory based on the It\^o integral and chaos expansions \cite{Wal86, BC95, Qua11, Cor12}. 
In this paper, we will consider the solution of the \eqref{she} started with Dirac delta initial data $\mathcal{Z}_0(x) = \delta_0(x)$. For this initial data, \cite{flo} established that $\mathcal{Z}_t(x) > 0$
for all $t > 0$ and $x \in \R$ almost surely (see also \cite{mue91}). Thus $\mathcal{H}=\log \mathcal{Z}$ is well-defined and is called the Cole-Hopf solution of the KPZ equation. This is the notion of solution that we will work with in this paper, and it coincides with other existing notions of solutions \cite{Hai13, Hai14, GIP15, GP17, GJ14, GP18}, under suitable assumptions.

 {The work of \cite{ACQ,cldr,dot,ss} demonstrates} that the one-point distribution of the rescaled KPZ equation as $t\to \infty$ goes to the TW distribution, whereas as $t\to 0$ then under a different rescaling, the KPZ equation converges to the Gaussian distribution (the one-point distribution of the EW fixed point). Thus, the KPZ equation serves as a mechanism for crossing over between the EW  and the KPZ universality classes. Going back to the diffusion models in time-dependent random environments, in \cite{ldt} it was argued that these diffusion models are rich enough to admit the KPZ equation as limiting statistics.  By physical arguments {(explained briefly in Section \ref{sec1.4.2})}, \cite{ldt} derived that in the moderate deviation (MD) regime (when location $= O({\mbox{time}}^{3/4})$), \eqref{def:kpz} arises as a second-order correction for the quenched density (see Figure \ref{scheme}). Their heuristic arguments were later supported by \cite{bld} via rigorous moment-level computations for certain integrable discrete and continuous diffusion models.  {More recently, using high-precision numerical simulations, \cite{hass23} provided strong numerical evidence for this KPZ equation limiting behavior.}

\begin{figure}[h!]
    \centering
    \captionsetup{width=.9\linewidth}
    \begin{tikzpicture}[line cap=round,line join=round,>=triangle 45,x=5.5cm,y=2.5cm]
			\draw[line width=1.2pt, decorate, decoration={random steps,segment length=2.5pt,amplitude=1.5pt}] plot[domain=-0.6:2.2] (\x, {exp(-2*(\x)^2)});
            \draw[line width=1pt,-{Latex[length=1mm]}] (0,-0.05)--(0,1.1);
            \draw[line width=1pt,dashed,gray] (0.3,-0.1)--(0.3,1.1);
            \draw[line width=1pt,dashed,gray] (-0.3,-0.1)--(-0.3,1.1);
            \draw[line width=1pt,dashed,gray] (1,-0.1)--(1,0.6);
            \draw[line width=1pt,dashed,gray] (1.3,-0.1)--(1.3,0.6);
            \draw[line width=1pt,dashed,gray] (1.8,-0.1)--(1.8,0.6);
            \draw[line width=1pt,dashed,gray] (2.1,-0.1)--(2.1,0.6);
            \draw[line width=1pt] (-0.6,-0.05)--(2.2,-0.05);
            \draw[line width=1pt,dashed,{Latex[length=2mm]}-{Latex[length=2mm]}] (-0.3,-0.15)--(0.3,-0.15);
            \draw[line width=1pt,dashed,{Latex[length=2mm]}-{Latex[length=2mm]}] (1,-0.15)--(1.3,-0.15);
            \draw[line width=1pt,dashed,{Latex[length=2mm]}-{Latex[length=2mm]}] (1.8,-0.15)--(2.1,-0.15);
            \node at (0,-0.3) {Diffusive regime};
            \node at (1.15,-0.3) {MD regime};
            \node at (1.95,-0.3) {LD regime};
            \node at (1.95,0.85) {$e^{-I(c)t +\operatorname{TW}}$};
            \node at (1.15,0.85) {$e^{-\frac{c^2\sqrt{t}}{2} +\operatorname{KPZ}}$};
            \node at (0,1.2) {$e^{-\frac{x^2}{2t}}\cdot \operatorname{EW}$};
            \node at (0,-0.55) {$x\propto \sqrt{t}$};
            \node at (1.15,-0.55) {$x=c t^{3/4}$};
            \node at (1.95,-0.55) {$x=c t$};
   \end{tikzpicture}
    \caption{Schematic diagram of the quenched density $p(x)$ of the position of a particle at time $t$. The results in the diffusive regime need to be interpreted appropriately and may be found in \cite{yu}. The fluctuations in the Tracy-Widom regime are known only in some exactly solvable cases \cite{mark}, but are conjectured to hold generally. {We expect EW fluctuations whenever $x\ll t^{3/4}$ and TW fluctuations whenever $x\gg t^{3/4}$. The exponent $3/4$ is expected to be the unique exponent where the KPZ equation fluctuations appear.}}
    \label{scheme}
\end{figure}
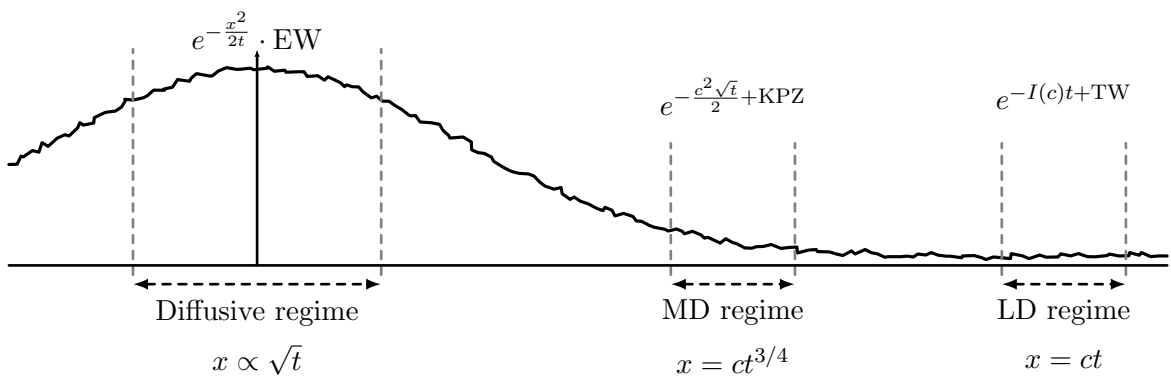

In this paper, we work with diffusion in continuum random environments. We consider the stochastic flow of kernels whose $k$-point motions solve the Howitt-Warren martingale problem \cite{HW09}. {Such stochastic flows of kernels are called Howitt-Warren flows and give rise to diffusions in continuum random environments (see Section \ref{sec1.2} for more details).} We show that the logarithm of the quenched density of the motion of a particle under the Howitt-Warren flow, upon appropriate centering and scaling, converges weakly to the KPZ equation.

    Our work is the first rigorous instance of weak convergence to the KPZ equation for diffusion in time-dependent random media under the moderate deviation regime.  We mention that such weak convergence to the KPZ equation has been shown in the large deviation regime under weak random environment settings \cite{gu,dom}. Our proof techniques rely on a certain Girsanov transform related to sticky Brownian motions (see Section \ref{sec:pfidea} for details). In particular, we do not rely on tools from integrable probability, and our results hold for all Howitt-Warren flows with finite and positive characteristic measures.

\subsection{The model: Sticky Brownian motion} \label{sec1.2} In order to define random
motions in a continuum random environment, we need to introduce the notion of \textit{stochastic flow of kernels}. For $s\le t$, a random probability kernel, denoted $K_{s,t}(x, A)$, is a measurable function defined on some underlying probability space $\Omega$, such that it defines a probability measure on $\R$ for each $(x, \omega) \in \R \times \Omega$. $K_{s,t}(x, A)$ is interpreted as the 
random probability to arrive in $A$ at time $t$ starting at $x$ at time $s$.
 
	\begin{defn}\label{sfok}
		A family of random probability kernels $(K_{s,t})_{s\le t}$ on $\R$ is called a stochastic flow of kernels if
		\begin{enumerate}[label=(\alph*),leftmargin=18pt]
			\item \label{d1} For any $s\le t\le u$ and $x\in \R$, almost surely $K_{s,s}(x,A)=\delta_x(A)$, and
			\begin{align*}
				\int_{\R} K_{t,u}(y,A)K_{s,t}(x,dy)=K_{s,u}(x,A)
			\end{align*}
			for all $A$ in the Borel $\sigma$-algebra of $\R$.
			\item For any $t_1\le t_2 \le \cdots \le t_k$, the $(K_{t_i,t_{i+1}})_{i=1}^{k-1}$ are independent.
			\item For any $s\le u$ and $t\in \R$, $K_{s,u}$ and $K_{s+t,u+t}$ have the same finite dimensional distributions.
		\end{enumerate}
	\end{defn}
The general theory of stochastic flow was developed by Le Jan and Raimond in \cite{lejan}; see also Tsirelson \cite{tsir}. For the stochastic flow of kernels that we consider in this text, one can ensure that the random set of probability $1$ on which \ref{d1} holds is independent of $x\in \R$ and $s\le t\le u$.  This allows us to interpret $(K_{s,t})_{s\le t}$ as bona fide transition
kernels of a random motion in a continuum random environment. The averaged law of such a motion is called the $1$-point motion associated to $(K_{s,t})_{s\le t}$. More generally, the $k$-point motion of a stochastic flow of kernels is defined as the $\R^k$ valued stochastic process $\mathbf{X}=(X^1,\ldots,X^k)$ with transition probabilities given by
\begin{align*}
    P_t(\vec{x},d\vec{y})=\Ex\left[\prod_{i=1}^k K_{0,t}(x_i,dy_i)\right].
\end{align*}
We will be interested in a particular random motion in a continuum random environment originating from the \textit{Howitt-Warren} flow of kernels. Its corresponding $k$-point motion solves a well-posed martingale problem that was first studied by Howitt and Warren in \cite{HW09}. Below, we introduce the $k$-point motion by stating the martingale problem formulated in \cite{sss}.

\begin{defn}\label{hwmp}  {Let $\nu$ be any finite measure on $[0,1]$.} We say an $\R^k$-valued process $\mathbf{X}_t=(X_t^1,\ldots,X_t^k)$ solves the Howitt-Warren martingale problem with characteristic measure $\nu$ if $\mathbf{X}$ is a continuous, square-integrable martingale with the covariance process between $X^i$ and $X^j$ given by
    \begin{align}\label{covd}
        \langle X^i,X^j\rangle_t=\int_0^t \ind_{\{X_s^i=X_s^j\}}ds,
    \end{align}  and furthermore it satisfies the following condition:
    
 Consider any nonempty $\Delta \subset \{1,2,\ldots,k\}$. For $\mathbf{x}\in \R^k$, let
    \begin{align*}
        f_{\Delta}(\mathbf{x}):=\max \{x_i:i\in \Delta\}, \qquad \mbox{and} \qquad g_{\Delta}(\mathbf{x}):=\big|\{i\in \Delta \mid x_i=f_{\Delta}(\mathbf{x})\}\big|.
    \end{align*}
    Then the process
        $f_{\Delta}(\mathbf{X}_t)-\int_0^t \beta_+\big( g_{\Delta}(\mathbf{X}_s)\big)ds$ is a martingale with respect to the filtration generated by $\mathbf{X}$, where $\beta_+(1):=0$ and
    \begin{align*}
        \beta_+(m):={2}\int \sum_{k=0}^{m-2} (1-y)^k \nu(dy), \quad m\ge 2.
    \end{align*}
\end{defn}
Note that from the covariance process formula above, we see that each $X^i$ is marginally a Brownian motion. Focusing on the $k=2$ case, one can check that the last condition in Definition \ref{hwmp} is equivalent to 
\begin{equation} \label{eq:two-point-sbm}
    |X_t^1-X_t^2|-{4}{\nu([0,1])}\int_0^t \ind_{\{X_s^1=X_s^2\}}ds
\end{equation}
being a martingale. 

{We define the local time at $a$ of a continuous semimartingale $X$ as 
\begin{equation}\label{eq:localTimeDef}
    L^X_a(t):= \lim_{\varepsilon \to 0^+}\frac{1}{2 \varepsilon}\int_0^t \mathbbm{1}_{\{a- \varepsilon < X_s < a + \varepsilon\}}d \langle X, X \rangle_s.
\end{equation} Then, using Tanaka's formula we see that \eqref{eq:two-point-sbm} implies that
\begin{align*}
    L_0^{X^1-X^2}(t)={4}{\nu([0,1])}\int_0^t \ind_{\{X_s^1=X_s^2\}}ds.
\end{align*}} Thus, from the above formula, we see that $X^1$ and $X^2$ can be interpreted as Brownian motions evolving independently of each other when apart, but when they meet there is some stickiness. Due to this stickiness, the two motions momentarily move together {in the sense that they are equal on a nowhere-dense set of positive measure}. The $k$-point motion defined above is a generalization of this stickiness phenomenon and is thus referred to as \textit{sticky Brownian motion} ($\sbm$) in the literature (Figure \ref{sticky}).

\begin{figure}[h!]
    \centering
    \captionsetup{width=.9\linewidth}
    \includegraphics[height=5.65cm]{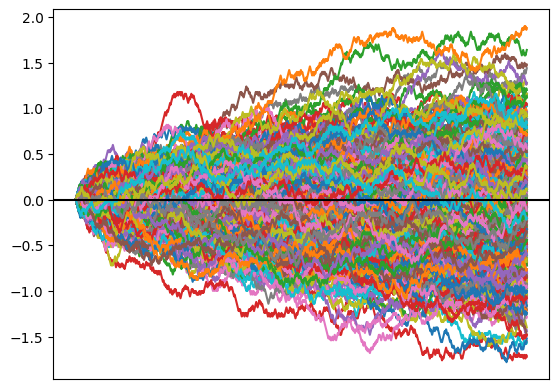}
    \includegraphics[height=5.65cm]{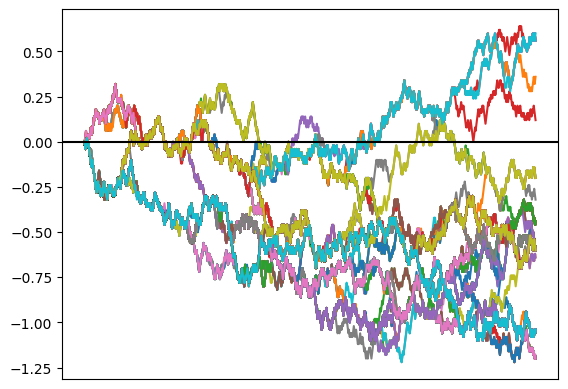}
    \caption{500-point $\sbm$ with weak and strong stickiness simulated from discrete random walks in random environments.}
    \label{sticky}
\end{figure}

The study of Brownian motions with stickiness goes back to the work of Feller \cite{feller}, where he studied general boundary conditions for diffusions on the half line. Since then, sticky Brownian motion has been observed to arise as a diffusive scaling limit of various models: storage processes \cite{hl81}, discrete random
walks in random environments \cite{ami91,hw09b}, and certain families of exclusion processes with a tunable interaction \cite{rs15}. An $\sbm$ with a uniform characteristic measure inherits integrability from the beta random walk in random environment model studied in \cite{bc}. This was exploited in \cite{mark,bld,dom2} to extract various exact formulas and asymptotics. $\sbm$ models also bear connections to the Kraichnan model of turbulent flow \cite{kar68}. Indeed, in the works \cite{gaw,war}, the sticky behavior of particles was observed under certain fine-tuning in the Kraichnan model.  We refer to the series of the physics works \cite{CFKL95, GK95, BGK98, GK96, GV00} and expository notes \cite{ss00,kup10} on the Kraichnan model for more details in this direction.

\smallskip

Howitt and Warren \cite[Proposition 8.1]{HW09} proved that the martingale problem in Definition \ref{hwmp} is well-posed, and its solutions form a consistent family of Feller processes. By a {remarkable} result of Le Jan and Raimond \cite[Theorem 2.1]{lejan}, any consistent family of Feller processes can be viewed as a $k$-point motion of some stochastic flow of kernels, unique in finite-dimensional distributions. Thus, in particular, the solution of the Howitt-Warren martingale problem can be viewed as the $k$-point motion of some stochastic flow of kernels. {{This stochastic flow of kernels is called the Howitt-Warren flow.} When referring to the $k$-point motions, we will continue to use $\sbm$.}
We refer to \cite{sss0,sss,sss2} for more background on how $\sbm$ can be viewed as random motions in continuum random environments, and how to give a concrete construction of such a flow in a space of measure-valued flows using a coupling with the Brownian web and net.
 
\subsection{Main results} Let $K_{0,t}(dx)$ denote the Howitt-Warren flow started from a Dirac mass at $0$ whose characteristic measure $\nu$ is non-degenerate in the sense that $\nu([0,1])>0$. As mentioned before, $K_{0,t}(x)$ can be interpreted as the random probability of a particle hitting $x$ at time $t$. The goal of this paper is to study the density of this quenched probability in the moderate deviation regime where we take $(t,x)\mapsto (Nt,tN^{3/4}+xN^{1/2})$. Formally speaking, we show that the log of the quenched density after appropriate centering:
\begin{align*}
     {\log K_{0,Nt}\big(tN^{3/4}+xN^{1/2}\big)}+{\tfrac12\log N+\tfrac{t}2N^{1/2}+xN^{1/4}}
 \end{align*}
converges in law to the solution $\mathcal{H}_t(x)$ of the KPZ equation  defined in \eqref{def:kpz}.

 \smallskip
 
It is actually the case that $K_{0,t}$ is \textit{singular} with respect to the Lebesgue measure. In fact, by \cite[Theorem 2.8]{sss} it is almost surely purely atomic at deterministic times, and may thus be viewed (formally) as a large system of interacting Brownian particles of different masses that dynamically split and recombine according to a time-homogeneous rule determined by the characteristic measure $\nu$. Since $K$ only exists as a measure and not as a function in general, some care must be taken in order to make sense of the convergence statement written above. To do this, we introduce the fields $\{\mathscr{X}^N\}_{N\ge 1}$ below.

 For $t\ge 0$ and a bounded test function $\phi:\R\to \R$ we first define \begin{equation}\label{a}\mathscr X^N_t(\phi):= \int_{\mathbb R} e^{-\frac{t}{2}N^{1/2} + uN^{-1/4}}\phi(N^{-1/2} (u-N^{3/4}t)) K_{0,Nt}(du)
	\end{equation}
	so that on a purely formal level, one has (via the substitution $x = N^{-1/2}(u-N^{3/4}t)$) \begin{align}
	    \label{scale0}
     \mathscr X^N_t(\phi) = \int_{\mathbb R} \phi(x) N^{1/2} e^{\frac{t}{2}N^{1/2} + xN^{1/4}}K_{0,Nt} (t N^{3/4} + x N^{1/2})dx.
	\end{align}
	The above formally defines a \textit{spatial} pairing of $\mathscr X_t^N$ with $\phi$ in $L^2(\mathbb R)$, and we can also define \textit{space-time} pairings of smooth compactly supported test functions $\varphi:\mathbb R^2\to\mathbb R$ by the formula \begin{equation}\label{b}(\mathscr X^N,\varphi)_{L^2(\mathbb R^2)}:= \int_0^\infty \mathscr X_t^N(\varphi(t,\cdot))dt,
 \end{equation}
{though again we emphasize that $\mathscr X^N$ is not an element of $L^2(\mathbb{R}^2)$ and the subscript here is merely suggestive. Our first result shows that for a fixed $t>0$ and a spatial test function $\phi$, the moments of $\mathscr X_t^N(\phi)$ converge to the moments of the stochastic heat equation paired with $\phi$.} 

	\begin{prop}[Convergence of moments] \label{p:mcov} Fix $t> 0$, and $k\in \mathbb N$. For all $\phi \in \mathcal S(\mathbb R)$ (the Schwartz space on $\R$),  one has the following limit:
		\begin{align}
			\label{e:mcov}
			\lim_{N\to\infty}\mathbf \Ex[\mathscr X_t^N(\phi)^k] = \mathbf E_{B^{\otimes k}}\bigg[\prod_{j=1}^k \phi(B^j_t) e^{{\frac{\sigma}{2}} \sum_{i<j} L^{B^i-B^j}_0(t)} \bigg]= \mathbf E[\mathcal{Z}_t(\phi)^k].
		\end{align} Here $\mathbb E$ denotes the annealed expectation over the environment $\omega$ and $\mathcal{Z}_t(\phi):= \int_\mathbb R \mathcal{Z}_t(x) \phi(x)dx$, where $(t,x)\mapsto \mathcal{Z}_t(x)$ solves \eqref{she} with $\sigma= \frac1{{2}\nu([0,1])}$ under Dirac delta initial condition. The expectation of the middle term is with respect to $k$ independent Brownian motions, and $L_0^Y(t)$ denotes the local time accrued by $Y$ at zero by time $t$. 
	\end{prop}

        Using different methods, a {similar} multipoint moment convergence result is established in \cite{bld} for the case where the characteristic measure $\nu$ is a uniform measure. However, in contrast to the field \eqref{a} that we use in this paper, \cite{bld} used a slight variant which we refer to as the ``quenched tail field." {We refer the reader to Section \ref{sec:qtfresults} where we define the quenched tail field and discuss our results related to it.}

        {We will now describe our weak convergence result for the above field under the appropriate topology. Note that $\mathcal{Z}_t(x)$ is not uniquely characterized by its moments, since they grow too fast. Therefore, the convergence of moments alone will not be enough to establish weak convergence of $\mathscr X^N$. However, Proposition \ref{p:mcov} will still be relevant and will help us to identify the limit points of $\mathscr X^N$ once we show tightness in an appropriate Banach space.}

 We next introduce these suitable topologies for $\mathscr{X}^N$. Fix any $T>0$ and set $\Lambda_T:=[0,T]\times \R$. We denote by $C_c^\infty(\Lambda_T)$ the set of functions $\Lambda_T\to \mathbb R$ that are restrictions to $\Lambda_T$ of some function in $C_c^\infty(\mathbb R^2)$. For $\alpha<0$ we let $r=-\lfloor \alpha\rfloor$ and define the weighted parabolic H\"older space $C^{\alpha,\sig}_\mathfrak s(\Lambda_T)$ to be the closure of $C_c^\infty(\Lambda_T)$ with respect to the norm given by $$\|f\|_{C^{\alpha,\sig}_\mathfrak s(\Lambda_T)}:= \sup_{(t,x)\in\Lambda_T} \sup_{\lambda\in (0,1]} \sup_{\varphi \in B_r} \frac{(f,S^\lambda_{(t,x)}\varphi)_{L^2(\Lambda_T)}}{(1+x^2)^{\sig}\lambda^\alpha}$$ where the scaling operators are defined by $S^\lambda_{(t,x)}\varphi (s,y) = \lambda^{-3}\varphi(\lambda^{-2}(t-s),\lambda^{-1}(x-y))$, and where $B_r$ is the set of all smooth functions of $C^r$ norm (see \eqref{crnorm}) less than $1$ with support contained in the unit ball of $\mathbb R^2$. These spaces are separable and embed naturally into $\mathcal S'(\mathbb R^2)$ (the space of tempered distributions); {see Sections 2 and 3 of \cite{hairer2018multiplicative}}. 

 Similarly, for $\alpha < 0, r=-\lfloor \alpha\rfloor$, we define the weighted elliptic H\"older space $C^{\alpha,\sig}(\mathbb R)$ to be the closure of $C_c^\infty(\mathbb R)$ with respect to the norm given by $$\|f\|_{C^{\alpha,\sig}(\mathbb R)}:= \sup_{x\in\mathbb R} \sup_{\lambda\in (0,1]} \sup_{\phi \in B_r} \frac{(f,S^\lambda_x\phi)_{L^2(\mathbb R)}}{(1+x^2)^{\sig}\lambda^\alpha}$$ where the scaling operators are defined by $S^\lambda_{x}\phi (s,y) = \lambda^{-1}\phi(\lambda^{-1}(x-y))$, and where $B_r$ is now the set of all smooth functions of $C^r$-norm (see \eqref{crnorm}) less than $1$ with support contained in the unit ball of $\mathbb R$. As before, these spaces are separable and embed naturally into $\mathcal S'(\mathbb R)$; {see Sections 2 and 3 of \cite{hairer2018multiplicative}}.
Finally, for a Banach space $\Xi$ we define the function space $C([0,T],\Xi)$ containing all continuous paths $[0,T]\to \Xi$, equipped with a norm given by $\|v\|:= \sup_{t\in [0,T]} \|v(t)\|_{\Xi}.$ In particular, we will consider the spaces $C([0,T],C^{\gamma,\sig}(\mathbb R))$.

Our main result, stated below, shows that the collection $\{\mathscr X^N\}_{N \ge 1}$ converges weakly to the stochastic heat equation when viewed as elements of certain weighted parabolic H\"older spaces or certain function spaces. 
 
  \begin{thm}[Weak Convergence]  Fix any $T>0$, and $\sig>1$.\label{t.weakConv}
  \begin{enumerate}[label=(\alph*), leftmargin=18pt]
      \item Fix any $\alpha<-1$. For each $N\ge 1$, $\mathscr{X}^N$ as defined by \eqref{a} and \eqref{b} can be viewed as an element of {$C_{\mathfrak{s}}^{\alpha,\tau}(\Lambda_T)$}. Furthermore, the collection $\{\mathscr X^N\}_{N \ge 1}$ is tight with respect to the above topology of {$C_{\mathfrak{s}}^{\alpha,\tau}(\Lambda_T)$}. Additionally, any limit point as $N\to \infty$ is concentrated on $C((0,T],C(\mathbb R))$ and coincides with the law of \eqref{she} with $\sigma=\frac1{{2}\nu([0,1])}$,  started from Dirac delta initial condition. \label{t.main} 
      \item Fix any $\gamma<-2$. For each $N\ge 1$, $(\mathscr{X}_t^N)_{t\in [0,T]}$ as defined by \eqref{a} can be viewed as an element of the space {$C([0,T],C^{\gamma,\sig}(\mathbb R))$}. Furthermore, the collection $\{\mathscr X^N\}_{N \ge 1}$ is tight with respect to that topology. Any limit point as $N\to \infty$ is concentrated on $C((0,T],C(\mathbb R))$ and coincides with the law of \eqref{she} with $\sigma=\frac1{{2}\nu([0,1])}$,  started from Dirac delta initial condition. \label{t.main2}
  \end{enumerate}

\end{thm}

\begin{rk}\label{rk.main} A few remarks related to the above theorem are now in order. 
\begin{enumerate}[label=(\alph*),leftmargin=18pt]
    \item\label{finiteconv} {We explain how to interpret the above results in terms of test functions.} For any $\varphi_1,\ldots,\varphi_m \in C_c^\infty(\mathbb R^2)$, either part of Theorem \ref{t.weakConv} implies 
    the joint convergence of $\big((\mathscr X^N,\varphi_j)_{L^2(\mathbb R^2)}\big)_{1\le j\leq m}$ as defined in \eqref{b} to $\big(\int_{\mathbb (0,\infty)\times\mathbb R} \mathcal{Z}_t(x)\varphi_j(t,x)dtdx\big)_{1\le j\leq m}$.
    \item\label{finiteconv2} {Theorem \ref{t.weakConv}\ref{t.main2} implies convergence of $(\mathscr X_t^N(\phi))_{t\in [0,T]}$ to $\big( \int_\mathbb R \mathcal Z_t(x)\phi(x)dx\big)_{t\in [0,T]}$, viewed as random variables in $C[0,T]$, for any $\phi\in C_c^\infty(\mathbb R)$. In particular, one may evaluate the field at some finite collection of fixed times to obtain multi-time convergence.}
    \item {There is a tradeoff between the two parts of the theorem. Theorem \ref{t.weakConv}\ref{t.main} is a statement about the convergence of the field $\mathscr X^N$ when tested against smooth functions in both space and time, and it does not imply convergence in law of $\mathscr X_t^N(\phi)$ for fixed $t>0$. However, $\alpha<-1$ is indeed the optimal H\"older exponent that one could hope to obtain for convergence of the fields $\mathscr X^N$ in the parabolic spaces (the heat kernel itself does not have better regularity). On the other hand, Theorem \ref{t.weakConv}\ref{t.main2} implies convergence of the spatial field for fixed $t>0$, but we believe $\gamma < -2$ is no longer the optimal H\"older exponent for the function space.} 
    \item {The weights $\tau>1$ are not optimal. It should be possible to get rid of the weights altogether, since \eqref{she} started from Dirac initial condition is known to have nice decay properties in both space and time, but some technical aspects of the paper are simplified by using weights.}
\end{enumerate}
\end{rk}

\smallskip

Theorem \ref{t.weakConv} is part of a series of efforts that have sought to show the weak KPZ universality conjecture, which postulates that a large class of weakly asymmetric models rescale to the KPZ equation (see the introduction of \cite{hq18} for a brief background). For instance, convergence to the KPZ equation has been established in a variety of models:  directed polymers in the intermediate disorder regime \cite{akq}, exclusion processes \cite{BG97,ACQ,dembo,ct17,gj17,cst18,cgst20,yang22,yang23}, and a large class of stochastic growth models \cite{hq18,ac22,cha22,yang23b}. In the context of diffusion in time-dependent random environments, \cite{gu} studied nearest
neighbor random walk on $\mathbb{Z}$ in random environments. They showed that under the weak scaling of the environment, the rescaled quenched transition probability evaluated
in the large deviation regime converges to the solution of the stochastic
heat equation. {The work \cite{mark} considered the analogous setting for $\sbm$ with a uniform characteristic measure. They showed that by weakly scaling the environment, which in this case corresponds to rescaling the characteristic measure, the moments of the Howitt-Warren flow evaluated
in the large deviation regime converge to the moments of the stochastic
heat equation.} In a similar spirit, \cite{dom} considered a continuous SPDE model that models the trajectory of a particle in a turbulent fluid. They showed that under a weak environment setting, the limiting fluctuations of the quenched law of the underlying process are given by the KPZ equation.

\smallskip

We emphasize that we do not deliberately introduce any weak asymmetry into our model, i.e., the environment is independent of $N$ and there are no parameters of the model that are being tuned.  Rather, the KPZ fluctuations suggest that the weak asymmetry is somehow introduced naturally as a consequence of the moderate deviation scaling. In fact, by the scaling property of $\sbm$ \cite[Proposition 2.4]{sss}, our result can be converted to a large deviation regime result under weak stickiness, {which would confirm the moment-based predictions for the case of a uniform characteristic measure in \cite{mark}.}
        
\subsubsection{Quenched tail field and connection to extreme value theory} \label{sec:qtfresults}    In this subsection, we describe how KPZ equation convergence can be established for the quenched tail probability from our results on the density field stated in the previous section.
   
\begin{defn}\label{qtf}
    We define the quenched tail field by $$F_N(t,x):= N^{1/4} e^{\frac{t}2 \sqrt{N}+N^{1/4}x} K_{0,Nt}[N^{3/4}t+N^{1/2}x,\infty).$$
\end{defn}
  {The prefactor $N^{1/4}$, as opposed to the $N^{1/2}$ observed in \eqref{scale0}, essentially comes from integration by parts  which absorbs an $N^{1/4}$ factor from the $N^{1/2}$ (see the proof of Proposition \ref{hi})}. We remark that although the $F_N$ are function-valued, they are \textit{discontinuous} functions due to the atomic nature of $K_{0,t}$ \cite[Theorem 2.8]{sss}. Our next theorem states that the family $\{\log F_N\}_{N\ge 1}$ of space-time processes converges to the KPZ equation, in the sense of finite-dimensional distributions of pointwise values $(t,x)$.
        
   \begin{thm}\label{t:qtfconv}
   	For any finite collection of space-time points $\{(t_i,x_i)\}_{i=1}^m\in ((0,\infty)\times \mathbb{R})^m$ one has the joint convergence $$\big( \log F_N(t_i,x_i) \big)_{i=1}^m \stackrel{d}{\to} \big( \mathcal H_{t_i}(x_i)\big)_{i=1}^m,$$
   	where $\mathcal{H}_t(x)$ solves \eqref{def:kpz} with $\sigma=\frac1{{2}\nu([0,1])}$.
   \end{thm}

The above theorem was conjectured in \cite{bld} where the authors established a multipoint moment convergence of the field $F_N(t,x)$ to that of the stochastic heat equation for the case where the characteristic measure $\nu$ is a uniform measure. Again, since the moments do not determine the distribution of the stochastic heat equation, the results in \cite{bld} do not yield Theorem \ref{t:qtfconv} even for the uniform case. 

\smallskip

To prove Theorem \ref{t:qtfconv}, we rely on Theorem \ref{t.weakConv}\ref{t.main2} and an integration by parts argument to first obtain the joint convergence 
\begin{align}\label{e.qtffin1}
	\bigg( \int_\mathbb R \phi_i(x) F_N(t_i,x) dx\bigg)_{i=1}^m \stackrel{d}{\to} \bigg( \int_\mathbb R \phi_i(x) \mathcal Z_{t_i}(x) dx\bigg)_{i=1}^m.
\end{align}
for $\phi_1,\ldots,\phi_m \in C_c^{\infty}(\R)$ and $t_1,\ldots,t_m>0$.
We then establish regularity bounds for the two-point spatial difference of the quenched tail field. This essentially follows from the work of \cite{sss}, \cite{yu}, and \cite{mark}. Given this regularity bound, the finite-dimensional convergence in \eqref{e.qtffin1} can be upgraded to finite-dimensional convergence of pointwise space-time values by an application of Fatou's lemma. The full details of the proof of Theorem \ref{t:qtfconv} are presented in Section \ref{sec:qtailfield}.

\smallskip

  One could go even further and ask about the convergence of $F_N$ in a stronger topology such as the Skorohod topology (recall the $F_N$ are discontinuous), which implies the multipoint result. We do not pursue this in the present paper and leave this as a future work.

\smallskip

{As a consequence of Theorem \ref{t:qtfconv}, we obtain the limiting distribution for the maximum particle of a $k$-point sticky Brownian motion $(X_t^1,\ldots,X_t^k)$ in the regime $t=O((\log k)^2)$.}
  
{\begin{thm}\label{t:max}
Fix $c,t>0$ and $d\in \mathbb R$. Let $(X_t^1,\ldots,X_t^k)$ be a $k$-point sticky Brownian motion with characteristic measure $\nu$.  Set the number of particles $k=k(N):= \lfloor \exp(\frac12cN^{1/2} +dN^{1/4}+r_N)\rfloor$ where $r_N$ can be any sequence satisfying $r_N=o(N^{1/4})$. Then
\begin{align*}
    \max_{1\le i\le k(N)} \big\{N^{-\frac14} X^i(Nt)\big\} -a_N(c,d,t) \stackrel{d}{\to} {\sqrt{\tfrac{t}{c}}\big( G+\mathcal{H}_{c}(d)\big)},
\end{align*}
where
\begin{align*}
   a_N(c,d,t):= \sqrt{ctN} - dN^{\frac14}\sqrt{\tfrac{t}{c}} -\sqrt{\tfrac{c}{t}} \big(r_N-\tfrac14 \log N\big).
\end{align*}
Here $G$ is a Gumbel random variable (i.e., $P(G\le a)=e^{-e^{-a}}$)  which is independent of $\mathcal H$, the solution to \eqref{def:kpz} with $\sigma=\frac{\sqrt{c}}{{2}\nu([0,1])\sqrt{t}}$.
 \end{thm}
We remark that $k$-point sticky Brownian motion refers to the \textit{annealed} law of the $X^i$: we are not making a pathwise statement about the maximum for each realization of the kernels $K_{s,t}$, which is consistent with the fact that Theorem \ref{t.weakConv} is a weak convergence statement, not an almost sure convergence statement. We also remark that rather than allowing $\sigma$ to depend on $t$ and $c$, one may instead take $\sigma = \frac{1}{{2}\nu([0,1])}$ but then $\mathcal{H}_{c}(d)$ must be changed accordingly to $\mathcal{H}_{\frac{t^2}{c}}\big(\frac{td}{c}\big) + \log \big(\frac{t}{c}\big).$}

Taking $d=r_N=0$ and $t=1$, we see that the above statement is a result of the maximum of $e^{\frac12c\sqrt{N}}$ many sticky Brownian motion particles at time $N$.
This is the same as understanding the maximum of $N$ particles when time is of the order $(\log N)^2.$ 
This leads to the question of what happens when (for a fixed characteristic measure $\nu$) one looks at the maximum of $N$ particles at timescales different from $(\log N)^2$. 
At timescales of order $1$, we do not expect universality, as the answer may depend on the characteristic measure. If the characteristic measure satisfies $ \int_{[0,1]} q^{-1}\nu(dq)<\infty$, then the support result of \cite[Theorem 2.5a]{sss} implies that the maximum of $N$ particles at time $t=1$ converges in law (without any centering or scaling) to a Gaussian of mean ${2}\int_{[0,1]} q^{-1}\nu(dq)$.  
If $\int_{[0,1]} q^{-1}\nu(dq)=\infty$, we have no conjecture what happens.

At timescales of order $\log N$, the maximum of $N$ particles fluctuates like $(\log N)^{1/3}$ times a Tracy-Widom distribution. This is conjectured to be universal, but it is currently only provable in certain exactly solvable cases (see \cite[Corollary 5.8]{bc} or \cite{mark}). Finally, on timescales greater than $(\log N)^2$ we believe that the Gumbel term will dominate, since in this regime the motions closely resemble i.i.d.~Brownian motions. We conjecture that $(\log N)^2$ is the unique timescale at which one sees a mix of Gumbel with the KPZ equation. It remains to explore what happens on timescales between 1 and $\log N$, or between $\log N$ and $(\log N)^2$.

The physics paper \cite{hass23} contains numerical simulations which explore these regimes and seem to support some of our conjectures, although they consider the random walk in random environment which is a discrete analogue of our model. The physics paper \cite{alex} also contains interesting conjectures related to timescales slightly larger than $(\log N)^2$.

\subsection{Issues with the chaos expansion technique} \label{1.6} Before explaining the core ideas and novel techniques of the proof, it is important to highlight the constraints of traditional methods used in showing convergence to the \eqref{she}. Among the existing methods, the polynomial chaos method is a widely used approach in establishing weak convergence to the \eqref{she}.
In this method, the prelimiting object is first identified as a sequence of multi-linear polynomials of independent random variables (called polynomial chaos expansions). Then each term of the chaos series is shown to converge in $L^2$ to that of the Wiener chaos series of \eqref{she}. This idea was first implemented by Alberts, Khanin, and Quastel \cite{akq} for directed polymers. Later, \cite{poly} set up a general framework, formulating general conditions under which a polynomial chaos series converges in law to a Wiener chaos expansion. This framework has since been utilized extensively to show that \eqref{she} arises as a limit from several models of interest. In particular, Corwin and Gu \cite{gu} used the framework of \cite{poly} to obtain KPZ equation convergence for the random walk in a random environment (RWRE) model in the large deviation regime under a {weak environment} scaling. Although sticky Brownian motion bears a strong resemblance to the RWRE model and can be realized as its diffusive limit \cite{sss2}, there are two serious obstacles in carrying out the polynomial chaos approach in our setting.

\begin{enumerate}[leftmargin=18pt]
    \item Firstly, it is not clear how to set up the polynomial chaos for the quenched density in the context of a continuum random environment. Indeed, as shown in \cite{jan2004sticky} the noise generated by the Howitt-Warren flow is a \textit{black noise} in the sense of Tsirelson \cite{tsir} (see also \cite{jan2004sticky}). Black noises arise as a scaling limit in various discrete models, such as systems of coalescing random walks \cite{tsirelson2004boris,ellis2016brownian} and 2D critical percolation \cite{scsm}. These non-classical noises are a much more subtle and less understood subject than white noise. A stochastic calculus for black noise is not known, and in particular, there is no notion of iterated stochastic integrals with respect to black noise.

    \item Secondly, even for the discrete RWRE model, it is not straightforward to replicate the ideas of Corwin and Gu \cite{gu} to prove KPZ equation convergence under the moderate deviation regime. Although a polynomial chaos expansion for the quenched density is available in this regime, taking a naive limit of this discrete chaos expansion interestingly gives a noise coefficient in the limiting stochastic heat equation which is \textit{strictly smaller} than the physics prediction from \cite{bld}. {We refer to Section 1.1 of our sequel paper on discrete random walks in dynamical random environments \cite{DDP24} for more details on this}. This suggests that this polynomial chaos does not satisfy the conditions in \cite[Theorem 2.3]{poly} needed to apply their framework. {In this particular scenario, a nonzero proportion of the $L^2$-mass of the polynomial chaos series escapes into the tails of the series in the $N\to\infty$ limit, suggesting that additional independent noise is generated in the limit}.
\end{enumerate}

As far as we know the latter phenomenon has not been observed previously. We study this phenomenon in our upcoming work \cite{DDP24} where we prove a similar KPZ equation universality result for the quenched transition kernel of the RWRE using a similar strategy. Like this paper and in contrast with \cite{gu}, the environment law will be \textit{fixed} under the scalings, and we will focus on the {moderate} deviation setting.

\subsection{Proof Idea}	\label{sec:pfidea} In this section, we describe the broad ideas of the proof of our main theorem: Theorem \ref{t.weakConv}. {We focus on the proof of Theorem \ref{t.weakConv}\ref{t.main}. The proof of Theorem \ref{t.weakConv}\ref{t.main2} will follow readily from Theorem \ref{t.weakConv}\ref{t.main}, together with a short embedding lemma about these H\"older spaces under the action of the heat kernel.} 

\subsubsection{Girsanov's formula} \label{pfidea1} The main technique in our analysis will be a Girsanov-type formula for sticky Brownian motion. For simplicity, we illustrate here the $1$-point case. Using the definition of $\mc{X}_t^{N}(\phi)$ from \eqref{a} we may write 
		\begin{align}
  \label{ee2}
			\mathscr X_t^N(\phi) = \Ex_\omega^{(1)} \left[ e^{-\frac12 t \sqrt{N} +N^{-1/4}X_{Nt}} \cdot \phi\left(N^{-1/2} (X_{Nt}-N^{3/4}t)\right)\right],
		\end{align}
		where $\Ex_\omega^{(1)}$ denotes the quenched expectation with respect to a single motion $X$ sampled from the environment $\omega = \{K_{s,t}: - \infty\leq s\leq t<\infty\}.$  Note that in the annealed sense $K_{0,Nt}$ is simply the law of $B_{Nt}$ for a standard Brownian motion $B$. Thus, taking the annealed expectation on both sides of the above equation, and by the tower property for conditional expectation, we obtain
  \begin{align*}
		\Ex[\mathscr X_t^N(\phi)] & = \mE_{BM} \left[ e^{-\frac12 t \sqrt{N} +N^{-1/4}X_{Nt}} \cdot \phi\left(N^{-1/2} (X_{Nt}-N^{3/4}t)\right)\right] \\ & = \mE_{BM} \left[ e^{-\frac12 t \sqrt{N} +N^{1/4}X_{t}} \cdot \phi\left(X_{t}-N^{1/4}t\right)\right].
\end{align*}
Here the expectations are taken with respect to a standard Brownian motion $B$. In the last line, we used the scale invariance of Brownian motion to say that $X\stackrel{d}{=} (N^{-1/2}X_{Nt})_{t\ge 0}$. Note that $Z_t:=e^{-\frac{t}2\sqrt{N}+N^{1/4}X_t}$ is the stochastic exponential of $N^{1/4}X_t$. By Girsanov's theorem for Brownian motion, under the changed measure $Q(A):=\Ex_{BM}(Z_T\ind_{A})$, the process $(X_t-N^{1/4}t)_{t\in [0,T]}$ is again a Brownian motion. Thus, the last expression in the above equation is precisely equal to $\mathbf E_{BM}[\phi(B_t)]$ which no longer depends on $N$. This matches the first moment of $\int_\mathbb R \mathcal{Z}_t(x)\phi(x)dx$ where $(t,x)\mapsto \mathcal{Z}_t(x)$ is defined in Proposition \ref{p:mcov}.

\smallskip

In the case of higher moments, $\mathscr{X}_t^N(\phi)^k$ can be viewed as the quenched expectation with respect to $k$-point motion sample from the environment $\omega$. Then taking the annealed expectation over the quenched expectation will lead to expressions in terms of a $k$-point sticky Brownian motion. Then the key idea is to use a Girsanov-type formula for sticky Brownian motion (see Lemma \ref{rnd}) to get rid of divergent terms appearing in the annealed expectation expression. The resulting higher moment formulas appear in Lemma \ref{dn}. Unlike the first-moment computation, the resulting expressions for higher moments still depend on $N$. However, the expressions are amenable to taking the large $N$ limit. The expressions obtained in Lemma \ref{dn} are essentially annealed expectations with respect to a ~$k$-point sticky Brownian motion with characteristic measure $N^{1/2}\nu$. As $N\to \infty$, the stickiness disappears, and we are left with expectations with respect to a standard Brownian motion on $\R^{k}$. In Theorem \ref{converge} we compute these limits and show that they indeed match with the moments of \eqref{she} defined in Proposition \ref{p:mcov}.

\smallskip

{Through our method, the $N^{3/4}$ term appearing in the scaling is seen to be the unique and natural choice of exponent that universally gives KPZ fluctuations. Indeed, following \eqref{ee2}, one could potentially consider a model with more general exponents: 
\begin{align*}
	\mathscr X_t^{N,c_N}(\phi) := \Ex_\omega^{(1)} \left[ e^{-\frac12c_N^2 Nt +c_N X_{Nt}} \cdot \phi\left(N^{-1/2} (X_{Nt}-c_NNt)\right)\right],
\end{align*}
 and certain aspects of the paper would still go through. Indeed, following the arguments in the proof of Lemma \ref{dn}, one can check that the second moment is given by 
 \begin{align*}
	\Ex\big[\mathscr X_t^{N,c_N}(\phi)^2\big] = \mE\left[ \phi(X_t)\phi(Y_t)\exp\bigg(c_N^2N \int_0^t\ind_{\{X_s=Y_s\}}ds\bigg)\right],
\end{align*}
 where $(X,Y)$ is a certain tilted version of $2$-point $\sbm$ with characteristic measure $N^{1/2}\nu$. The key observation here is that $c_N=N^{-1/4}$ is the unique choice for which local times appear in the limiting expressions of the intersection times (see Theorem \ref{converge}). When $c_N\ll N^{-1/4}$, the contribution of the intersection times would vanish in the limit, whereas for $c_N\gg N^{-1/4}$ the expressions blow up. }

\smallskip

\subsubsection{Tightness} \label{sec1.4.2}
We now explain the main idea used in proving the tightness of the field $\mathscr X^N$. {Roughly speaking, the original conjecture made in \cite{bld} interpreted the Howitt-Warren flows $K_{0,t}$ as a Kolmogorov forward equation associated to an SDE with drift coefficient formally given by space-time white noise. They then apply a shear transform of space-time given by $(t,x) \mapsto (Nt,N^{3/4}t+N^{1/2}x)$ and note that (at least formally) this transforms the Kolmogorov forward equation into an SPDE which is essentially \eqref{she} plus some term that should vanish in the limit. Their derivation is non-rigorous because such a Kolmogorov SPDE is ill-posed due to the roughness of the noise. The main idea in our proof is to use a rigorous variant of this idea.}

\smallskip

More precisely, in Lemma \ref{qv} we will show that the fields $\mathscr X^N$ satisfy a forced heat equation of the form 
\begin{align}
    \label{feq}
    (\partial_t-\tfrac12\partial_x^2)\mathscr X^N = dM^N
\end{align}
in the sense of space-time Schwartz distributions, where $M^N$ is a \textit{martingale} forcing that is constructed in Section \ref{sec:mp} below. We do not aim to explicitly describe $M^N$ but simply work with it as though \eqref{feq} is the definition. Despite this non-explicit description of $M^N$, we can nonetheless show that the quadratic variations of $M^N$ admit the following nice decomposition:
\begin{equation}\label{edec}
		\langle M^N(\phi)\rangle_t = Q_t^N(\phi^2) +\mathcal E_t^N(\phi),
  \end{equation}
where $Q^N$ is the quadratic variation field introduced in \eqref{QVfield} and $\mathcal{E}_t^N(\phi)$ is an error term defined in \eqref{etn} that goes to zero in $L^2$ norm. Our tightness proof then proceeds in two steps:

\begin{itemize}[leftmargin=20pt]
    \item In the first step, we obtain various moment estimates for $Q^N$.  This is done by the same method outlined in Section \ref{pfidea1}. Indeed, the Girsanov approach allows us to get precise expressions for other relevant observables related to the field $\mathscr X^N$, not just the moments. In particular, it gives us access to moment estimates for $Q^N$ as well (Proposition \ref{tight1}).

    \item The next step is to use the moment estimates for $Q^N$ to obtain tightness estimates on the fields $\mathscr{X}^N$. Indeed, since the fields $M^N$ have a martingale structure, the Burkholder-Davis-Gundy inequality yields moment estimates for $M^N$ from moment estimates for $Q^N$. By Schauder estimates for the heat equation, we may, in turn, translate moment estimates for $M^N$ into tightness estimates for the fields $\mathscr X^N$ using \eqref{feq}.
\end{itemize}

From the above steps, we obtain tightness for the fields $\mathscr{X}^N$, $Q^N$, and $M^N$ in an appropriate topology (see Propositions \ref{xtight} and \ref{mcts}). This roundabout method turns out to be much more tractable than trying to obtain tightness for the fields $\mathscr X^N$ directly, see Proposition \ref{tight1} below. This type of method is somewhat similar to that used in \cite{BG97} where the authors proved KPZ fluctuations for WASEP. 

\subsubsection{Identification of the limit points} After tightness is obtained, it remains to identify the limit points. To do this, we will use the martingale characterization of the solution of the multiplicative noise stochastic heat equation. Specifically, consider a measure $\mu$  on $C([0,T],C(\mathbb R))$, and let $(X_t)_{t\in [0,T]}$ denote the canonical process on that space. The canonical filtration $\mathcal F_t$ on $C([0,T],C(\mathbb R))$ is the one generated by $\{X_s:s\leq t\}$. A result of \cite[Proposition 4.11]{BG97} inspired by the work of \cite{konno} says that if for all $\phi \in C_c^\infty(\mathbb R)$ the processes 
\begin{align}
\label{me1}
    M_t(\phi):= (X_t,\phi)_{L^2(\mathbb R)} - \frac12\int_0^t (X_s,\phi'')_{L^2(\mathbb R)}ds
\end{align}
 are $(\mathcal F_t,\mu)$-martingales with quadratic variation given by 
 \begin{align}\label{me2}
     \langle M(\phi)\rangle_t = \sigma\int_0^t (X_t^2, \phi^2)_{L^2(\mathbb R)}ds,
 \end{align}
 then (under reasonable assumptions on the spatial growth of $X_t$ at infinity) the measure $\mu$ necessarily coincides with the law of \eqref{she} started from an initial condition that is distributed as $X_0$ under $\mu$. 

 Let $(X^\infty, Q^\infty, M^\infty)$ be a limit point of $(\mathscr{X}^N,Q^{N}, M^N)$. Since in the prelimit the observables satisfy \eqref{feq}, from that equation it is not hard to deduce that $(X^{\infty},M^{\infty})$ satisfies \eqref{me1} with $(X,M)=(X^{\infty},M^{\infty})$. To show \eqref{me2}, we again rely on the Girsanov approach to extract moment formulas for certain observables in the prelimit. Using these formulas, loosely speaking, we shall show in Proposition \ref{4.1} that as $N\to \infty$
	\begin{align*}
		Q_t^N(a)-\sigma\int_0^t\left(\mc{X}_s^N(a)\right)^2 ds \stackrel{L^2}{\to} 0
	\end{align*}
 for each $a\in \R\setminus \{0\}$. The precise formulation of the above equation requires more care, as  $Q_t^N$ and $\mc{X}_t^N$ exist only as distributions. Assuming this, thanks to the decomposition in \eqref{edec} and the fact that $\mathcal{E}_t^N(\phi)$ vanishes in the limit, we get \eqref{me2} with $(X,M)=(X^{\infty},M^{\infty})$.

\smallskip

The proof strategy outlined above has the potential to generalize to the random walk in random environment (RWRE) setting. In an upcoming work, we plan to prove a similar KPZ equation universality result for the quenched transition kernel of the RWRE in the moderate deviation regime using the same strategy.

\begin{rk}[Universality]
{From the proof outlined above, we see that only the 1-point and 2-point motions associated with the kernels $K_{s,t}$ are consequential in the limit. The 1-point motion is simply Brownian motion which is why $M_t(\phi)$ as defined above is a martingale, while the 2-point motions appear in the expressions for the quadratic variations of those martingales. 
Indeed the kernels $K_{s,t}$ and their ``squares" $K_{s,t}^{Sq}$ (see \eqref{def:sq}) which appear in the expressions for the martingale $M^N$ and the quadratic martingale field $Q^N$ are purely in terms of the quenched expectations of at most two-point motions of $\sbm$ (see for instance \eqref{qrep}). Since the 1-point and 2-point motions of $\sbm$ are completely determined in law by just the total mass $\nu([0,1])$ of the characteristic measure (this can be seen for the 2-point motion by looking at \eqref{eq:two-point-sbm}), the limiting law in Theorem \ref{t.weakConv} only depends on $\sbm$ via $\nu([0,1])$. In other words, our result is universal in the sense that it yields the same limit for all characteristic measures $\nu$ with the same total mass.}
\end{rk}

 \subsection*{Organization} The rest of the article is organized as follows. In Section \ref{sec:girs} we describe the Girsanov transform and collect estimates related to sticky Brownian motion. In Section \ref{sec:momconv} we prove the moment convergence (Proposition \ref{p:mcov}). In Section \ref{sec:mp} we identify the martingale $M^N$ and show that the field $\mc{X}^N$ satisfies a forced heat equation with forcing $M^N$. Section \ref{sec:qmf} is devoted to analyzing the quadratic variation of $M^N$. Finally, in Section \ref{sec:smp} we prove Theorem \ref{t.weakConv} by utilizing the estimates from the previous sections to obtain tightness estimates and to identify the limit points for the fields $\mc{X}^N$. In Section \ref{sec:qtailfield}, we prove results related to the quenched tail field discussed in Section \ref{sec:qtfresults}. In Appendix \ref{app} we prove a few technical estimates related to Brownian bridges which are used in proving our main theorems.

 \subsection*{Notation and Conventions} Throughout this paper we use $\Con = \Con(a, b, c, \ldots) > 0$ to denote a generic deterministic positive finite
constant depending on $a, b, c, \ldots$ that may change from line to line. We write $\mathcal{S}(\mathbb{R}^d)$ to denote the space of all Schwartz functions on $\R^d$ and use $\mathcal{S}'(\mathbb{R}^d)$ to denote it's dual: the space of all tempered distributions. We write $\Ex$ and $\Ex_{\omega}$ for annealed and quenched expectations in the context of random motions in random environments. We use $\mathbf{E}$ to denote expectation under path measures such as Brownian motion, sticky Brownian motion, etc.

 \subsection*{Acknowledgements} We thank Ivan Corwin, Yu Gu, and Jon Warren for useful discussions related to this project and for comments on an earlier draft of the paper. We thank Rongfeng Sun for helpful discussions related to chaos expansion issues for sticky Brownian motion. We also thank Guillaume Barraquand for several illuminating discussions that eventually led us to derive results related to the quenched tail field.  {We are grateful to the two anonymous referees for their many valuable suggestions.} The project was initiated during the authors' participation in the ``Universality and Integrability in Random Matrix Theory and Interacting Particle Systems'' semester program at MSRI in the fall of 2021. The authors thank the program organizers for their hospitality and acknowledge the support from NSF DMS-1928930. SD and HD's research was partially supported by Ivan Corwin's NSF grant DMS-1811143, the Fernholz Foundation's ``Summer Minerva Fellows'' program, and also the W.M. Keck Foundation Science and Engineering Grant on ``Extreme diffusion''. HD was also supported by the NSF Graduate Research Fellowship under Grant No. DGE-2036197.

	\section{Girsanov transform and sticky Brownian motion estimates}\label{sec:girs}
	
In this section, we develop the basic framework of our proof. As mentioned in the introduction, our proof relies on a certain Girsanov-type transform for sticky Brownian motion. In Section \ref{sec:girt} we describe this transform that we will use repeatedly in our later analysis. In Section \ref{sec:sbmest}, we collect several estimates related to sticky Brownian motion.

	\subsection{Girsanov transform} \label{sec:girt} We begin with some necessary notation and definitions. Throughout this paper we assume $\nu$ is a fixed finite measure on $[0,1]$ with $\nu([0,1])>0$. Fix $T>0$. For each $k\in \mathbb N$, we denote 
	
	\begin{itemize}[leftmargin=18pt]
		\setlength\itemsep{0.5em}
		\item $\mathbf P_{B^{\otimes k}}$ \ : the law on the canonical space $C([0,T],\mathbb R^k)$ of $k$ independent Brownian motions. 
		
		\item $\mathbf P_{SB_\nu^{(k)}}$ \ : the law on the canonical space $C([0,T],\mathbb R^k)$ of the $k$-point motion of a sticky Brownian motion with characteristic measure $\nu$. 
	\end{itemize}

	\medskip

	Given $(X^1,\ldots,X^k)$ distributed as $\mathbf P_{SB_\nu^{(k)}}$, using \eqref{covd} we have
	\begin{align}\label{e:sumv}
		\bigg\langle \sum_{j=1}^k X^j\bigg\rangle(t)=kt+\sum_{i\neq j} \int_0^t \ind_{\{X_s^i=X_s^j\}}ds \le k^2t.
	\end{align}
	Thus by Novikov's criterion for each $\lambda\in \R$
	\begin{align}\label{mart2}
		Z(t):=\exp\bigg[ \lambda \sum_{j=1}^k X^j(s) - \frac{k\lambda^2 t}{2} - \lambda^2\sum_{i<j} \int_0^t \ind_{\{X^i_s = X^j_s\}} ds\bigg].
	\end{align}
	is a martingale.  With this information in hand, we now introduce two more measures in the following definition.
	
	\smallskip
	
	\begin{defn}\label{tmeasure} We denote by $\mathbf P_{T_\lambda SB_\nu^{(k)}}$ the measure which is absolutely continuous with respect to ~$\mathbf P_{SB_\nu^{(k)}}$ 
		with density $\mathbf P_{T_\lambda SB_\nu^{(k)}}=Z(t) \cdot \mathbf P_{SB_\nu^{(k)}}$ on $\mathcal{F}_t$ where $(\mathcal{F}_t)_{t\in [0,T]}$ is the filtration generated by $\mathbf X$. We define the measure $\mathbf P_{D_\lambda SB^{(k)}_\nu}$ to be the law on the canonical space $C([0,T],\mathbb R^k)$ of the process $t\mapsto (X^i_t-\lambda t)_{i=1}^k$ where $(X^i)_{i=1}^k$ is distributed as $\mathbf P_{T_\lambda SB^{(k)}_{\nu}}.$ We shall often refer to them as the $T_\lambda SB_\nu^{(k)}$ and $D_\lambda SB_\nu^{(k)}$ measures.
	\end{defn}

	The following lemma shows that $\mathbf P_{D_\lambda SB^{(k)}_\nu}$ is also absolutely continuous with respect to ~$\mathbf P_{SB^{(k)}_\nu}$ on the interval $[0,T]$ with an explicit Radon-Nikodym derivative. To state the lemma, we introduce a few more pieces of notation. For $\mathbf x\in \mathbb R^k$, define $m_i(\mathbf x)$ to be the cardinality of the set $\{j\leq k: x_i=x_j\}.$ Define
	\begin{align}
		\label{def:gm}
		G(t) := \sum_{j=1}^k \int_0^t \bigg[ 1-\frac{1}{m_j(\mathbf X(s))}\bigg]dX^j(s), \quad \mathcal M(t): = \exp\bigg(\lambda G(t) - \frac{\lambda^2}2\langle G\rangle(t)\bigg),
	\end{align} where $\mathbf X=(X^1,\ldots,X^k)$ is distributed as $\mathbf P_{SB_{\nu}^{(k)}}$.
	
	\begin{lem}\label{rnd} The measure $\mathbf P_{D_\lambda SB_\nu^{(k)}}$ defined in Definition \ref{tmeasure} is absolutely continuous with respect to ~$\mathbf P_{SB_\nu^{(k)}}$ 
		with density $\mathbf P_{D_\lambda SB_\nu^{(k)}}=\mathcal{M}(t) \cdot \mathbf P_{SB_\nu^{(k)}}$ on $\mathcal{F}_t$ where $(\mathcal{F}_t)_{t\in [0,T]}$ is the filtration generated by $\mathbf X$.
	\end{lem}
	
	\begin{proof}  Define the measure $\mathbf P_{Q_\lambda SB^{(k)}_\nu}$ to be the law on the canonical space $C([0,T],\mathbb R^k)$ of the process $t\mapsto (X^i_t-\lambda t)_{i=1}^k$ where $(X^i)_{i=1}^k$ is distributed as $\mathbf P_{SB^{(k)}_{\nu}}.$ By \cite[Lemma 5.4]{HW09}, we know that $\mathbf P_{Q_\lambda SB^{(k)}_\nu}$ is absolutely continuous with respect to ~$\mathbf P_{SB^{(k)}_{\nu}}$ on the interval $[0,T]$ with Radon-Nikodym derivative given by 
		\begin{align}\label{rnd1}
			\left.\frac{d\mathbf P_{Q_\lambda SB^{(k)}_\nu}}{d\mathbf P_{SB^{(k)}_\nu}}\right|_{\mathcal{F}_t}(\mathbf X) = \exp\bigg(-\lambda G_1(t) - \frac{\lambda^2}2\langle G_1\rangle(t)\bigg),
		\end{align}
		where $$G_1(t):= \sum_{j=1}^k \int_0^t \frac1{m_i(\mathbf X(s))}dX^i(s).$$ Take any bounded $\mathcal F_t$-measurable functional $F:C([0,T],\mathbb R^k) \to \mathbb R$. We have 
		\begin{align*}\mathbf E_{D_\lambda SB^{(k)}_\nu}[ F(\mathbf X)] &= \mathbf E_{T_\lambda SB^{(k)}_\nu}\big[F\big( (\mathbf{X}_u - \pmb{\lambda} u)_{0\le u\le t} \big)\big]\\ &=  \mathbf E_{ SB^{(k)}_\nu}\bigg[\exp \bigg( \lambda \sum_{j=1}^k X_t^j - \frac{\lambda^2}2 \bigg\langle \sum_{j=1}^k X^j \bigg\rangle(t) \bigg) F\big( (\mathbf{X}_u - \pmb{\lambda} u)_{0\le u\le t} \big)\bigg] \\ &= \mathbf E_{Q_\lambda SB^{(k)}_\nu}\bigg[\exp \bigg( \lambda \sum_{j=1}^k \big(X_t^j +\lambda t\big)- \frac{\lambda^2}2 \bigg\langle \sum_{j=1}^k X^j \bigg\rangle(t) \bigg) F(\mathbf X)\bigg] \\ &= \mathbf E_{ SB^{(k)}_\nu}\bigg[\exp \bigg( \lambda \sum_{j=1}^k X_t^j-\lambda G_1(t) +\lambda^2 kt- \frac{\lambda^2}2 \bigg\langle \sum_{j=1}^k X^j \bigg\rangle(t) -\frac{\lambda^2}{2}\langle G_1\rangle(t)\bigg) F(\mathbf X)\bigg].
		\end{align*}
		Here $\pmb{\lambda}u$ denotes the vector $(\lambda u, \ldots, \lambda u)\in \R^k$ so that $(\mathbf{X}_u-\pmb{\lambda} u)=(X_u^i-\lambda u)_{i=1}^k \in \R^k$. The first three equalities in the above equation follow directly from Definition \ref{tmeasure} and the definition of $\mathbf P_{Q_{\lambda} SB_{\nu}^{(k)}}$. The last equality uses \eqref{rnd1}.
		Note that $G(t)$, defined in \eqref{def:gm}, equals to $\sum_{j=1}^k X^j(t)-G_1(t)$. Using \eqref{e:sumv} and the fact that $\langle G_1,X^j\rangle(t)=t$ for each $j$, we observe that the last expression in the above equation is precisely equal to $\mathbf E_{ SB^{(k)}_\nu}[\mathcal{M}(t) \cdot F(\mathbf X)]$ where $\mathcal{M}(t)$ is defined in \eqref{def:gm}. This completes the proof.
	\end{proof}
	
	\begin{rk} In view of Lemma \ref{rnd}, a bit of definition chasing shows that for all bounded $\mathcal{F}_t$-measurable $F$
		\begin{align*}
\mE_{SB_{\nu}^{(k)}}\left[Z(t)F\big(\left(\mathbf{X}_u-\pmb{\lambda} u\right)_{0\le u\le t}\big)\right]=  \mE_{SB_{\nu}^{(k)}}\left[\mathcal{M}(t)F\big((\mathbf{X}_u)_{0\le u\le t}\big)\right],
		\end{align*}
		where $Z$ and $\mathcal{M}$ are defined in \eqref{mart2} and \eqref{def:gm} respectively.  Using the same argument, one also has
		\begin{align*}
			\mE_{SB_{\nu}^{(k)}}\left[\frac{Z(t)}{Z(s)}F\big(\left(\mathbf{X}_u-\pmb{\lambda} (u-s)\right)_{s\le u\le t}\big)\mid \mathcal{F}_s\right]=  \mE_{SB_{\nu}^{(k)}}\left[\frac{\mathcal{M}(t)}{\mathcal{M}(s)}F\big(\left(\mathbf{X}_u\right)_{s\le u\le t}\big)\mid \mathcal{F}_s\right]
		\end{align*}
  		almost surely. 	This implies that for all bounded $\mathcal{F}_t$-measurable $H$
		\begin{equation}
			\label{e:condeq2}
			\begin{aligned}
				& \mE_{SB_{\nu}^{(k)}}\left[\exp\bigg(\lambda\sum_{j=1}^k (X_t^j-X_s^j)-\tfrac{k}{2}\lambda^2(t-s)\bigg)H\big(\left(\mathbf{X}_u-\pmb{\lambda} (u-s)\right)_{s\le u\le t}\big)\mid \mathcal{F}_s\right] \\ & =  \mE_{SB_{\nu}^{(k)}}\left[e^{\lambda(G(t)-G(s))-\tfrac{\lambda^2}{2}[\langle G \rangle (t)-\langle G \rangle (s)]} H\big(\left(\mathbf{X}_u\right)_{s\le u\le t}\big)\exp\bigg(\lambda^2\sum_{i<j}\int_{s}^t \ind_{\{X_u^i=X_u^j\}}du\bigg)\mid \mathcal{F}_s\right]\!.\!
			\end{aligned}
		\end{equation}
	\end{rk}
	
\subsection{Sticky Brownian motion estimates}\label{sec:sbmest} In this section, we gather various estimates related to sticky Brownian motion that are necessary for our later proofs.
	In our analysis, we will often encounter intersection times of a $k$-point sticky Brownian motion. For ease of notation, for each $k \in \mathbb N$ and $1\leq i\neq j \leq k$ we define the functional $V^{ij}:C([0,T],\mathbb R^k)\to C([0,T],[0,\infty)) $ by 
  \begin{align}
      \label{vijai}
      V^{ij}_t(\mathbf X):= \int_0^t \ind_{\{X^i_s=X^j_s\}}ds.
  \end{align}
We remark that the processes $V^{ij}$ are Borel-measurable and adapted to the canonical filtration on $C([0,T],\mathbb R^k)$. The following two lemmas record moment estimates related to these functionals. 
	\begin{lem}\label{traps}
		There exists an absolute constant $\Con>0$ such that for all $k >1$, $\theta,t>0$ and for all characteristic measures $\nu$ we have $$ \mathbf E_{SB_\nu^{(k)}}\bigg[\exp{\bigg(\theta \sum_{i=1}^k \bigg[\big(\sup_{s\leq t}|X^i_s|\big) +{4}\nu([0,1])\sum_{j\neq i}V^{ij}_t(\mathbf{X}) }\bigg]\bigg)\bigg] \le \Con \cdot \exp(\Con \cdot k^4\theta^2 t).$$
	\end{lem}
	
	\begin{proof}
		Applying H\"older's inequality we get 
		\begin{equation}
			\label{e:twoterm}
			\begin{aligned}
				& \mathbf E_{SB_\nu^{(k)}}
				\bigg[\exp{\bigg(\theta \sum_{i=1}^k \big(\sup_{s\leq t}|X^i_s|\big) +
					{4}\nu([0,1])\sum_{j\neq i}V^{ij}_t(\mathbf{X})] }\bigg)\bigg]
				\\ & \hspace{0.5cm}\leq \prod_{i=1}^k\mathbf E_{SB_\nu^{(k)}}
				\bigg[\exp\bigg(2k\theta \sup_{s\le t}|X^i_s|\bigg)\bigg]^{\frac1{2k}} \cdot 
				\prod_{i=1}^k\mathbf  E_{SB_\nu^{(k)}} \big[ e^{{8}k\theta \nu([0,1]) \sum_{j\neq i}V^{ij}_t(\mathbf{X})}\big]^{\frac1{2k}}.
			\end{aligned}
		\end{equation}
		We now proceed to bound the above two product terms individually. Note that marginally each $X^i$ is a Brownian motion under $\mathbf P_{SB_\nu^{(k)}}$. For a Brownian motion, one has exponential tail estimates of the form 
		\begin{align}
			\label{e:btail}
			\mathbf P_{B}\bigg[\sup_{s\le t}|X_s|>a\bigg] \leq 2e^{-a^2/2t}.
		\end{align}
		We thus have
		\begin{align}
			\label{e:1term}
			\prod_{i=1}^k\mathbf E_{SB_\nu^{(k)}}\bigg[\exp\bigg(2k\theta \sup_{s\le t}|X^i_s|\bigg)\bigg]^{\frac1{2k}}= \mathbf E_{B} \bigg[ \exp\bigg(2k\theta \sup_{s\le t}|X_s|\bigg)\bigg]^{\frac12}\leq \Con e^{\Con k^2\theta^2 t},
		\end{align}
		for some absolute constant $\Con>0$.
		
		\medskip
		
		To deal with the second product in \eqref{e:twoterm}, another application of the H\"older inequality on each of the expectation terms in the product shows that for each $i=1,\ldots,k$ we have $$\mathbf E_{SB_\nu^{(k)}} \bigg[ \exp\bigg({8}k\theta \nu([0,1]) \sum_{j\neq i}V^{ij}_t(\mathbf{X}) \bigg)\bigg]\leq \prod_{j\neq i}\mathbf E_{SB_\nu^{(k)}} \bigg[ \exp\bigg({8}k(k-1)\theta \nu([0,1]) \int_0^t \ind_{\{X_s^i=X_s^j\}}ds \bigg)\bigg]^{\frac1{k-1}}.$$ 
		Notice that for all $i\neq j$ the pair $(X^i,X^j)$ is  a 2-point sticky Brownian motion under $\mathbf P_{SB_\nu^{(k)}}$ and therefore the latter product is the same as $$\mathbf E_{SB_\nu^{(2)}} \bigg[ \exp\bigg({8}k(k-1)\theta \nu([0,1]) \int_0^t \ind_{\{X_s^1=X_s^2\}}ds \bigg)\bigg]. $$ Summarizing, we find that 
		\begin{align*}
			\prod_{i=1}^k\mathbf  E_{SB_\nu^{(k)}} \big[ e^{{8}k\theta \nu([0,1]) \sum_{j\neq i}V^{ij}_t(\mathbf{X})}\big]^{\frac1{2k}}  \le \mathbf E_{SB_\nu^{(2)}} \bigg[ \exp\bigg({8}k(k-1)\theta \nu([0,1]) \int_0^t \ind_{\{X_s^1=X_s^2\}}ds \bigg)\bigg]^{\frac12}.
		\end{align*}
  
		Note that for a 2-point sticky Brownian motion $(X,Y)$, recall from \eqref{eq:two-point-sbm} that the process ${4}\nu([0,1])\int_0^t \ind_{\{X_s^1=Y_s\}}ds$ agrees with the local time $L_0^{X-Y}(t).$  By \cite{ito1963brownian}, $X-Y$ can be viewed as a time changed Brownian motion. Indeed, given a Brownian motion $B$, {if we define 
			\begin{align}\label{timeT}
				T(t)=t+\frac1{{4}\nu([0,1])}L_0^B(t),
			\end{align}
			the process $\sqrt{2} B_{T^{-1}(\cdot)}$ has the same distribution as $X-Y$. Note that for $0\le s\le t$ we have
   \begin{align*}
			T(T^{-1}(s)+t-s) & =T^{-1}(s)+t-s+\tfrac{1}{{4}\nu([0,1])}L_0^{X-Y}(T^{-1}(s)+t-s)) \\ & \ge T^{-1}(s)+t-s+\tfrac{1}{{4}\nu([0,1])}L_0^{X-Y}(T^{-1}(s))) \ge s+t-s=t.
		\end{align*}
   Writing $t=T(T^{-1}(t))$, we see that
   \begin{align}\label{tinc}
       T^{-1}(t)-T^{-1}(s) \le t-s.
   \end{align}
    This means that the time increments of the time-changed process are slower than that of the standard time. In particular, $T^{-1}(t)\le t$,  
    hence $L_0^{X-Y}(t)$ is stochastically dominated by $L_0^{\sqrt{2}B}(t)={\sqrt2}L_0^B(t)$. By Levy's identity for Brownian local time, we have $L_0^B(t)\stackrel{d}{=}\max_{s\le t} B(t)$ for each fixed $t>0$}.   Using \eqref{e:btail} we thus have
		\begin{align*}
			\prod_{i=1}^k\mathbf  E_{SB_\nu^{(k)}} \big[ e^{{8}k\theta \nu([0,1]) \sum_{j\neq i}V^{ij}_t(\mathbf{X})}\big]^{\frac1{2k}} & \le \mathbf E_{SB_\nu^{(2)}} \bigg[ \exp\bigg({8}k(k-1)\theta \nu([0,1]) \int_0^t \hspace{-0.1cm}\ind_{\{X_s^1=X_s^2\}}ds \bigg)\bigg]^{\frac12} \\ & \le \Con e^{\Con k^2(k-1)^2\theta^2t}
		\end{align*}
		for some absolute constant $\Con>0$. Inserting the above bound and the bound in \eqref{e:1term} back in \eqref{e:twoterm} we get the desired bound. This completes the proof. 
	\end{proof}
	
	\begin{lem}\label{lte}
		Fix $k\in\mathbb N$ and $p\ge 1$. There exists a constant $\Con=\Con(p,k)>0$ such that for all $t>s\ge 0$ 
		$$\sup_{\nu\in \mathcal M([0,1])}  \mathbf E_{SB_\nu^{(k)}}\bigg[\bigg({4}\nu([0,1])\sum_{i<j}(V^{ij}_t(\mathbf X)-V^{ij}_s(\mathbf X))\bigg)^p \bigg] \leq \Con|t-s|^{p/2},$$
  where $\mathcal M([0,1])$ is the space of finite and non-negative Borel measures on $[0,1].$
	\end{lem}
	
	\begin{proof} Use Minkowski's inequality and the fact that the 2-point motions of $\mathbf P_{SB_\nu^{(k)}}$ are distributed as $\mathbf P_{SB_\nu^{(2)}}$ to write 
		\begin{align}
			\nonumber
			\mathbf E_{SB_\nu^{(k)}}\bigg[\bigg({4}\nu([0,1])\sum_{i<j}(V^{ij}_t(\mathbf X)-V^{ij}_s(\mathbf X))\bigg)^p \bigg]^{1/p} & \leq \sum_{i<j}\mathbf E_{SB_\nu^{(k)}}\bigg[\bigg({4}\nu([0,1])(V^{ij}_t(\mathbf X)-V^{ij}_s(\mathbf X))\bigg)^p \bigg]^{1/p} \\ \nonumber &= \sum_{i<j}\mathbf E_{SB_\nu^{(2)}}\bigg[\bigg({4}\nu([0,1])\int_s^t\ind_{\{X_u=Y_u\}}du\bigg)^p \bigg]^{1/p} \\ &= \frac{k(k-1)}2 \mathbf E_{SB_\nu^{(2)}}\bigg[\big(L_0^{X-Y}(t)-L_0^{X-Y}(s)\big)^p \bigg]^{1/p}. \label{tr}
		\end{align}
		The last identity follows from the fact that for a 2-point sticky Brownian motion $(X,Y)$ the process ${4}\nu([0,1])\int_0^t \ind_{\{X_s^1=Y_s\}}ds$ agrees with the local time $L_0^{X-Y}(t)$. Recall that $X-Y \stackrel{d}{=} \sqrt{2}B_{T^{-1}(\cdot)}$ where $B$ is a standard Brownian motion and $T$ is defined in \eqref{timeT}. {It follows from \eqref{tinc} that $L_0^{X-Y}(t)-L_0^{X-Y}(s)$ is stochastically dominated by ${\sqrt{2}}(L_0^{B}(t)-L_0^{B}(s))$. By Levy's identity for Brownian local time, we have  $L_0^{B}(t)-L_0^{B}(s)\stackrel{d}{=}M(t)-M(s)$ where $M(u):=\max_{v\le u} B(v)$. Since $\Ex[|M(t)-M(s)|^p] \le \Ex [\sup_{s\leq u \leq t} |B(u)-B(s)|^p] = \Con |t-s|^{p/2}$ for some constant $\Con>0$ depending on $p$}, we have $\mbox{\eqref{tr}} \le \Con |t-s|^{1/2}$ for some constant $\Con>0$ depending on $p,k$. This completes the proof.
	\end{proof}
	We end this section by recording a uniform exponential moment estimate for a certain class of $\mathbf P_{SB^{(k)}_{N^{1/2}\nu}}$-martingales that will be very important.
	\begin{prop}\label{exp}Fix $k\in\mathbb N$. Let $\mathcal G_N=\mathcal G_N(\mathbf X)$ be any family of continuous $\mathbf P_{SB^{(k)}_{N^{1/2}\nu}}$-martingales satisfying
		\begin{align}\label{e:gcond}
			\langle \mathcal G_N\rangle (t) \leq \Con N^{1/2} \sum_{i<j}V_t^{ij}(\mathbf X)
		\end{align}
		for some deterministic constant $\Con>0$ independent of $N$. For all $p,t>0$ we have 
		\begin{align}
			\label{e:exp}
			\sup_{N\ge 1}\mathbf E_{SB^{(k)}_{N^{1/2}\nu}} \big[e^{p\mathcal G_N(t) - \frac{p}2 \langle \mathcal G_N\rangle(t)}\big]<\infty.
		\end{align}
	\end{prop}
	
	\begin{rk}\label{rk:mn} Suppose $\mathbf{X}=(X^1,\ldots,X^k)$ is distributed as $\mathbf P_{SB_{N^{1/2}\nu}^{(k)}}$. Consider $G$ defined in \eqref{def:gm}. Clearly $\mathcal{G}_N:=N^{1/4}G$ satisfies \eqref{e:gcond}. Let us write
		\begin{align}
  \label{mnt}
			\mathcal{M}_N(t):=\exp\left(N^{1/4}G(t)-\tfrac12N^{1/2}\langle G\rangle (t)\right).
		\end{align}
		Note that $\mathcal{M}_N$ is the same as $\mathcal{M}$ defined in \eqref{def:gm} with $\lambda\mapsto N^{1/4}$ and $\nu\mapsto N^{1/2}\nu$. The above proposition shows that the $L^p$ norms of $\mathcal{M}_N(t)$ are uniformly bounded.
	\end{rk}
	
	\begin{proof}[Proof of Proposition \ref{exp}] For convenience, we write $\mathbf E_N$ for $\mathbf E_{SB^{(k)}_{N^{1/2}\nu}}$. Due to the hypothesis \eqref{e:gcond}, for every $\lambda>0$ we have
		\begin{align}\label{rt}
			\sup_{N\ge 1}\mathbf E_N\left[e^{\lambda \langle \mathcal{G}_N\rangle (t)}\right] \le \sup_{N\ge 1}\mathbf E_N\bigg[\exp\bigg(\lambda \Con N^{1/2}\sum_{i<j} V_t^{ij}(\mathbf X)\bigg)\bigg]<\infty.
		\end{align}
		where the last inequality is due to Lemma \ref{traps} (applied with $\nu \mapsto N^{1/2}\nu$). Thus,
		\begin{equation}
			\label{e:ineq}
			\begin{aligned}
				\mathbf E_N\big[e^{p\mathcal G_N(t) - \frac{p}2 \langle \mathcal G_N\rangle(t)}\big] & \le \mathbf E_N\big[e^{p\mathcal G_N(t)}\big] \le \mathbf E_N\big[e^{2p^2\langle \mathcal{G}_N\rangle (t)}\big]^{1/2}.
			\end{aligned}
		\end{equation}
		The last inequality will be explained shortly. Assuming it is true, taking supremum over $N\ge 1$ on both sides of \eqref{e:ineq}, in view of \eqref{rt}, we get \eqref{e:exp}. 

  To prove the last inequality in \eqref{e:ineq}, we note more generally that uniformly over all continuous real-valued martingales $M$ defined on a probability space $(\Omega, \mathcal F,\mathbb P)$ with $M_0=0$, one has the bound $$\mathbb E[e^{M_t}]\leq \mathbb E[e^{2\langle M \rangle_t}]^{1/2}.$$
Indeed, if the right side is infinite then the statement is trivial. If the right-hand side is finite, then by Novikov's condition $e^{2M_t-2\langle M\rangle_t}$ is a martingale. Consequently we have the bounds 
\begin{align*}
\Bbb E[e^{M_t}] = \Bbb E[e^{M_t-\langle M\rangle_t} e^{\langle M\rangle_t}] \leq \Bbb E[ e^{2M_t - 2\langle M\rangle_t}]^{1/2} \Bbb E[e^{2\langle M\rangle_t}]^{1/2} = 1\cdot \Bbb E[e^{2\langle M\rangle_t}]^{1/2}.
	\end{align*}
 \end{proof}
	
	\section{Proof of Proposition \ref{p:mcov}} 
 
 \label{sec:momconv} The goal of this section is to prove Proposition \ref{p:mcov}, the moment convergence of the field $\mc{X}_t^N$ (defined in \eqref{a}) to that of the stochastic heat equation. The key idea is to observe that the moments of $\mc{X}_t^N$ can be written as expectations of certain functionals of a tilted version of a sticky Brownian motion (see Lemma \ref{dn}). We then carefully analyze sticky Brownian motions and the associated tilted measures to obtain a weak convergence result for the tilted measure. This eventually leads to the moment convergence for the field $\mc{X}_t^N$.  
 
 \subsection{Moment computations and moment convergence}
 
 We first identify moments of $\mc{X}_t^N(\phi)$ as expectations of certain functions under the $D_{N^{1/4}}SB_{N^{1/2}\nu}^{(k)}$ measure introduced in Definition \ref{tmeasure}.

	\begin{lem}\label{dn}For all bounded functions $\phi$ on $\R$, $t>s\ge 0$, and $k\in \mathbb N$ one has that 
		\begin{align}\label{e:dn}
			\Ex[\mathscr X_t^N(\phi)^k] &= \mathbf E_{D_{N^{1/4}}SB_{N^{1/2}\nu}^{(k)}}\bigg[\prod_{j=1}^k \phi(X^j_t) \exp\bigg(N^{1/2}\sum_{1\le i<j\le k}V^{ij}_t(\mathbf{X})\bigg)\bigg], 
		\end{align} 
		where the functionals $V^{ij}$ are defined in \eqref{vijai}.
	\end{lem}
	
	\begin{proof} The proof essentially follows by keeping track of all the definitions. Indeed, using the definition of $\mc{X}_t^{N}(\phi)$ from \eqref{a} we may write 
		\begin{align}
			\label{xrep}
			\mathscr X_t^N(\phi) = \Ex_\omega^{(1)} \left[ e^{-\frac12 t \sqrt{N} +N^{-1/4}X_{Nt}} \cdot \phi\left(N^{-1/2} (X_{Nt}-N^{3/4}t)\right)\right],
		\end{align}
		where $\Ex_\omega^{(1)}$ denotes the quenched expectation with respect to a single motion $X$ sampled from the environment $\omega = \{K_{s,t}: - \infty\leq s\leq t<\infty\}.$ Taking the $k^{th}$ power we find that $$\mathscr X_t^N(\phi)^k = \Ex_\omega^{(k)} \left[ e^{-\frac{k}2 t \sqrt{N} +N^{-1/4}(X^1_{Nt} +\cdots+X^k_{Nt})}\prod_{j=1}^k\phi\big(N^{-1/2} (X^j_{Nt}-N^{3/4}t)\big)\right],$$
		where $\Ex_\omega^{(k)}$ denotes the quenched expectation with respect to a $k$-point motion $(X^1,\ldots,X^k)$ sampled from the environment $\omega = \{K_{s,t}: -\infty \leq s\leq t<\infty\}.$
		Taking an {annealed} expectation over the above expression and then using the fact that the averaged law of $(X^1,\ldots,X^k)$ is a $k$-point sticky Brownian motion with characteristic measure $\nu$ we obtain that $$\Ex\big[\mathscr X_t^N(\phi)^k\big] = \mathbf E_{SB^{(k)}_{\nu}} \left[ e^{-\frac{k}2 t \sqrt{N} +N^{-1/4}(X^1_{Nt} +\cdots+X^k_{Nt})}\prod_{j=1}^k\phi\big(N^{-1/2} (X^j_{Nt}-N^{3/4}t)\big)\right].$$ A standard fact about sticky Brownian motion is that $N^{-1/2}(X^1_{Ns},\ldots,X^k_{Ns})$ is distributed as $SB^{(k)}_{N^{1/2}\nu}$ whenever $(X_s^1,\ldots,X_s^k)$ is distributed as $SB^{(k)}_\nu$ (see \cite[Proposition 2.4]{sss}). Therefore the expectation in the last math display can be rewritten as $$\mathbf E_{SB^{(k)}_{N^{1/2}\nu}} \left[ e^{-\frac{k}2 t \sqrt{N} +N^{1/4}(X^1_{t} +\cdots+X^k_{t})}\prod_{j=1}^k\phi(X^j_{t}-N^{1/4}t)\right].$$ Using the definition of $D_{N^{1/4}}SB_{N^{1/2}\nu}^{(k)}$ measure from Definition \ref{tmeasure} and the definition of $V^{ij}$ from  \eqref{vijai}, it is not hard to check that the above expression is precisely equal to the right-hand side of \eqref{e:dn}. This establishes the lemma.
	\end{proof}
	
	Thus, in view of the above lemma, it suffices to study weak convergence of $D_{N^{1/4}}SB_{N^{1/2}\nu}^{(k)}$ measures to extract moment convergence for $\mc{X}_t^N(\phi)$. 
	Our next theorem is the main technical result of this section. It shows that if we consider the path $\mathbf X$ to be distributed according to the $D_{N^{1/4}}SB_{N^{1/2}\nu}^{(k)}$ measure, as $N\to \infty$, 
	\begin{itemize}[leftmargin=18pt]
		\item $\mathbf X$ converges to  $k$ independent Brownian motions.
		\item On the $N^{1/2}$ scale, pairwise intersection times converge to pairwise local times of the corresponding Brownian motions. 
	\end{itemize}

    {Throughout the theorem and its proof, we will diverge from our usual notational conventions by describing the random variables instead of referring to the path measures on the canonical space explicitly. In particular, we will write the associated expectations generically as $\mathbf E$, rather than $\mathbf E_{SB_{\nu}^{(k)}}$ and $\mathbf E_{B^{\otimes k}}$. This will avoid heavy notation, and hopefully does not cause confusion.}
	\begin{thm}\label{converge}
		Fix any $k\in\mathbb N$ and constants $A, b, C, T>0$. Fix any continuous function $f: C([0,T],\mathbb R^k\times \mathbb R^{k(k-1)/2})\to \mathbb R$ such that $|f(p)| \leq Ae^{b\|p\|_\infty}$ for all paths $p: [0,T] \to \mathbb R^k\times \mathbb R^{k(k-1)/2}$. Suppose $\mathbf{X}^N$ is distributed according to the $SB_{N^{1/2}\nu}^{(k)}$ measure from Definition \ref{tmeasure}. Set $V^{ij}=V^{ij}(\mathbf{X}^N)$ as defined in \eqref{vijai}. The following holds. 
  \begin{enumerate}[label=(\alph*),leftmargin=18pt]
  \itemsep\setlength{0.8em}
      \item \label{conva} We have the following convergence
      \begin{equation*}
			\begin{aligned}
				& \lim_{N\to \infty} \mathbf E\big[ f\big(\mathbf X^N\;,\;{4}N^{1/2}\nu([0,1])\big(V^{ij}\big)_{1\leq i<j\leq k}\big)\big] = \mathbf E \big[ f\big(\mathbf{U}\;,\;\big(L_0^{ij}\big)_{1\leq i<j\leq k}\big)\big],
			\end{aligned}
		\end{equation*}
where $\mathbf{U}$ is a standard $k$-dimensional Brownian motion on $[0,T]$, and $L_0^{ij}(t):=L^{U^i-U^j}_0(t)$ is the local time accrued by $U_i-U_j$ at zero by time $t$.
  \item  \label{convb}  {Let $\mathcal G_N=\mathcal G_N(\mathbf X^N)$ be any family of continuous $\mathbf P_{SB^{(k)}_{N^{1/2}\nu}}$-martingales with $\mathcal G_N(0)=0$ satisfying \begin{equation}\label{assn1}\int_s^t d\langle \mathcal G_N\rangle (u) \leq CN^{1/2} \sum_{i<j} (V_t^{ij}(\mathbf X^N)-V_s^{ij}(\mathbf X^N)\big),
		\end{equation}
		for all $0\leq s<t\leq T$. Then $\{\mathcal G_N\}_N$ is a tight family of random variables in $C[0,T]$. Furthermore, any limit point of the triple $\big(\mathcal{G}_N,\; \mathbf{X}^N,\; {4}N^{1/2}\nu([0,1])\big(V^{ij}\big)_{1\leq i<j\leq k}\big)$ is of the form $(\mathcal{G}, \mathbf{U}, \big(L_0^{ij}\big)_{1\leq i<j\leq k}),$ where $\mathbf U, L_0^{ij}$ are as in part (a), and $\mathcal{G}$ is a $C[0,T]$-valued random variable satisfying $\mE[e^{p {\mathcal{G}}(T)-\frac{p}2\langle {\mathcal{G}}\rangle (T)}]<\infty$ for all $p>0$ as well as \begin{align}\label{e:fil2}
	\mE[\exp(\mathcal G(T)-\tfrac12\langle \mathcal G\rangle (T)) \mid \mathcal{F}_T(\mathbf{U})]=1,
		\end{align} where $\mathcal{F}_T(\mathbf{U})$ denotes the $\sigma$-algebra generated by $\mathbf U$. In particular, we have \begin{equation}
			\label{e:converge}
			\begin{aligned}
				& \lim_{N\to \infty} \mathbf E\bigg[e^{\mathcal G_N(T) - \frac12 \langle \mathcal G_N\rangle(T)} f\bigg(\mathbf X^N\;,\; {4}N^{1/2}\nu([0,1])\big(V^{ij}\big)_{1\leq i<j\leq k}\bigg)\bigg]= \mathbf E\big[ f\big(\mathbf{U},\big(L_0^{ij}\big)_{1\leq i<j\leq k}\big)\big].
			\end{aligned}
		\end{equation}}
  
  \item \label{convc} {Suppose $k=2m$ is even. Let us write
		\begin{align}\label{deltadef}
			\Delta_m(s,t):=\{(s_1,s_2,\ldots,s_m)\in [s,t]^m \mid s\le s_1\le s_2\le\cdots\le s_m\le t\}
		\end{align}
		for the set of all $m$ ordered points in $[s,t]$. We define $\mathcal M(\Delta_m(0,T))$ to be the space of finite and non-negative Borel measures on that simplex, equipped with the topology of weak convergence. Consider the following sequence of $\mathcal M(\Delta_m(0,T))$-valued random variables
\begin{align}
\label{gamman}
	\gamma_N(du_1,\ldots,du_m):= N^{m/2}\prod_{j=1}^m \ind_{\{X^{2j-1}_N(u_j)=X^{2j}_N(u_j)\}}\,du_1\ldots \,du_m,
\end{align}
where $\mathbf X^N = (X^1_N,...,X^k_N)$. The random variables $\{\gamma_N\}_{N\ge 1}$ are tight in $\mathcal M(\Delta_m(0,T))$. Moreover, any limit point of the triple $(\mathbf{X}^N,\; {4}N^{1/2}\nu([0,1])\big(V^{ij}\big)_{1\leq i<j\leq k},\; \gamma_N)$ is of the form
\begin{align*}
    \left(\mathbf{U},\; \big(L_0^{ij}\big)_{1\leq i<j\leq k}, \;\left(\frac{\sigma}{{2}}\right)^m\prod_{j=1}^m dL_0^{U^{2j-1}-U^{2j}}(u_j)\right) 
\end{align*}
where $\mathbf U, L_0^{ij}$ are as in part (a), {where $\sigma := \frac{1}{{2}\nu([0,1])}$}, and $dL(t)$ denotes the Lebesgue-Stiltjes measure induced by the increasing function $t\mapsto L(t)$.}
  \end{enumerate} 
	\end{thm}

	By Remark \ref{rk:mn}, taking $\mathcal{G}_N:=N^{1/4}G$ with $G$ defined as in \eqref{def:gm}, we see that equation \eqref{e:converge} in part \ref{convb} of the above theorem deals with the weak convergence of  $D_{N^{1/4}}SB_{N^{1/2}\nu}^{(k)}$ measures. 
 Let us first show how it implies Proposition \ref{p:mcov}.
	
	\begin{proof}[Proof of Proposition \ref{p:mcov}]  Fix any  $\phi\in \mathcal{S}(\R)$. Given $\mathbf{X}$ distributed as $\mathbf P_{SB_{N^{1/2}\nu}^{(k)}}$,
		the martingales $$\mathcal G_N(t) = N^{1/4} \sum_{j=1}^k \int_0^t \bigg[ 1-\frac{1}{m_j(\mathbf X(s))}\bigg]dX^j(s)$$
		clearly satisfy \eqref{assn1}. Thus, Theorem \ref{converge} part \ref{convb} holds for this choice of $\mathcal{G}_N$. In view of Lemma \ref{rnd}, equation \eqref{e:converge} yields that
		$$\lim_{N\to \infty} \mathbf E_{D_{N^{1/4}}SB_{N^{1/2}\nu}^{(k)} }\bigg[ f\bigg(\mathbf X\;,\;{4}N^{1/2}\nu([0,1])\big(V^{ij}\big)_{1\leq i<j\leq k}\bigg)\bigg]= \mathbf E_{B^{\otimes k}} \bigg[ f\bigg(\mathbf{U}\;,\;\big(L_0^{ij}\big)_{1\leq i<j\leq k}\bigg)\bigg],$$
		for all continuous functions $f: C([0,T],\mathbb R^k\times \mathbb R^{k(k-1)/2})\to \mathbb R$ with $|f(p)| \leq Ae^{b\|p\|_\infty}$ for some $A,b>0$. Specializing $f$ to a particular choice given by a product of $\phi$'s multiplied by the exponential of the intersection times, we get  
		\begin{align*}
			\lim_{N\to\infty} \mathbf E_{D_{N^{1/4}}SB_{N^{1/2}\nu}^{(k)}}\bigg[\prod_{j=1}^k \phi(X^j_t) e^{N^{1/2}\sum_{i<j} V_t^{ij}(\mathbf{X})}\bigg] = \mathbf E_{B^{\otimes k}}\bigg[\prod_{j=1}^k \phi(B^j_t) e^{\frac{\sigma}{{2}}\sum_{i<j} L^{B^i-B^j}_0(t)} \bigg].
		\end{align*}
		where $\sigma=\frac1{{2}\nu([0,1])}$. Thanks to the moment formula from Lemma \ref{dn}, we thus have the first equality in \eqref{e:mcov} from the above equation. 
		
		By the Feynman-Kac formula, the stochastic heat equation \eqref{she} admits well-known moment formulas in terms of local times of Brownian bridges (see \cite{BC95} for example). In particular, from \cite{BC95} we have 
		\begin{align}\label{eq:FKMoments}
			\mathbf E\left[\prod_{i=1}^k \mathcal{Z}_t(x_i)\right] = p_t(x_1)\cdots p_t(x_k)\mathbf E_{\vec{x}}\bigg[e^{\frac{\sigma}{{2}} \sum_{i<j} L^{B^i-B^j}_0(t)} \bigg],
		\end{align}
		where $p$ is the standard heat kernel, and the $B^i$ are independent Brownian bridges on $[0,t]$ from $0$ to $x_i$ respectively. {Note that in \cite{BC95}, local time is defined as an integral against Lebesgue measure: $L^X_a(t):= \lim_{\varepsilon \to 0^+}\frac{1}{2 \varepsilon}\int_0^t \mathbbm{1}_{\{a- \varepsilon < X_s < a + \varepsilon\}}ds$, in contrast with our definition of local time  \eqref{eq:localTimeDef} which is as an integral against $d\langle X, X \rangle_s$. Hence, when taking the local time of $B^i - B^j$, our definition of local time differs from theirs by a factor of $2$. This explains why we have $\frac{\sigma}{2}$ in the exponent on the right-hand side of \eqref{eq:FKMoments} instead of just $\sigma$, while this extra factor of $1/2$ does not appear in \cite{BC95}.}
  
        From here one arrives at the same moment formula for  $\mathcal{Z}_t(\phi):=\int_{\R} \mathcal{Z}_t(x)\phi(x)dx$, yielding the second equality in \eqref{e:mcov}. This completes the proof.
	\end{proof}
	
\subsection{Proof of Theorem \ref{converge}} Throughout the proof, we will write $X^{N,i}$ and $U^i$ for the $i^{th}$ coordinate of the $\mathbb R^k$-valued processes $\mathbf X^N$ and $\mathbf U$ respectively. 
\\
\\
		\noindent\textbf{Proof of \ref{conva}.} 
  {We prove part \ref{conva} for any \textit{bounded} continuous function $f$.
		In view of Lemma \ref{traps}, a uniform integrability argument then extends the result to all continuous functions $f$ with exponential growth.  Recall that the stickiness of the sticky Brownian motion has an inverse relationship to the characteristic measure, so intuitively, as we take $N \to \infty$, the Brownian motions will no longer stick together, resulting in $k$ independent Brownian motions in the limit.}

\smallskip

  We now flesh out the technical details of the above-claimed Brownian motion convergence. We assume for convenience that all $(\mathbf{X}^N)_{N\ge 1}$ are coupled onto the same probability space $(\Omega,\mathcal F,\mathbf P)$. For each $N$, we partition $[0,T]$ into the two random sets $A_N, B_N$ as
		\begin{equation}
			\label{e:sets}
			\begin{aligned}
				A_N & := \left\{ s \in [0,T] \mid \#\{X_s^{N,1},X_s^{N,2},\ldots,X_s^{N,k}\}=k\right\}, \\
				B_N & := \left\{ s \in [0,T] \mid \#\{X_s^{N,1},X_s^{N,2},\ldots,X_s^{N,k}\}\le k-1\right\}. 
			\end{aligned}
		\end{equation}
		Only for this proof, we will use $|S|$ to denote the Lebesgue measure of a set $S$. Clearly by the definition of $V_T^{ij}$ (see \eqref{vijai}), we have $|B_N|\leq \sum_{i<j} V_{T}^{ij}(\mathbf X^N)$. Using $a \le e^a$ for $a>0$, we see that
		\begin{align} \label{bncn}
			\mE\left[|B_N| \right] \le  N^{-1/2}\cdot \mE \bigg[ \exp\bigg( N^{1/2} \sum_{i<j} V_T^{ij} (\mathbf X^N)\bigg)\bigg] \leq \Con \cdot N^{-1/2},
		\end{align}
		where $\Con=\Con(k,T)>0$ may be chosen independent of $N$. The last inequality above is a consequence of Lemma \ref{traps} with $\nu \mapsto N^{1/2}\nu$. Let $\mathbf W$ be a Brownian motion independent of $\mathbf X^N$ defined on the same probability space. Let us set  $$\mathbf Y^N_t:= \int_0^t \ind_{\{s\in A_N\}} d\mathbf X^N_s + \int_0^t \ind_{\{s\in B_N\}}d\mathbf W_s.$$ 
		As $X_s^{N,i}\neq X_s^{N,j}$ for all $s\in A_N$, we get that 
		\begin{align*}
			\langle Y^{N,i}, Y^{N,j} \rangle(t) & =\int_0^t\ind_{\{s\in A_N \mid X_s^{N,i}=X_s^{N,j}\}}ds+\ind_{\{i=j\}}\int_0^t \ind_{\{s\in B_N\}}ds  =\ind_{\{i=j\}}\cdot t.
		\end{align*}
		for all $t\in [0,T]$. Thus, by Levy's criterion, we find that the $\mathbf Y^N$ is a $k$-dimensional standard Brownian motion in the combined filtration of $(\mathbf X^N,\mathbf W)$. We claim that
		\begin{align}\label{e:doob1}
			\mathbf P\bigg[ \sup_{t\leq T} \big\Vert\mathbf Y_t^N-\mathbf X_t^N\big\Vert >N^{-1/6} \bigg]  \leq 2k \cdot \Con N^{-1/6}. 
		\end{align}
		where the $\Con$ is the same as in \eqref{bncn}. An application of Doob's submartingale inequality followed by It\^{o}'s isometry  yields 
		\begin{align*}
			\mathbf P\bigg[ \sup_{t\leq T} \big\Vert\mathbf Y_t^N-\mathbf X_t^N\big\Vert >N^{-1/6} \bigg] & \leq  N^{1/3} \cdot \mE \left[ \big\Vert\mathbf Y_T^N-\mathbf X_T^N\big\Vert^2\right] \\ &  =  N^{1/3} \cdot\mE\bigg[ \bigg\Vert \int_0^T \ind_{\{s\in B_N\}}d\big(\mathbf W_s-\mathbf X^N_s\big)\bigg\Vert^2 \bigg] \\  &= 2k N^{1/3} \cdot \mE \bigg[ \int_0^T \ind_{\{s\in B_N\}}ds\bigg]  = 2k N^{1/3} \cdot \mE[|B_N|]. 
		\end{align*}
		Inserting the bound from \eqref{bncn} leads to \eqref{e:doob1}. As $\mathbf{Y}^N$ is a standard $k$-dimensional Brownian motion, \eqref{e:doob1} implies $\mathbf X^N$ converges weakly to a standard $k$-dimensional Brownian motion. 
		
		\medskip
		
		Let us now settle the weak convergence of pairwise intersection times. 
  Observe that by Lemma \ref{lte}, ${4}N^{1/2}\nu([0,1])V^{ij}(\mathbf X^N)$ is a tight sequence in $C[0,T]$ for each $i<j$. 
  Let us consider any limit point $(\mathbf U, (K^{ij})_{1\le i<j\le k})$ of $$(\mathbf{X}^N \ ,\  ({4}N^{1/2}\nu([0,1])V^{ij}(\mathbf X^N))_{1\le i<j\le k}).$$ By the property \eqref{eq:two-point-sbm} of sticky Brownian motion, for each $i<j$ we have that
		$$|X^{N,i}(t)-X^{N,j}(t)|-{4}N^{1/2}\nu([0,1])V^{ij}(\mathbf X^N)$$
		is a martingale for each $N$. Thanks to the estimates in Lemma \ref{traps}, the above expression is uniformly integrable. Thus in the limit point, the process $|U^i-U^j|-K^{ij}$ is a martingale as well. However, we have already established that $\mathbf{U}$ is a standard $k$-dimensional Brownian motion. Thus, by Tanaka's formula, $K^{ij}(\cdot)$ has to be $L_0^{U^{i}-U^j}(\cdot)$.
		
		Thus, summarizing, we have identified the limit point as $(\mathbf{U}, (L_0^{ij})_{1\le i<j\le k})$ where $\mathbf{U}$ is a standard $k$-dimensional Brownian motion and $L_0^{ij}(t):=L_0^{U^i-U^j}(t)$. This verifies part \ref{conva}.

		\medskip
		
		\noindent\textbf{Proof of \ref{convb}.} Let us take any $\mathcal G_N$ satisfying the assumption of the theorem. We first prove that $\mathcal G_N$ is tight in $C[0,T]$. Indeed, by Burkholder-Davis-Gundy inequality, we have that $$\mathbf E \big[ |\mathcal G_N(t)-\mathcal G_N(s)|^p \big] \leq \mathbf E \bigg[ \bigg(\int_s^t d\langle \mathcal G_N\rangle (u)\bigg)^{\frac{p}2}\bigg] \leq C\cdot\mathbf E \bigg[ \bigg(N^{1/2}\sum_{i<j}  \big(V_t^{ij}(\mathbf X^N)-V_s^{ij}(\mathbf X^N)\big)\bigg)^{\frac{p}2}\bigg].$$ 
		The second inequality above is due to the assumption \eqref{assn1}. Lemma \ref{lte} implies that the term on the right-hand side of the above equation is bounded by $\Con|t-s|^{p/4}$ for some constant $\Con>0$ independent of $N$. This verifies the tightness of $\mathcal G_N$.
		
		\medskip
		
		Note that \eqref{e:converge} would be immediate from \eqref{e:fil2}, because the latter would imply that any subsequence of the left side of \eqref{e:converge} has a further subsequence which converges to the right side of \eqref{e:converge}. Let us therefore prove \eqref{e:fil2}. Consider any joint limit point $(\mathcal G, \mathbf U, (K^{ij})_{1\le i<j\le k})$  of $$(\mathcal{G}_N \ , \ \mathbf X^N \ ,\  ({4}N^{1/2}\nu([0,1])V^{ij}(\mathbf X^N))_{1\le i<j\le k}).$$ 
		By part \ref{conva}, $\mathbf U$ is a standard $k$-dimensional Brownian motion and $K^{ij}=L_0^{U^i-U^j}$. By our moment estimates, we have uniform integrability of both prelimiting martingales $\mathcal{G}_N$ and $\mathbf X^N$. Thus in the limit, $\mathcal G$ and $\mathbf U$ are martingales in their \textit{joint} filtrations. By Proposition \ref{exp} and the assumption \eqref{assn1}, we have boundedness of $p^{th}$ moments (hence uniform integrability) of $\{e^{\mathcal G_N(T) - \frac12 \langle \mathcal G_N\rangle_T}\}_{N\ge 1}.$ Thus, in the limit, $\exp(\mathcal G(t)-\frac12\langle \mathcal G \rangle(t))$ is a martingale as well.

		\medskip
		
		To complete the proof of \eqref{e:fil2}, the main idea is to use \eqref{assn1} to prove that the martingales $\mathcal G_N$ {are} asymptotically decoupled from the sticky Brownian motions. To be precise, we shall prove that for each $i$ and for all $0\le t\le T$
		\begin{align}\label{e:covv}
			\lim_{N\to \infty}\mathbf E \bigg[ \int_0^t d|\langle X^{N,i}, \mathcal G_N\rangle(s)|\bigg]=0.
		\end{align}
		Let us assume \eqref{e:covv} for the moment and complete the proof of \eqref{e:fil2}. Note that in the prelimit $X^{N,i}\cdot \mathcal G_N - \langle X^{N,i},\mathcal G_N\rangle$ are martingales and \eqref{e:covv} shows that the covariations vanish in probability. Therefore, in the limit, we can conclude that $\mathcal{G}\cdot U^i$ is a martingale so that $\langle U^i, \mathcal{G} \rangle = 0.$ 

  \medskip
  
		Take any bounded $\mathcal F_T(\mathbf{U})$-measurable functional $H:C([0,T],\mathbb R^k)\to \mathbb R$.  Note that $t\mapsto \mE[H(\mathbf{U})|\mathcal F_t(\mathbf U)]$ is a martingale in the filtration of $\mathbf{U}$. As $\mathbf{U}$ is a standard $k$-dimensional Brownian motion, by the martingale representation theorem we have
		$$H(\mathbf{U})=\mE [H(\mathbf{U})]+\sum_{i=1}^k\int_0^T h_s^i dU_s^i$$
		for some adapted $\mathbb R$-valued processes $h^1,\ldots,h^k$. For the stochastic exponential we have $$\exp\left(\mathcal{G}(T) - \frac12 \langle \mathcal{G}\rangle (T)\right) = 1+\int_0^T \exp\left(\mathcal{G}(s) - \frac12 \langle \mathcal{G}\rangle (s)\right) d\mathcal{G}(s).$$ Using this we find that 
		\begin{align*} \mE\left[\big(e^{\mathcal{G}(T) - \frac12 \langle \mathcal{G}\rangle (T)} -1\big)\big(H(\mathbf{U})-\mE[H(\mathbf{U})]\big)\right] &= \sum_{i=1}^k\mE\bigg[\bigg(\int_0^T e^{\mathcal{G}(s) - \frac12 \langle \mathcal{G}\rangle(s)}d\mathcal{G}(s)\bigg)\bigg(\int_0^T h_s^i  dU^i_s\bigg)\bigg] \\&= \sum_{i=1}^k \mE\bigg[\int_0^T h_s^i e^{\mathcal{G}(s) - \frac12 \langle \mathcal{G}\rangle(s)} d\langle \mathcal{G},U^i\rangle(s)\bigg] = 0,
		\end{align*}
		where the last equality is due to the fact $\langle U^i,\mathcal{G}\rangle \equiv 0$ almost surely. This proves that for all bounded measurable $H$ we have $\mE[e^{\mathcal{G}(T) - \frac12 \langle \mathcal{G}\rangle(T)}H(\mathbf{U})]=\mE[H(\mathbf{U})]$, establishing \eqref{e:fil2} modulo \eqref{e:covv}.
		
		\medskip

		Let us now explain why \eqref{e:covv} holds. Note that by assumption \eqref{assn1} 
		$$\mathcal G_N(t) = \int_0^t \ind_{\{ s\in B_N\}} d\mathcal G_N(s),$$ where $B_N$ is defined in \eqref{e:sets}. By the Kunita-Watanabe inequality {(see Proposition 2.14 in Chapter 3 of \cite{karatzas})}, we have that
		\begin{align}\notag\mathbf E \bigg[ \int_0^t d|\langle X^{N,i}, \mathcal G_N\rangle(s)|\bigg] &= \mathbf E \bigg[ \int_0^t \ind_{\{ s\in B_N\}} d|\langle X^{N,i},\mathcal G_N\rangle(s)|\bigg] \\ &\leq \mathbf E \bigg[ \bigg( \int_0^t \ind_{\{ s\in B_N\}} ds\bigg)^{\frac12} \sqrt{\langle \mathcal{G}_N\rangle(t)}\bigg]. \label{w1}
		\end{align}
		By definition, $\int_0^t \ind_{\{ s\in B_N\}} ds\le \sum_{i<j} V_t^{ij}$. Using assumption \eqref{assn1} we get that
		\begin{align*}
			\mbox{r.h.s.~of \eqref{w1}} \le N^{1/4}\cdot \mathbf E \bigg[ \sum_{i<j}V^{ij}_t\bigg] \leq N^{-1/4} \cdot \mathbf E \bigg[ \exp\bigg(N^{1/2}\sum_{i<j}V^{ij}_t \bigg)\bigg]\leq \Con \cdot N^{-1/4},
		\end{align*}
		where we used $a\leq e^a$ and then applied Lemma \ref{traps} in the last two bounds. This verifies \eqref{e:covv}, completing the proof of the theorem.

  \medskip

  \noindent\textbf{Proof of \ref{convc}.} Note that
\begin{align*}
	\gamma_N(\Delta_m(0,T)) \le \prod_{j=1}^m\left[N^{\frac12}\int_0^T \ind_{\{X^{2j-1}_{u}=X^{2j}_{u}\}} du\right].
\end{align*}
From Lemma \ref{traps} we know the exponential moments for each of the terms in the product are uniformly bounded under $SB_{N^{1/2}\nu}^{(k)}$ measure. Thus for all $p\ge 1$,
\begin{align}\label{gmass1}
	\sup_{N\ge 1} \mathbf E_{SB^{(k)}_{N^{1/2}\nu}} [\gamma_N(\Delta_m(0,T))^p]<\infty.
\end{align}
Hence the laws of $\{\gamma_N\}_{N\ge 1}$ are tight, because the total mass of $\gamma_N$ is a tight family of random variables and because $\Delta_m(0,T)$ as defined in \eqref{deltadef} is a compact space. Let $\left(\mathbf{U}, \big(K^{ij}\big)_{1\leq i<j\leq k}, \gamma\right)$ be any limit point of the sequence $(\mathbf{X}^N, {4}N^{1/2}\nu([0,1])\big(V^{ij}\big)_{1\leq i<j\leq k}, \gamma_N)$. By part \ref{conva}, we know $\mathbf{U}$ is a standard $k$-dimensional Brownian motion and $K^{ij}=L_0^{ij}$. Note that the joint cumulative distribution function of the measure $\gamma$ is necessarily given by $\left(\frac{\sigma}{{2}}\right)^m$ times a product of the $K^{ij}$, {where $\sigma = \frac{1}{{2}\nu([0,1])}$}. Since we know that $K^{ij}=L_0^{ij}$, this immediately implies that $\gamma=\left(\frac{\sigma}{{2}}\right)^m\prod_{j=1}^m dL_0^{U^{2j-1}-U^{2j}}(u_j)$. This establishes part \ref{convc}.

	\section{Martingales associated to the sticky kernels}\label{sec:mp}
	
	In this section, we begin our analysis of the field $\mathscr X_t^N$ by identifying some useful martingale observables which will later be used to identify the limit in Section \ref{sec:smp}. 
 
 \begin{defn}For a finite measure $\nu = \lambda + \sum_{k\in \mathbb N} p_k \delta_{x_k}$ on the real number line with $\lambda$ atomless, we define its ``square" 
 \begin{align}
     \label{def:sq}
     \nu^{Sq} := \sum_k p_k^2 \delta_{x_k}.
 \end{align}
 \end{defn}
 For a continuous function $\phi$ we also denote by $(\nu,\phi):=\int_\mathbb R \phi \;d\nu$ the natural (spatial) pairing. When $\phi$ is smooth, we will use $\phi',\phi''$ to denote the first and second (spatial) derivatives of $\phi$ respectively. We then have the following lemma.
	
	\begin{lem}\label{mart1}Let $K_{0,t}$ denote the Howitt-Warren flow. Then for all $\phi \in C_c^\infty(\mathbb R)$, the process $$S_t(\phi):= (K_{0,t}, \phi) - \frac12\int_0^t (K_{0,s},\phi'')ds$$ is a continuous martingale in the filtration $\mathcal F_{t} := \sigma(\{K_{a,b}(x,A): 0\leq a<b \leq t, x\in \mathbb R, A \mbox{ Borel}\}).$ Furthermore its quadratic variation is given by $$\langle S(\phi)\rangle_t = \int_0^t (K_{0,s}^{Sq} , (\phi')^2 )ds.$$
		
	\end{lem}
From \cite[Theorem 2.8]{sss}, we know for each fixed $t>0$, $K_{0,t}$ is atomic almost surely. Thus $K_{0,s}^{Sq}$ is nonzero for almost every $s>0$.	
	\begin{proof}

		{For the moment being, we consider a fixed realization of the kernels $\{K_{s,t}:0\leq s\leq t\}.$ By {\cite[Section 2.6]{lejan}}, we can sample paths $B^1, B^2,B^3,\ldots$ starting from 0 in such a way that: 
  \begin{itemize}
  \item given $\{K_{s,t}:0\leq s\leq t\},$ the paths $\{B^i\}_{i\ge 1}$ are all independent.
  \item given $\{K_{s,t}:0\leq s\leq t\},$ the law of $B^i_t$ given $(B_u^i)_{0\leq u\leq s}$ is $K_{s,t}(B_s,\cdot)$ for all $0\leq s\leq t.$
  \end{itemize}
 It is clear that the joint quenched law of these paths $\{B^i\}_{i\ge 1}$ is a deterministic function of the collection $\{K_{s,t}:0\leq s\leq t\}$, because the two bullet points automatically determine the multi-time marginal laws of any finite subcollection $\{B^i\}_{i=1}^r$ in terms of the kernels $K_{s,t}$. It is also clear that the \textit{annealed} law of $(B^1,\ldots,B^r)$ is just $r$-point sticky Brownian motion (see the discussion after Definition \ref{sfok}). We define $\mu_t^r := \frac1r \sum_{i=1}^r\delta_{B_t^i}.$ By the law of large numbers, as $r\to \infty$ we have $$\mu_t^r \to K_{0,t} \;\;\;\;\;\;\text{    and    } \;\;\;\;\;\; \int_0^t \mu_s^r ds\to \int_0^t K_{0,s}ds$$ almost surely in the topology of weak convergence of Borel measures, for each $t>0$. Notice that $$|(\mu^r_t,\phi)-(\mu^r_s,\phi)| \leq \frac1r \sum_{i=1}^r |\phi(B_t^i)- \phi(B^i_s)| \leq \frac{\|\phi'\|_{L^\infty(\mathbb R)}}{r} \sum_{i=1}^r |B_t^i-B_s^i|$$ for all $s,t\ge 0$, all $r\in\mathbb N$, and all $\phi\in C_c^\infty(\mathbb R)$. By Minkowski's inequality and the fact that the $B^i$ are individually Brownian motions in the annealed sense, we obtain \begin{equation}\label{lpmu}\mathbb E[|(\mu^r_t,\phi)-(\mu^r_s,\phi)|^p]^{1/p} \leq \frac{\|\phi'\|_{L^\infty(\mathbb R)}}{r}\sum_{i=1}^r \mathbb E[|B_t^i-B_s^i|^p]^{1/p} = C_p \|\phi'\|_{L^\infty} |t-s|^{1/2}, \end{equation} for any $p\ge 1$. 
 Define 
		\begin{align}\label{def:mtn}
			\til{S}_{t}^r(\phi):= (\mu_t^r, \phi) - \frac12\int_0^t (\mu_s^r,\phi'')ds= \frac1r \sum_{i=1}^r\int_0^t\phi'(B_s^i)dB_s^i.
		\end{align}
		where the second equality above is due to It\^{o}'s formula.  Hence, $\til{S}_{t}^r(\phi)$ is a martingale. Clearly $|(\mu^r_t,\phi)|\leq \|\phi\|_{L^\infty(\mathbb R)}$. Combining this with \eqref{lpmu} and the first equality in \eqref{def:mtn}, one has $$\mathbb E[|\til{S}_{t}^r(\phi)-\til{S}_{s}^r(\phi)|^p]^{1/p} \leq C_p \|\phi'\|_{L^\infty}|t-s|^{1/2}+\frac12 \|\phi''\|_{L^\infty}|t-s|,$$ uniformly over $s,t\ge 0$. On any compact interval $[0,T]$ the right side can be bounded above by $\Con|t-s|^{1/2}$ for some constant $\Con=\Con(p,T,\phi)$ independent of $s,t\in [0,T]$. Since $S_t(\phi) = \lim_{r \to \infty} \til{S}_{t}^r(\phi)$ almost surely, this $L^p$ bound implies that the limit $S_t(\phi)$ is a continuous $L^p$-martingale for each $\phi \in C_c^\infty(\mathbb R)$. 
  The quadratic variations of $\til{S}^r(\phi)$ can be computed explicitly as
		\begin{align*}
			\langle \til{S}^r(\phi)\rangle_t =\frac{1}{r^2}\sum_{i,j}\int_0^t \phi'(B_s^i)\phi'(B_s^j)d\langle B^i,B^j \rangle_s
   =\frac{1}{r^2}\sum_{i,j}\int_0^t \phi'(B_s^i)^2\mathbf{1}_{\{B_s^i = B_s^j\}}ds = \int_0^t ((\mu_s^r)^{Sq},(\phi')^2)ds.
		\end{align*}
		We consider what happens to the quadratic variation as we take $r\to \infty$. By \cite{sss}, we have that for each $t > 0$, $K_{0,t}$ is a.s.~atomic. 
        Using this, it is easy to show that for each $s>0$, the quantity $((\mu_s^r)^{Sq},(\phi')^2)$ will converge almost surely as $r\to\infty$ to $(K_{0,s}^{Sq} , (\phi')^2 ).$ Since $\mu^r_t$ is a probability measure, we have $((\mu_s^r)^{Sq},(\phi')^2)\leq (\mu_s^r,(\phi')^2)\leq \|\phi'\|_{L^\infty}^2$ deterministically, and likewise $(K_{0,s}^{Sq} , (\phi')^2 )\leq (K_{0,s} , (\phi')^2 )\leq \|\phi'\|_{L^\infty}^2$. Therefore the dominated convergence theorem applied on the product space $\mathbb P\otimes ds$ gives $$\lim_{r\to\infty} \mathbb E \bigg[\int_0^t \big| ((\mu_s^r)^{Sq},(\phi')^2)-(K_{0,s}^{Sq} , (\phi')^2 )\big|ds\bigg]=0.$$ Here as always, $\mathbb E$ is annealed expectation. In other words, the expression for $\langle \til{S}^r(\phi)\rangle_t$ will converge in $L^1(\mathbb P)$ as $r\to \infty$ to the quantity
		$\int_0^t (K_{0,s}^{Sq} , (\phi')^2 )ds.$
Putting all of this together, we obtain that $S_t(\phi) = \lim_{r \to \infty} \til{S}_{t}^r(\phi)$ is a martingale with quadratic variation
		$$ \langle S(\phi)\rangle_t = \lim_{r\to \infty}  \langle \til{S}^r(\phi)\rangle_t = \int_0^t (K_{0,s}^{Sq} , (\phi')^2 )ds,$$ where the limit is interpreted in $L^1(\mathbb P).$ This completes the proof.}
	\end{proof}
Given the above lemma, our next goal is to derive a (local) martingale that is related to the rescaled field $\mc{X}^N$ from \eqref{a} instead of $K_{0,t}$. To motivate the choice of our next martingale, we first perform some martingale calculations on a formal level. 
The process $S$ from Lemma \ref{mart1} resembles an \textit{orthomartingale} in the sense that the cross-variation vanishes whenever $\phi$ and $\psi$ have disjoint support, and the above lemma states that $K_{0,t}$ solves the SPDE $$d K = \partial_x^2 K dt + dS.$$
	Note that if $u$ solves the forced heat equation $\partial_t u = \frac12 \partial_x^2 u + F$ and if $a,b,c \in \mathbb R$ then $e^{ax+bt} u(t,x+ct)$ solves the equation $\partial_t v = Lv + \bar F$ where $$L = \frac12 \big(\partial_x - (a-c)I\big)^2 + \big(b-\frac12 c^2\big) I,$$ $$\bar F(t,x) = e^{ax+bt} F(t,x+ct).$$ This makes sense even when $F$ is a Schwartz distribution. Note that when $a=c$ and $b = \frac12 c^2$ then $L = \frac12 \partial_x^2$ is unchanged.

	We now formally go from $K_{0, t}$ to $\mathscr X^N_t$ defined in \eqref{a} in two steps. We first apply the above argument for $u=K$ and $F=dS$, with $a = c = N^{-1/4}$ and $b = \frac12{N^{-1/2}}$. This gives us that the measure-valued process
		\begin{align*}
			\til{K}^N(t,x)&:= e^{ax + bt}K_{0, t}(x + ct) =e^{\frac{t}{2}N^{-1/2} + xN^{-1/4}}K_{0,t} (x + N^{-1/4}t)
		\end{align*}
		solves the SPDE $$\partial_t \til{K}^N = \partial_x^2 \tilde{K}^N dt + d{S}^N$$ where ${S}^N$ is the orthomartingale measure defined by $$d{S}^N(t,x) =e^{\frac{t}{2}N^{-1/2} + xN^{-1/4}} \cdot dS(t, x + N^{-1/4}t),$$ to be interpreted by integration against test functions. Next, we apply a diffusive scaling to $\til{K}^N$. Let us set
		\begin{align*}
			\mathscr X^N_t(x)&:= N^{1/2}\til{K}_N(Nt, N^{1/2}x)= N^{1/2} e^{\frac{t}{2}N^{1/2} + xN^{1/4}}K_{0,Nt} (tN^{3/4} + xN^{1/2}).
		\end{align*}
  Note that {$\partial_t\mathscr X^N_t = N^{3/2} \partial_t \til{K}^N(Nt, N^{1/2}x)$ and $\partial_x^2 \mathscr{X}^N dt= N^{3/2}\partial_x^2 \til{K}^N(Nt, N^{1/2}x) dt$. Thus, $\mathscr{X}_t^N$ solves the SPDE} $$\partial_t \mathscr X^N = \partial_x^2 \mathscr X^N dt + dM^N$$ where $M^N$ is the orthomartingale measure defined by \begin{align*}
			dM^N(t,x) &=N^{3/2} d{S}^N(Nt, N^{1/2}x)=N^{3/2} e^{\frac{t}{2}N^{1/2} + xN^{1/4}} \cdot dS(Nt, tN^{3/4} +xN^{1/2}).
		\end{align*}

 The above formal calculations suggest the following lemma.

	\begin{lem}\label{qv}
		With $\mathscr X^N$ as defined in \eqref{a}, for all $\phi \in C_c^\infty(\mathbb R)$ the processes \begin{equation}\label{mart_ob}M_t^N(\phi):= \mathscr X_t^N(\phi)- \phi(0)- \frac12 \int_0^t \mathscr X_s^N(\phi'')ds
		\end{equation}
		are continuous martingales with \begin{align*}\langle M^N(\phi)\rangle_t = \int_0^t \big( (\mathscr X_s^N)^{Sq}, (\phi' + N^{1/4}\phi)^2 \big) ds
		\end{align*}
	where $(\mathscr X_s^N)^{Sq}$ is defined via \eqref{def:sq}.	
	\end{lem}
	We remark that the $M^N(\phi)$ are actually tight in $C[0,T]$ as $N\to\infty$. This will follow from the Burkholder-Davis-Gundy inequality as well as moment bounds on the quadratic variation that are proved in Proposition \ref{tight1}. 
    Tightness will be proved in Proposition \ref{mcts}, using moment formulas and bounds that we will derive in Section \ref{sec:qmf}. This will be crucial in solving the martingale problem for \eqref{she}.
	
	\begin{proof}
		The proof follows in a similar manner to the proof of Lemma \ref{mart1}. {Note that $\phi'$ from Lemma \ref{mart1} has changed to $\phi'+N^{1/4}\phi$ due to the transformation of space-time $(t,x)\mapsto (Nt, N^{3/4}t+N^{1/2}x)$ which is used to obtain $\mathscr X_t^N$ from $K_{0,t}$}. We omit the details.
	\end{proof}

 Lemma \ref{qv} can be extended to include time-dependent test functions as well. 
 
	\begin{lem}\label{qvt} Let $\varphi \in \mathcal{S}(\mathbb R^2)$. Consider $M_t^N(\phi)$ as defined in \eqref{mart_ob}. The process
	     \begin{align}\label{tmart}
	         R_t^N(\varphi):= M^N_t\big(\varphi(t,\cdot)\big)-\int_0^t M^N_s\big(\partial_s\varphi(s,\cdot)\big)ds
	     \end{align} is a continuous martingale, and its quadratic variation up to time $t$ is given by 
      \begin{align}\label{tqmart}
          \langle R^N(\varphi)\rangle_t=\int_0^t \big( (\mathscr X_s^N)^{Sq},{N^{1/2}}(\varphi(s,\cdot)+N^{-1/4}\partial_x\varphi (s,\cdot))^2\big)ds.
      \end{align}
	\end{lem}
 Intuitively, one should think of $R_t^N(\varphi)$ as being equal to $\int_0^t (\varphi(s,\cdot),dM^N_s)_{L^2(\mathbb R)},$ which yields \eqref{tmart} by a formal integration-by-parts.
	\begin{proof} Let $\mathcal{F}_t:=\sigma(\{K_{Na,Nb}(x,A): 0\leq a<b \leq t, x\in \mathbb R, A \mbox{ Borel}\})$. Fix $0\le u\le t$. By Lemma \ref{qv} we have
 \begin{align}\label{qvt1}
     \Ex[R_t^N(\varphi)\mid \mathcal{F}_u]=M_u^N\big(\varphi(t,\cdot)\big)-\int_0^u M^N_s\big(\partial_s\varphi(s,\cdot)\big)ds-\int_u^t M^N_u\big(\partial_s\varphi(s,\cdot)\big)ds.
 \end{align}
 Note that
 \begin{align*}
     \int_u^t M^N_u\big(\partial_s\varphi(s,\cdot)\big)ds= M_u^N\big(\varphi(t,\cdot)\big)-M_u^N\big(\varphi(u,\cdot)\big).
 \end{align*}
Inserting this identity back in \eqref{qvt1}, we get $\Ex[R_t^N(\varphi)\mid \mathcal{F}_u]=R_u^N(\varphi)$. This verifies that $R_t^N(\varphi)$ is a martingale with respect to ~$\mathcal{F}_t$.	 The claimed quadratic variation for $R^N_t(\varphi)$ in \eqref{tqmart} can be verified similarly. 
	\end{proof}

	\section{Analysis of the quadratic martingale field} \label{sec:qmf}

	Recall the martingale $M_t^N$ from \eqref{mart_ob}. As explained in the introduction, in order to identify limit points of the field $M_t^N(\phi)$, one must study the quadratic variation of this martingale. In this section, we identify the leading order contribution of $\langle M^N(\phi) \rangle$ as a functional which we call the \textit{quadratic martingale field}---introduced below.

	\begin{defn}  Recall the field $\mc{X}_t^N$ from \eqref{a}, as well as the squaring operation from Definition \ref{def:sq}. For a bounded test function $\phi$ on $\R$, we define the quadratic martingale field (QMF) as 
		\begin{align}\label{QVfield} 
			Q_t^N(\phi) & :=N^{1/2} \int_0^t \big( (\mathscr X_s^N)^{Sq}, \phi\big)ds \\ & = N^{1/2}\int_0^t \int_{\mathbb R} e^{-s\sqrt{N} +2uN^{-1/4}} \phi(N^{-1/2}(u-N^{3/4}s)) K_{0,Ns}^{Sq}(du)ds,\notag
		\end{align}
	\end{defn}
	From Lemma \ref{qv}, in view of the above definition, one has that 
	\begin{equation*}
		\langle M^N(\phi)\rangle_t = Q_t^N(\phi^2) +\mathcal E_t^N(\phi),
	\end{equation*}
	where $\mathcal E_t^N(\phi)$ is the ``error term" given by
	\begin{align}\label{etn}
		\mathcal E_t^N(\phi) & :=  \int_0^t \big( (\mathscr X_s^N)^{Sq}, 2N^{1/4} \phi\phi' +  (\phi')^2\big)ds = Q_t^N\big(N^{-1/2}(\phi')^2 + 2N^{-1/4} \phi\phi'\big).
	\end{align}
	
	\medskip
	
	The main goal of this section is to show that only $Q_t^N$ contributes to the limit (in a very specific way, see Proposition \ref{4.1}) and that $\mathcal E_t^N$ vanishes in the limit (Proposition \ref{4.2}). The starting point of our analysis is a probabilistic interpretation of $Q_t^N(\phi)$. Indeed, similar to the representation \eqref{xrep} for $\mc{X}_t^N(\phi)$, the field $Q_t^N(\phi)$  also admits a formula in terms of a quenched expectation of $2$-point motion sampled from the environment $\omega:=\{K_{s,t} : -\infty<s<t<\infty\}$. Given a $2$-point motion sampled from the environment $\omega:=\{K_{s,t} : -\infty<s<t<\infty\}$, we may view $Q_t^N(\phi)$ as
	\begin{align}
		\nonumber Q_t^N(\phi) & = N^{1/2}\int_0^t \Ex^{(2)}_\omega \bigg[ e^{-s\sqrt{N} + 2N^{-1/4}X_{Ns}} \phi(N^{-1/2}(X_{Ns}-N^{3/4}s))\ind_{\{X_{Ns}=Y_{Ns}\}}\bigg] ds \\ & = N^{1/2}\int_0^t \Ex^{(2)}_\omega \bigg[ e^{-s\sqrt{N} + N^{-1/4}(X_{Ns}+Y_{Ns})} \phi(N^{-1/2}(X_{Ns}-N^{3/4}s))\ind_{\{X_{Ns}=Y_{Ns}\}}\bigg] ds \label{qrep}
	\end{align}
    where the second equality follows by replacing $2N^{-1/4}X_{Ns}$ in the exponential with  $N^{-1/4}(X_{Ns}+Y_{Ns})$ due to the presence of the indicator $\ind_{\{X_{Ns} = Y_{Ns}\}}.$ The representation \eqref{qrep} will be used as a starting point in the next section, Section \ref{sec:feqmf}, in extracting formulas and estimates for the QMF.
	
	\subsection{Formulas and limits for the Quadratic Martingale Field} \label{sec:feqmf}
	
	Recall the field  $\mc{X}_t^N$ from \eqref{a}. The goal of this section is to show that the limit of $Q_t^N$ can be related to that of $\mc{X}_t^N$ in a precise way (see Proposition \ref{4.1}). To do this, we first extract moment formulas of certain observables of $(\mc{X}_t^N, Q_t^N)$ in terms of sticky Brownian motion in the same spirit as Lemma \ref{dn}. 
	
	\begin{lem}[Moment formulas]\label{l:Qmom} Fix any bounded functions $\psi,\phi$ on $\R$ and $N\ge 1$. Suppose that $(X^1,\ldots, X^{2k})$ denotes a $2k$-point sticky Brownian motion with characteristic measure $N^{1/2}\nu$. Recall $\Delta_m(s,t)$ from \eqref{deltadef} and $\sigma=\frac1{{2}\nu([0,1])}$. We have the following moment formulas.
		\begin{enumerate}[label=(\alph*),leftmargin=15pt]
			\item \label{l:QXmom} For all $t>0$, we have 
			\begin{equation}
				\label{e:QXmom}
				\begin{aligned}
					& \Ex\left[\bigg(Q_t^N(\psi)-\sigma \int_0^t \mathscr X_s^N(\phi)^2 ds\bigg)^{k}\right] \\ & \hspace{1cm}= k! \int_{\Delta_k(0,t)}  \mathbf E_{SB_{N^{1/2}\nu}^{(2k)}}\bigg[ \prod_{i=1}^k e^{-s_i\sqrt{N} +N^{1/4}(X_{s_i}^{2i-1}+X_{s_i}^{2i})} \Upsilon_{s_i}^N(X^{2i-1},X^{2i})\bigg]\prod_{i=1}^k ds_i,
				\end{aligned}
			\end{equation}
			where
			\begin{align*}
				\Upsilon_u^N(X,Y) := N^{1/2}\psi\big(X_u-N^{1/4}u\big)\ind_{\{X_u=Y_u\}}-\sigma \phi\big(X_u-N^{1/4}u\big)\phi\big(Y_u-N^{1/4}u\big).
			\end{align*}
			\item \label{l:Qincmom} For all $0\le s<t$ we have the following moment formula for the increment of the QMF
			\begin{equation}
				\label{e:Qincmom}
				\begin{aligned}
					& \Ex\left[\bigg(Q_t^N(\psi)-Q_s^N(\psi)\bigg)^{k}\right] \\ 
					&  = N^{\frac{k}2}k! \int_{\Delta_k(s,t)}  \mathbf E_{SB_{N^{1/2}\nu}^{(2k)}}\bigg[ \prod_{i=1}^k e^{-s_i\sqrt{N} +N^{\frac14}(X_{s_i}^{2i-1}+X_{s_i}^{2i})} \psi(X_s^{2i}-N^{1/4}s_i)\ind_{\{X_{s_i}^{2i-1}=X_{s_i}^{2i}\}}\bigg]\prod_{i=1}^k ds_i.
				\end{aligned}
			\end{equation}
		\end{enumerate}
	\end{lem}
	
	\begin{proof} Recall the probabilistic interpretation of $\mc{X}_t^N(\phi)$ and $Q_t^N(\psi)$ from \eqref{xrep} and \eqref{qrep} respectively. Combining them we get that 
		\begin{align}
			\label{e:QXdif}
			Q_t^N(\psi) - \sigma \int_0^t \mathscr X_s^N(\phi)^2 ds = \int_0^t \Ex^{(2)}_\omega \left[ e^{-s\sqrt{N} + N^{-1/4}(X_{Ns}^1+X_{Ns}^2)}\til{\Upsilon}_{s}^N(X^1,X^2)\right] ds
		\end{align}
		where $\Ex_{\omega}^{(2)}$ denotes the quenched expectation of a $2$-point motion sampled from the environment $\omega:=\{K_{s,t} : -\infty<s<t<\infty\}$ and
		\begin{align*}
	\til{\Upsilon}_{s}^N(X^1,X^2):=N^{1/2}\psi\left(N^{-1/2}(X_{Ns}^1-N^{3/4}s)\right)\ind_{\{X_{Ns}^1=X_{Ns}^2\}}-\sigma \prod_{i=1}^2\phi\left(N^{-1/2}(X_{Ns}^i-N^{3/4}s)\right).
		\end{align*}
		Taking the $k$-th power of \eqref{e:QXdif} we see that
  \begin{equation}
      \label{ee1}
      \begin{aligned}
			& \left(Q_t^N(\psi) - \sigma \int_{0}^t \mathscr X_s^N(\phi)^2 ds\right)^k \\ & \qquad \qquad = k!\int_{\Delta_k(0,t)} \Ex^{(2k)}_\omega \left[\prod_{i=1}^k e^{-s_i\sqrt{N} + N^{-1/4}(X_{Ns_i}^{2i-1}+X_{Ns_i}^{2i})}\til{\Upsilon}_{s_i}^N(X^{2i-1},X^{2i})\right]\prod_{i=1}^k ds_i.
		\end{aligned}
  \end{equation}
		In the above line, we have also used the fact that $\int_{[0,t]^k} \prod_{i=1}^k F(x_i)dx_i = k!\int_{\Delta_k(0,t)} \prod_{i=1}^k F(x_i)dx_i$ (recall $\Delta_k$ from \eqref{deltadef}). Here $\Ex_{\omega}^{(2k)}$ denotes the quenched expectation of a $2k$-point motion sampled from the environment $\omega:=\{K_{s,t} : -\infty<s<t<\infty\}$. Taking annealed expectation on both sides of the above equation, interchanging the order of integration and expectation, and then using the fact that the averaged law of $(X^1,\ldots, X^{2k})$ is a $2k$-point sticky Brownian motion with characteristic measure $\nu$, we get
		\begin{align*}
			& \Ex\left[\left(Q_t^N(\psi) - \sigma \int_0^t \mathscr X_s^N(\phi)^2 ds\right)^k\right] \\ & \qquad \qquad = k!\int_{\Delta_k(0,t)} \mE_{SB_{\nu}^{(2k)}}\left[\prod_{i=1}^k e^{-s_i\sqrt{N} + N^{-1/4}(X_{Ns_i}^{2i-1}+X_{Ns_i}^{2i})}\til{\Upsilon}_{s_i}^N(X^{2i-1},X^{2i})\right]\prod_{i=1}^k ds_i.
		\end{align*}
 The interchanging of the order of integration and expectation is permissible by Fubini's theorem. Indeed, Fubini's theorem is applicable as applying Lemma \ref{traps}, one can check that for each fixed $N$, the expectation of the absolute value of the integrand on the right-hand side of \eqref{ee1} is uniformly bounded as $\vec{s}$ varies in $\Delta_k(0,t)$.  Using the fact that $N^{-1/2}(X_{Ns}^1,\ldots, X_{Ns}^{2k})$ is distributed as $SB^{(2k)}_{N^{1/2}\nu}$ measure, \eqref{e:QXmom} follows from the above formula. Relying on \eqref{qrep} alone, one can derive the formula in \eqref{e:Qincmom} by the exact same argument. This completes the proof. 
	\end{proof}
	
	We now come to the main result of this section, which shows that the quadratic martingale field $Q_t^N$ is well approximated by the integrated version of the square of $\mc{X}_t^N$. 
	Loosely speaking, for each $a\in \R\setminus \{0\}$ we shall show as $N\to \infty$
	\begin{align*}
		Q_t^N(\delta_a)-\sigma\int_0^t\left(\mc{X}_s^N(\delta_a)\right)^2 ds \stackrel{L^2}{\to} 0.
	\end{align*}
	{In other words, mollifying the squared field is comparable to squaring the mollified field.} Since $Q_t^N$ and $\mc{X}_s^N$ exist as distributions, to make sense of the above display, we work with a sequence of Gaussian test functions converging to $\delta_a$.
	
	\begin{prop}[Limiting behavior of the QMF] \label{4.1}Let $a \in \mathbb R$ and let $\xi(x):= \frac1{\sqrt{\pi}}e^{-x^2}$ be the Gaussian test function and let $\xi_\e^a (x):= \e^{-1}\xi(\e^{-1} (x-a))$.  
  Then for all $t>0$ and $a\in\mathbb R\setminus\{0\}$,
		\begin{align}
			\label{Qllim}
			\limsup_{\e \to 0} \limsup_{N \to \infty}  \Ex\bigg[ \bigg( Q_t^N(\xi_\e^a)-\sigma \int_0^t\mathscr X_s^N\big(\xi_{\e\sqrt{2}}^a\big)^2ds\bigg)^2 \bigg] = 0,
		\end{align}
		where as usual, $\sigma = \frac{1}{{2}\nu([0,1])}$. Furthermore, we have the bound \begin{equation}\label{e:QXpolylog}\sup_{\substack{\e>0\\a\in\mathbb R\setminus\{0\}}} \limsup_{N \to \infty}  \left[1\wedge(|\log a|^{-2})\right]\cdot \Ex\bigg[ \bigg( Q_t^N(\xi_\e^a)-\sigma \int_0^t\mathscr X_s^N\big(\xi_{\e\sqrt{2}}^a\big)^2ds\bigg)^2 \bigg] <\infty.
		\end{equation}
	\end{prop}
	
	We remark that the relevance of the second bound \eqref{e:QXpolylog} is that even though we may not have convergence to 0 in \eqref{Qllim} when $a=0$, we can still control the behavior near $a=0$ by a blow-up that is polylogarithmic at worst. {Note that the $|\log a|^{-2}$ factor appears above only when $a\le e$}. In the proof of \eqref{mp3} in Theorem \ref{solving_mp} below, we will see that such blow-ups are irrelevant in the limit. 
	
	\medskip
	
	The starting point of the proof of Proposition \ref{4.1} is the moment formula in Lemma \ref{l:Qmom} \ref{l:QXmom}. As noted in \eqref{e:QXmom}, the moments can be expressed as integrals of the expectation of certain observables under sticky Brownian motion. Recall that in Theorem \ref{converge} we saw that the expectation of a certain class of observables under sticky Brownian motion converges to the expectation of those observables under standard $k$ dimensional Brownian motion. That theorem will be used to prove a similar convergence result for the expectations of the type of observables that appear in \eqref{e:QXmom}.

\begin{prop}\label{add} Fix any $k\in \mathbb N$ and $0=t_0<t_1< t_2<\cdots<t_{k}<\infty$. Set $s_0=t_0$ and $s_{2i-1}=s_{2i}=t_i$ for $1\le i\le k$. Suppose $(X^1,\ldots,X^{2k})$ is distributed according $SB_{N^{1/2}\nu}^{(2k)}$. Let $(\phi_i,\psi_i)_{i=1}^{2k}$  be bounded continuous functions on $\mathbb R$. Define 
\begin{align}
			\label{def:v}
			\V_k(\vec{t}):=\sum_{r=1}^k\sum_{2r-1\le p<q\le 2k} \left[L_0^{U^p-U^q}(t_{r})-L_0^{U^p-U^q}(t_{r-1})\right].
\end{align}	
We have the following two limits:

\begin{enumerate}[label=(\alph*),leftmargin=15pt]
\item \label{item:expectationLimits1}
		\begin{equation}
			\label{e:add1}
			\begin{aligned}
				& \lim_{N\to\infty} \mathbf E_{SB^{(2k)}_{N^{1/2}\nu}}\bigg[\prod_{i=1}^{2k} e^{-\frac12s_{i}\sqrt{N} + N^{\frac14}X_{s_i}^{i}}\phi_i(X_{s_i}^{i} - N^{\frac14}s_i)\bigg] = \mathbf E_{B^{\otimes 2k}} \bigg[ e^{{\frac\sigma2} \V_k(\vec{t})} \prod_{i=1}^{2k}  \phi_i(U_{s_i}^{i}) \bigg].
			\end{aligned}
		\end{equation}

\item \label{item:expectationLimits2} {Let $A$ be any subset of $\{1,2,\ldots,k\}$. Set $B=A^c$ and
define
\begin{align}
    E_1(\vec{t}):=\prod_{i\in A} e^{-\frac12t_{i}\sqrt{N} + N^{\frac14}X_{t_i}^{2i-1}}\phi_i(X_{t_i}^{2i-1} - N^{\frac14}t_i)\prod_{i\in A} e^{-\frac12t_{i}\sqrt{N} + N^{\frac14}X_{t_i}^{2i}}\phi_i(X_{t_i}^{2i} - N^{\frac14}t_i),
\end{align}
\begin{align}
    E_2(\vec{t}):=\prod_{i\in B} e^{-t_{i}\sqrt{N} + N^{\frac14}(X^{2i-1}_{t_{i}}+X^{2i}_{t_{i}})}\psi_i(X_{t_{i}}^{2i} - N^{\frac14}t_i)  \ind_{\{X_{t_{i}}^{2i-1} =X_{t_{i}}^{2i}\}}.
\end{align}
Recall $\Delta_k(s,t)$ from \eqref{deltadef}. For each $0\le s< t\le T<\infty$ we have
		\begin{equation}
			\label{e:add2}
			\begin{aligned}
				& \lim_{N\to\infty} N^{\frac{|B|}2}\cdot\mathbf E_{SB^{(2k)}_{N^{1/2}\nu}}\bigg[ \int_{\Delta_k(s,t)}E_1(\vec{t})E_2(\vec{t})\prod_{i=1}^k dt_i \bigg] \\ & \hspace{0cm}= \left({\frac{\sigma}2}\right)^{|B|}\mathbf E_{B^{\otimes 2k}} \bigg[\int_{\Delta_k(s,t)} e^{{\frac{\sigma}2}\V_k(\vec{t})} \prod_{i\in A} \phi_i(U_{t_i}^{2i-1})\phi_i(U_{t_i}^{2i}) dt_i\prod_{i\in B} \psi_i(U_{t_{i}}^{2i}) dL_0^{U^{2i-1}-U^{2i}}(t_i) \bigg]. 
			\end{aligned}
		\end{equation}}
		Here  $\int_0^t f(s) dL_0^{U^i-U^j}(s)$ denotes the integration of the continuous function $f:[0,t]\to\mathbb R$ against the random Lebesgue-Stiltjes measure $dL_0^{U^i-U^j}$ induced from the increasing function $t\mapsto L_0^{U^i-U^j}(t).$
  \end{enumerate}
	\end{prop}

	The proof of Proposition \ref{add} builds up on the estimates proved in Section \ref{sec:girs} and uses Theorem \ref{converge}. We postpone the proof of Proposition \ref{add} to Section \ref{sec:pftech} and complete the proof of Proposition \ref{4.1} assuming it.

	\begin{proof}[Proof of Proposition \ref{4.1}] For clarity, we split the proof into three steps. We first provide a brief outline of the steps below. Note that, given the moment formulas from  Lemma \ref{l:Qmom} \ref{l:QXmom} and the limit of those formulas from Proposition \ref{add}, one can compute the $N\to \infty$ limit of the expectation in \eqref{e:QXpolylog} in terms of expectation of certain observables under a $4$D Brownian motion measure $(X^1,X^2,Y^1,Y^2)$. We then proceed to carefully massage this expectation formula to obtain the desired result. 
 
 \begin{itemize}[leftmargin=18pt]
 \itemsep\setlength{0.5em}
     \item In \textbf{Step 1}, we use a simple transformation to express the limit in terms of an expectation under a different $4$D Brownian motion measure: 
     $$(X^1-Y^1,X^2-Y^2,X^1+Y^1,X^2+Y^2).$$
     \item In \textbf{Step 2}, we show how local time heuristics (which are rigorously shown in Appendix \ref{app}) can be used to rewrite the expectation formula in terms of a certain concatenation of Brownian bridge and Brownian motion measures.
     \item In \textbf{Step 3}, we use local time estimates for concatenated processes (Lemma \ref{bcts}) and heat kernel calculations to bound the formula obtained in the previous step. The bound obtained in this step is sharp enough to conclude the proof of Proposition \ref{4.1}.
 \end{itemize}  
		
		\medskip
		
		\noindent\textbf{Step 1.} Applying Lemma \ref{l:Qmom} \ref{l:QXmom} and Proposition \ref{add} \ref{item:expectationLimits2} with $k=2$ we get
		\begin{align}
			\nonumber    & \lim_{N\to \infty}  \Ex\left[\bigg(Q_t^N(\psi)-\sigma \int_0^t \mathscr X_s^N(\phi)^2 ds\bigg)^{2}\right] \\ \nonumber & = 2\sigma^2 \cdot \mE_{B^{\otimes 4}}\left[\int_{\Delta_2(0,t)} e^{{\frac\sigma2} \V_2(s_1,s_2)}\prod_{i=1}^2 \left(\psi(X_{s_i}^i) {\frac12}dL_0^{X^i-Y^i}(s_i)-\phi(X^i)\phi(Y^i)ds_i\right)\right]   \\ & = 2\sigma^2 \cdot \mE_{B^{\otimes 4}}\left[\int_{\Delta_2(0,t)} e^{{\frac\sigma2} \V_2(s_1,s_2)}\prod_{i=1}^2 \left(\psi\big(\tfrac12(X_{s_i}^i+Y_{s_i}^i)\big) {\frac12} dL_0^{X^i-Y^i}(s_i)-\phi(X^i)\phi(Y^i)ds_i\right)\right] \label{e:smom1}
		\end{align}
		for all $\psi,\phi \in \mathcal{S}(\R)$, where $\V_2$ is defined in \eqref{def:v}. The second equality in the above equation follows by observing that $X_u^i=Y_u^i$ for $u$ in the support of $L_0^{X^i-Y^i}(du)$. 
		
		\medskip
		
		We shall now write $\mE$ instead of $ \mE_{B^{\otimes 4}}$ for convenience. Let us now take  $$\psi(x) := \xi^a_\e(x)= \frac1{\sqrt{\pi \e^2 }}e^{-(x-a)^2/\e^2}, \quad \phi(x) := \xi^a_{\e\sqrt{2}}(x)= \frac1{\sqrt{2\pi \e^2 }}e^{-(x-a)^2/2\e^2},$$
		in \eqref{e:smom1}. Using the identity $\xi^a_{\e\sqrt{2}}(x)\xi^a_{\e\sqrt{2}}(y)= \xi^a_\e((x+y)/2)\xi^0_{2\e}(x-y)$, we may now write \eqref{e:smom1} as 
		\begin{align}
			\label{e:smom2}
			2\sigma^2 \cdot \mE\left[\int_{\Delta_2(0,t)} e^{{\frac\sigma2} \V_2(s_1,s_2)}\prod_{i=1}^2 \left(\xi_\e^a\big(\tfrac12(X_{s_i}^i+Y_{s_i}^i)\big)\cdot\left({\frac12}dL_0^{X^i-Y^i}(s_i)-\xi_{2\e}^{0}(X_{s_i}^i-Y_{s_i}^i)ds_i\right)\right)\right].
		\end{align}
		Let us write $U^{i,-}:=X^i-Y^i$ and $U^{i,+}:=X^i+Y^i$. Note that under $\mathbf P_{B^{\otimes 4}}$ the four processes $U^{1,-},U^{1,+},U^{2,-},U^{2,+}$ are independent Brownian motions with diffusion coefficient $2$. This enables us to view \eqref{e:smom2} as
		\begin{align}
			\nonumber & 2\sigma^2 \cdot \mE\left[\int_{\Delta_2(0,t)} e^{{\frac\sigma2} \V_2(s_1,s_2)}\prod_{i=1}^2 \left(\xi_\e^a\big(\tfrac12U_{s_i}^{i,+})\big)\cdot\left({\frac12}dL_0^{U^{i,-}}(s_i)-\xi_{2\e}^{0}(U_{s_i}^{i,-})ds_i\right)\right)\right]
			\\ & =: 2\sigma^2 [A_1(\e)-A_2(\e)-A_3(\e)+A_4(\e)], \label{e:smom3}
		\end{align}
		where
		\begin{align} \label{a1}
			A_1(\e) & := \mE\left[\int_{\Delta_2(0,t)}  e^{{\frac\sigma2} \V_2(s_1,s_2)}\xi_\e^a\big(\tfrac12U_{s_1}^{1,+}\big)\xi_\e^a\big(\tfrac12U_{s_2}^{2,+}\big)\,{\frac12}dL_0^{U^{1,-}}(s_1)\,{\frac12}dL_0^{U^{2,-}}(s_2)\right] \\ \label{a2}
			A_2(\e) & := \mE\left[\int_{\Delta_2(0,t)} e^{{\frac\sigma2} \V_2(s_1,s_2)}  \xi_\e^a\big(\tfrac12U_{s_1}^{1,+}\big)\xi_\e^a\big(\tfrac12U_{s_2}^{2,+}\big)\xi_{2\e}^{0}(U_{s_2}^{2,-})\, {\frac12}dL_0^{U^{1,-}}(s_1)\, ds_2\right]\\ \label{a3}
			A_3(\e) & := \mE\left[\int_{\Delta_2(0,t)} e^{{\frac\sigma2} \V_2(s_1,s_2)} \xi_\e^a\big(\tfrac12U_{s_1}^{1,+}\big)\xi_\e^a\big(\tfrac12U_{s_2}^{2,+}\big)\xi_{2\e}^{0}(U_{s_1}^{1,-})\, {\frac12} dL_0^{U^{2,-}}(s_2)\, ds_1\right] \\ \label{a4}
			A_4(\e) & := \mE\left[\int_{\Delta_2(0,t)} e^{{\frac\sigma2} \V_2(s_1,s_2)}\xi_\e^a\big(\tfrac12U_{s_1}^{1,+}\big)\xi_\e^a\big(\tfrac12U_{s_2}^{2,+}\big)\xi_{2\e}^{0}(U_{s_1}^{1,-})\xi_{2\e}^{0}(U_{s_2}^{2,-})\,ds_1\,ds_2\right].
		\end{align}
		
		\medskip
		
		\noindent\textbf{Step 2.} In this step we focus on each of the $A_i(\e)$ terms separately. We drop the $\e$ and write $A_i$ for simplicity. Note that informally ${\frac12}dL_0^{U^{i,-}}(s_i)$ may be written as $\delta_0(U^{i,-}) ds_i$ which suggests that each of the $A_i$ may be written in terms of Brownian bridge expectations. To this end, we introduce the function
		\begin{align}\label{defh}
			H_{s_1,s_2}(\vec{x},\vec{y}):=\mE\left[e^{{\frac\sigma2} \V_2(s_1,s_2)} \mid U_{s_i}^{i,-}=x_i, U_{s_i}^{i,+}=y_i, i=1,2\right]
		\end{align}
		where the above conditional expectation is interpreted as taking expectation under the measure where $U^{i,-}$ ($U^{i,+}$ resp.) is a concatenation of a Brownian bridge from $0$ to $x_i$ ($0$ to $y_i$ resp.) over the interval $[0,s_i]$ and an independent Brownian motion started from $x_i$ ($y_i$ resp.) on $[s_i,t]$. All Brownian objects considered are independent with diffusion coefficient $2$.

		\smallskip
		
		\noindent\underline{$A_1$ term:} Let us consider the $A_1$ term from \eqref{a1}. Let $\mathcal{F}_{t}^{-}$ be the $\sigma$-field generated by $\{U_s^{1,-},U_s^{2,-} \mid 0\le s\le t\}$. By the tower property of conditional expectation, we have
		\begin{align*}
			A_1 & := \mE\left[\int_{\Delta_2(0,t)}  \mE\left[e^{{\frac\sigma2} \V_2(s_1,s_2)}\xi_\e^a\big(\tfrac12U_{s_1}^{1,+}\big)\xi_\e^a\big(\tfrac12U_{s_2}^{2,+}\big)\mid \mathcal{F}_t^{-}\right] {\frac12}dL_0^{U^{1,-}}(s_1)\,{\frac12}dL_0^{U^{2,-}}(s_2)\right].
		\end{align*}
		By Lemma \ref{6.1}, the right-hand side of the above equation simplifies to
		\begin{align}\label{a11}
			A_1 = \int_{\Delta_2(0,t)} \mE\left[e^{{\frac\sigma2}\V_2(s_1,s_2)}\xi_\e^a\big(\tfrac12U_{s_1}^{1,+}\big)\xi_\e^a\big(\tfrac12U_{s_2}^{2,+}\big)\mid U_{s_i}^{i,-}=0, i=1,2\right]  p_{2s_1}(0)p_{2s_2}(0)\,ds_1\,ds_2
		\end{align}
		where again the conditional expectation is interpreted via concatenation of Brownian bridges and Brownian motions. Note that via Brownian motion decomposition, we have
		\begin{equation}
			\label{trick}
			\begin{aligned}
				& \mE\left[e^{{\frac\sigma2} \V_2(s_1,s_2)}\cdot \xi_\e^a(\tfrac12U_{s_1}^{1,+})\xi_\e^a(\tfrac12U_{s_2}^{2,+})\mid U_{s_i}^{i,-}=x_i.i=1,2\right] \\ & \hspace{2cm}= \int_{\R^2} H_{s_1,s_2}(\vec{x},\vec{y})\prod_{i=1}^2 \xi_\e^a(\tfrac12y_i)p_{2s_i}(y_i)dy_i
			\end{aligned}
		\end{equation}
		where $H$ is defined in \eqref{defh}. Inserting the above formula back in \eqref{a11} we get
		\begin{align}\label{a12}
			A_1 = \int_{\Delta_2(0,t)}\int_{\R^2} H_{s_1,s_2}((0,0),(y_1,y_2)) \prod_{i=1}^2 \xi_\e^a(\tfrac12y_i)p_{2s_i}(y_i) p_{2s_i}(0)\,dy_i\,ds_i.
		\end{align}
		
		\smallskip
		
		\noindent\underline{$A_2$ and $A_3$ terms:} Recall $A_2$ from \eqref{a2}. We condition only on $\mathcal{F}_t^{1,-}:=\sigma\{U_s^{1,-}\mid 0\le s\le t\}$ to get
		\begin{align*}
			A_2 & := \mE\left[\int_{0}^t F(U^{1,-}) {\frac12} dL_0^{U^{1,-}}(s_1) \right]
		\end{align*}
		where
		\begin{align*}
			F(U^{1,-}):=\int_0^{s_1} \mE\left[e^{{\frac\sigma2}\V_2(s_1,s_2)}  \xi_\e^a\big(\tfrac12U_{s_1}^{1,+}\big)\xi_\e^a\big(\tfrac12U_{s_2}^{2,+}\big)\xi_{2\e}^{0}(U_{s_2}^{2,-})\mid \mathcal{F}_t^{1,-}\right]\,ds_2
		\end{align*}
		Applying Lemma \ref{6.1} again, we see that 
		\begin{align*}
			A_2 = \int_0^t \Ex[F(U^{1,-}) \mid U_{s_1}^{1,-}=0] p_{2s_1}(0) ds_2.
		\end{align*}
		where again the conditional expectation is interpreted in a similar manner. Performing a similar trick as in \eqref{trick}, we see that
		\begin{align}\label{a22}
			A_2 = \int_{\Delta_2(0,t)} \int_{\R^3} H_{s_1,s_2}((0,x_2);\vec{y})\left[\prod_{i=1}^2 \xi_\e^a\big(\tfrac12y_i\big)p_{2s_i}(y_i)\right]\xi_{2\e}^{0}(x_2)p_{2s_2}(x_2)p_{2s_1}(0) d\vec{y}\,dx_2\, d\vec{s}.
		\end{align}
		Similarly for $A_3$ defined in \eqref{a3} one has
		\begin{align} \label{a32}
			A_3 = \int_{\Delta_2(0,t)} \int_{\R^3} H_{s_1,s_2}((x_1,0);\vec{y})\left[\prod_{i=1}^2 \xi_\e^a\big(\tfrac12y_i\big)p_{2s_i}(y_i)\right]\xi_{2\e}^{0}(x_1)p_{2s_1}(x_1)p_{2s_2}(0) d\vec{y}\,dx_1\, d\vec{s}.
		\end{align}
		
		\smallskip
		
		\noindent\underline{$A_4$ term:} Recall $A_4$ from \eqref{a4}. Note that $A_4$ does not involve any integration with respect to local times. Thus we may use $H$ from \eqref{defh} directly and tricks similar to \eqref{trick} to get
		\begin{align}\label{a42}
			A_4:=\int_{\Delta_2(0,t)} \int_{\R^4} H_{s_1,s_2}(\vec{x},\vec{y})\prod_{i=1}^2 \xi_\e^a\big(\tfrac12y_i\big)\xi_{2\e}^{0}(x_i)p_{2s_i}(x_i)p_{2s_i}(y_i)dx_i\,dy_i\, ds_i.
		\end{align}
		
		\medskip
		
		\noindent\textbf{Step 3.} In this step, we determine convergence and provide bounds for each of the $A_i$ terms defined at the end of \textbf{Step 1}. We claim that for each $i$,
		\begin{align}\label{alim}
			\lim_{\e\to 0} A_i(\e) = {4}\int_{\Delta_2(0,t)} H_{s_1,s_2}((0,0),(2a,2a))\prod_{i=1}^2 p_{2s_i}(0)p_{2s_i}(2a)ds_i,
		\end{align}
		where $H$ is defined in \eqref{defh}. Inserting this limit back in \eqref{e:smom3} verifies \eqref{Qllim}. Let us consider $A_1(\e)$ from \eqref{a1}. Let us consider the integrand on the right-hand side of \eqref{a12}:
		\begin{align*}
			p_{2s_i}(0)\int_{\R^2} H_{s_1,s_2}((0,0),(y_1,y_2) )\prod_{i=1}^2 \xi_\e^a(\tfrac12y_i)p_{2s_i}(y_i) \,dy_i.
		\end{align*}
		Clearly as $\e\to 0$, it converges to {$4 p_{2s_i}(0) H_{s_1,s_2}((0,0),(2a,2a)) p_{2s_1}(2a)p_{2s_2}(2a)$}. 
  Thus, to show \eqref{alim} holds for $A_1(\e)$, it suffices by the dominated convergence theorem to show that for all $a\neq 0$,
		\begin{align}\label{supbd}
			\sup_{\e\in (0,1]} p_{2s_i}(0)\int_{\R^2} H_{s_1,s_2}((0,0),(y_1,y_2) )\prod_{i=1}^2 \xi_\e^a(\tfrac12y_i)p_{2s_i}(y_i) \,dy_i
		\end{align}
		is dominated by some integrable function of $s_1,s_2$. Towards this end, note that $H$ is defined as the expectation of exponentials of local times of certain linear functions of concatenated processes introduced at the beginning of \textbf{Step 2}. In Lemma \ref{bcts}, we study these expectations and in particular show that for all $s_1,s_2\in [0,t]$ and for all $\vec{x},\vec{y}\in \R^2$, $H_{s_1,s_2}(\vec{x},\vec{y}) \le \Con$ for some constant $\Con>0$ that depends only on $t$. As $\xi_{\e}^a(\frac12y)=2p_{2\e^2}(y-2a)$, using $H_{s_1,s_2}(\vec{x},\vec{y})\le \Con$ and then using the semigroup property of the heat kernel (i.e., $\int_{\R} p_u(x-b)p_{t}(x)dx =p_{u+t}(b)$) we have from \eqref{a12} that
		\begin{align}\label{supbd2}
			\mbox{\eqref{supbd}} \le \Con  \sup_{\e\in (0,1]}\prod_{i=1}^2 p_{2s_i+2\e^2}(2a)p_{2s_i}(0). 
		\end{align}
		Note that the heat kernel globally satisfies the bound $p_t(x) \leq \Con(t+x^2)^{-1/2}\ind_{\{t>0\}}$ for some large enough constant $\Con>0.$ Consequently we have
		\begin{align*}
			\int_{\Delta_2(0,t)}\sup_{\e\in (0,1]}\prod_{i=1}^2 p_{2s_i+2\e^2}(2a)p_{2s_i}(0) ds_i &  \le \Con\int_{\Delta_2(0,t)} \prod_{i=1}^2 s_i^{-1/2}(s_i+a^2)^{-1/2}\,ds_i \\ & = \Con\int_{\Delta_2(0,a^{-2}t)} \prod_{i=1}^2 v_i^{-1/2}(v_i+1)^{-1/2}\,dv_i.
		\end{align*}
		where the equality above follows by making the substitution $s_i=a^2v_i$. Extending the range of integration we have
		\begin{align}
			\int_{\Delta_2(0,a^{-2}t)} \prod_{i=1}^2 v_i^{-1/2}(v_i+1)^{-1/2}\,dv_i &  \le \int_{[0,a^{-2}t]^2} \prod_{i=1}^2 v_i^{-1/2}(v_i+1)^{-1/2}\,dv_i  \le \Con \max\{1, |\log a|^2\}. \label{abd}
		\end{align}
		This verifies the integrability of the bound in \eqref{supbd2} and consequently proves \eqref{alim} for $A_1$. Starting from \eqref{a22}, \eqref{a32}, and \eqref{a42}, an analogous computation verifies \eqref{alim} for $A_2,A_3$ and $A_4$ respectively. This establishes \eqref{Qllim}. From the polylogarithmic bound in \eqref{abd} (and its analogous counterparts for $A_2$, $A_3$, $A_4$) one arrives at \eqref{e:QXpolylog}. This completes the proof.
	\end{proof}

	\begin{prop}[Error term limit]\label{4.2} For any $t\ge 0$ and $\phi\in \mathcal{S}(\R)$, we have that 
		\begin{align}\label{e42}
			\limsup_{N\to \infty}\Ex[\mathcal E_t^N(\phi)^2] = 0.
		\end{align}
	\end{prop}
	\begin{proof}
		Note that $\phi\phi' \leq \frac12( \phi^2 +(\phi')^2)$ thus 
		$$|\mathcal E_t^N(\phi)| \le 2N^{1/4} \left[\int_0^t ((\mathscr X_s^N)^{Sq}, \phi^2) ds+ \int_0^t((\mathscr X_s^N)^{Sq},(\phi')^2)ds\right] = 2N^{-1/4}\left[Q_t^N(\phi^2)+Q_t^N(\phi'^2)\right].$$ 
		Now by Lemma \ref{l:Qmom} \ref{l:Qincmom} and Proposition \ref{add} \ref{item:expectationLimits1}, we see that for every $\psi\in \mathcal{S}(\R)$, $L^2$ moments of $Q_t^N(\psi)$ are uniformly bounded in $N$. Since $\phi^2$ and $\phi'^2$ are also Schwartz functions, in view of the above inequality, we readily have \eqref{e42}.
	\end{proof}

	\subsection{Supporting estimates} \label{sec:pftech} In this section we prove Proposition \ref{add} and establish an estimate for the moments of the increments of the quadratic martingale field which will be useful in Section \ref{sec:tight}.

 \begin{proof}[Proof of Proposition \ref{add}] \noindent\textbf{Proof of \eqref{e:add1}.}  We continue with the same notation as in the statement of the Proposition \ref{add}. Let us set
 \begin{align*}
     \mathcal{G}_N^r(t) =\sum_{j= 2r-1}^{2k} \int_{0}^{t} \bigg[1-\frac1{m_j^r(\mathbf X(u))}\bigg]dX^j(u) 
	\end{align*} 
  where $m_j^r(\mathbf x):= \#\{ i\in \{2r-1,\ldots,2k\}: x_i=x_j\}$. Note that $\mathcal{G}_N^r$ are martingales of the form \eqref{def:gm} where only the last $2k-2r$ of the $2k$ particles are taken into account. Define the martingale
 \begin{align}
 \label{anbargn}
     d\bar{\mathcal G}_N(t):= N^{1/4}\sum_{r=1}^k  \ind_{\{t\in [t_{r-1},t_{r})\}} d \mathcal{G}_N^r(t).
 \end{align}
It turns out that the stochastic exponential of $\bar{\mathcal{G}}_N$ is precisely the tilt that gets rid of the divergent term appearing in the expectation of the left-hand side of \eqref{e:add1}. To be precise, we have that 
	\begin{equation}
		\label{e.multifr}
		\begin{aligned}
			& \mathbf E_{SB^{(2k)}_{N^{1/2}\nu}}\bigg[\prod_{i=1}^{2k} e^{-\frac12s_{i}\sqrt{N} + N^{\frac14}X_{s_i}^{i}}\phi_i(X_{s_i}^{i} - N^{\frac14}s_i)\bigg] \\ &= \mathbf E_{SB^{(2k)}_{N^{1/2}\nu}} \bigg[ e^{\bar{\mathcal G}_N(T) -\tfrac12\langle\bar{\mathcal G}_N \rangle(T)} \exp\bigg(N^{\frac12} \sum_{r=1}^k\sum_{2r-1\le i<j\le 2k}\hspace{-0.1cm} \int_{t_{r-1}}^{t_r} 
			\hspace{-0.1cm}\ind_{\{X_u^i=X_u^j\}}du\bigg)\prod_{r=1}^{2k} \phi_r(X_{s_r}^{r}) \bigg].
		\end{aligned}
	\end{equation}
	Note that  $N^{1/4}\mathcal G_N^r$ satisfies the assumption \eqref{assn1}. Consequently, $\bar{\mathcal G}_N$ satisfies the assumption \eqref{assn1} as well. From Theorem \ref{converge} \ref{convb}, it is immediate that, as $N\to\infty$, the right-hand side of \eqref{e.multifr} converges to the right-hand side of \eqref{e:add1}. This proves \eqref{e:add1} modulo \eqref{e.multifr}.

We now turn towards the proof of \eqref{e.multifr} which is done by iteratively applying \eqref{e:condeq2}. We illustrate the proof for the $k=2$ case only; the general case follows in the exact same manner. Let us set $k=2$ and write $\mE$ for the expectation with respect to ~${SB^{(2k)}_{N^{1/2}\nu}}$ measure. Recall that $s_{2i-1}=s_{2i}=t_i$ for $i=1,2$. Let $(\mathcal{F}_t)_{t\le T}$ denote the filtration 
of the $2k$-point sticky Brownian motion. Note that the expression inside the first expectation in \eqref{e.multifr} can be written as $\mathbf{A} \cdot \mathbf{B}$ where
\begin{align*}
	\mathbf{A} & := e^{-2t_1\sqrt{N} + N^{1/4}\sum_{j=1}^4 X_{t_1}^j} \prod_{i=1}^2 \phi_i(X_{s_i}^{i} - N^{\frac14}s_i) \\ \mathbf{B} & := e^{-(t_2-t_1)\sqrt{N} + N^{1/4}(X_{t_2}^3+X_{t_2}^4-X_{t_1}^3-X_{t_1}^4)} \cdot \prod_{i=3}^4 \phi_i(X_{s_i}^{i} - N^{\frac14}s_i).
\end{align*}
$\mathbf{A}$ is measurable with respect to ~$\mathcal{F}_{t_1}$. By the Markov property of sticky Brownian motion $\mE[\mathbf{B}\mid \mathcal{F}_{t_1}]$ is measurable with respect to ~$\sigma\big(X_{t_1}^i)_{i=1}^4\big)$. Let us write
\begin{align*}
	H\big(\mathbf X_{t_1}\big) : =\mE[\mathbf{B}\mid \mathcal{F}_{t_1}].
\end{align*}
By the tower property of conditional expectation followed by an application of \eqref{e:condeq2} with $\lambda=N^{\frac14}$, $s=0$, $t=t_1$, $\nu\mapsto N^{\frac12}\nu$, and $k=4$ we get
\begin{equation}
	\label{e.r2}
	\begin{aligned}
		\mbox{l.h.s.~of \eqref{e.multifr}} & = \mE\left[\mathbf{A}\cdot H\big(\mathbf X_{t_1}\big) \right] \\ & = \mE\left[ e^{N^{\frac14}\mathcal{G}_N^1(t_{1})-N^{\frac12} \langle\mathcal{G}_N^1\rangle(t_1)}\exp\bigg(N^{\frac12} \sum_{1\le i<j\le 4}\int_{t_{1}}^{t_2} 
		\hspace{-0.1cm}\ind_{\{X_u^i=X_u^j\}}du\bigg)\right. \\ & \hspace{3cm}\left. \cdot\prod_{i=1}^2 \phi_i(X_{s_i}^{i}) \cdot H\big(({X}_{t_1}^i+N^{\frac14}t_1)_{i=1}^k\big)\right].
	\end{aligned}
\end{equation}
Observe that by the Markov property of the sticky Brownian motion, we have
\begin{align}\label{e.r}
H\hspace{-0.05cm}\big(({X}_{t_1}^i+N^{\frac14}t_1)_{i=1}^k\big) \hspace{-0.1cm} =\hspace{-0.1cm} \mE \hspace{-0.1cm} \left[e^{-(t_2-t_1)\sqrt{N} + N^{\frac14}(X_{t_2}^3+X_{t_2}^4-X_{t_1}^3-X_{t_1}^4)} \prod_{i=3}^4 \phi_i(X_{s_i}^{i} - N^{\frac14}(t_2-t_1)) \hspace{-0.05cm} \mid \hspace{-0.05cm}\mathcal{F}_{t_1}\right]\!.\!
\end{align}
We use \eqref{e:condeq2} with $\lambda=N^{\frac14}$, $s=t_1$, $t=t_2$, $\nu\mapsto N^{\frac12}\nu$, and $k=2$ to get that
\begin{align*}
	\mbox{r.h.s.~of \eqref{e.r}}   & = \mE\left[e^{N^{\frac14}\mathcal{G}_N^2(t_{1},t_2)-N^{\frac12} \langle\mathcal{G}_N^2\rangle(t_1,t_2)} \cdot \prod_{i=3}^4 \phi_i(X_{s_i}^{i})  \exp\bigg(N^{\frac12} \int_{t_{1}}^{t_2} 
	\hspace{-0.1cm}\ind_{\{X_u^3=X_u^4\}}du\bigg) \mid \mathcal{F}_{t_1}\right], 
\end{align*}
where $g(s,t):=g(t)-g(s)$. Inserting the above expression back in \eqref{e.r2} and again using the tower property of the conditional expectation, we derive the identity in \eqref{e.multifr}.

\medskip

\noindent\textbf{Proof of \eqref{e:add2}.} Following the same argument as in the proof of \eqref{e:add1}, in the same spirit as \eqref{e.multifr}, one has that
	{\begin{equation}
 \label{e.multifr2}
	\begin{aligned}
		& \mathbf E_{SB^{(2k)}_{N^{1/2}\nu}}\bigg[E_1(\vec{t})E_2(\vec{t})\bigg] \\ &= \mathbf E_{SB^{(2k)}_{N^{1/2}\nu}} \Bigg[ e^{\bar{\mathcal G}_N(T) -\tfrac12\langle\bar{\mathcal G}_N \rangle(T)}\exp\bigg(N^{\frac12} \sum_{r=1}^k\sum_{2r-1\le i<j\le 2k}\hspace{-0.1cm} \int_{s_{r-1}}^{s_r} 
		\hspace{-0.1cm}\ind_{\{X_u^i=X_u^j\}}du\bigg)\\ & \hspace{3cm}\cdot\prod_{i\in A} \phi_i(X_{t_i}^{2i-1})\phi_i(X_{t_i}^{2i})\prod_{i\in B} \psi_i(X_{t_i}^{2i})\ind_{\{X_{t_i}^{2i-1}=X_{t_i}^{2i}\}} \Bigg].
	\end{aligned}
\end{equation}}
 Set
\begin{align}
    \label{e.ab}
    \mathbf{A}_N(\vec{t}):=e^{\bar{\mathcal G}_N(T) -\tfrac12\langle\bar{\mathcal G}_N \rangle(T)}, \quad \mathbf B_N(\vec{t}):=\exp\bigg(N^{\frac12} \sum_{r=1}^k\sum_{2r-1\le i<j\le 2k}\hspace{-0.1cm} \int_{t_{r-1}}^{t_r} 
		\hspace{-0.1cm}\ind_{\{X_u^i=X_u^j\}}du\bigg).
\end{align}
Recall $\gamma_N$ from \eqref{gamman}. By Theorem \ref{converge} \ref{convb} and \ref{convc},  any limit point of the sequence {$$\bigg(\bar{\mathcal{G}}_N, \mathbf{X}, {4}N^{1/2}\nu([0,1])\big(V^{ij}\big)_{1\leq i<j\leq 2k}, N^{|B|/2}\prod_{i\in B}\ind_{\{X_{t_i}^{2i-1}=X_{t_i}^{2i}\}}dt_i\bigg)$$ is of the form
\begin{align*}
    \left(\bar{\mathcal{G}}, \mathbf{U}, \big(L_0^{U_i-U_j}\big)_{1\leq i<j\leq k}, \left({\frac{\sigma}{2}}\right)^{|B|}\prod_{i\in B} dL_0^{U^{2i-1}-U^{2i}}(u_i)\right) 
\end{align*}}
for some $\bar{\mathcal{G}}$ with $\mE[e^{p \bar{\mathcal{G}}(T)-\frac{p}2\langle \bar{\mathcal{G}}\rangle (T)}]<\infty$ for all $p>0$ and satisfying \eqref{e:fil2}. Here $U$ is standard $2k$-dimensional Brownian motion. By the continuous mapping theorem, it follows that along a subsequence we have
{\begin{equation}
    \label{e:weak2}
    \begin{aligned}
	& N^{\frac{|B|}2}\int_{\Delta_k(s,t)} \hspace{-0.2cm}D_N(\vec{t}) \cdot \prod_{i\in B} \ind_{\{X_{t_i}^{2i-1}=X_{t_i}^{2i}\}} \prod_{i=1}^k dt_i \\ & \hspace{0cm} \stackrel{d}{\to} \left({\frac{\sigma}{2}}\right)^{|B|}\int_{\Delta_k(s,t)} e^{\bar{\mathcal G}(T) -\tfrac12\langle\bar{\mathcal G} \rangle(T)+{\frac{\sigma}{2}} \V_k(\vec{t})} \prod_{i\in A}\phi_i(U_{t_i}^{2i-1})\phi_i(U_{t_i}^{2i})dt_i\prod_{i\in B} \psi_i(U_{t_{i}}^{2i}) dL_0^{U^{2i-1}-U^{2i}}(t_i),
\end{aligned}
\end{equation}
where
\begin{equation}
	\label{e:weak}
	\begin{aligned}
	D_N(\vec{t}):=	 \mathbf{A}_N(\vec{t}) \cdot \mathbf{B}_N(\vec{t})\prod_{i\in A} \phi_i(X_{t_i}^{2i-1})\phi_i(X_{t_i}^{2i})\prod_{i\in B} \psi_i(X_{t_i}^{2i}).
	\end{aligned}
\end{equation}}

\medskip

We now upgrade \eqref{e:weak2} to imply convergence of first moments. To do this, we shall show the sequence of random variables in the left-hand side of \eqref{e:weak2} is uniformly integrable. Repeatedly using Proposition \ref{exp} and Doob's martingale inequality to the stochastic exponentials of each of the martingales $\mathcal G_N^r$ appearing in the expression for $\bar{\mathcal G}_N$ we find that for all $p\ge 1$, we have 
\begin{equation}\label{supmnab}\sup_{N\ge 1}\mathbf E_{SB^{(2k)}_{N^{1/2}\nu}}\bigg[\sup_{\vec{t}\in \Delta_k(0,T)} \big(\mathbf{A}_N(\vec{t})\big)^p\bigg]<\infty,  \quad \sup_{N\ge 1}\mathbf E_{SB^{(2k)}_{N^{1/2}\nu}}\bigg[\sup_{\vec{t}\in \Delta_k(0,T)} \big(\mathbf{B}_N(\vec{t})\big)^p\bigg]<\infty,
\end{equation}
and thus via the definition of $D_N(\vec{t})$ from \eqref{e:weak} we get
\begin{equation*}\sup_{N\ge 1}\mathbf E_{SB^{(2k)}_{N^{1/2}\nu}}\bigg[\sup_{\vec{t}\in \Delta_k(0,T)} \big(D_N(\vec{t})\big)^p\bigg]<\infty. 
\end{equation*}
In view of the above bound and \eqref{gmass1}, we obtain uniform integrability for the sequence of random variables in \eqref{e:weak2}. This implies that along the same subsequence,
{\begin{align*}
	& \lim_{N\to \infty} N^{\frac{|B|}2}\mE \left[\int_{\Delta_k(s,t)} \hspace{-0.2cm}D_N(\vec{t}) \cdot \prod_{i\in B} \ind_{\{X_{t_i}^{2i-1}=X_{t_i}^{2i}\}} \prod_{i=1}^k dt_i\right] \\ & \hspace{1cm} = \left({\frac{\sigma}{2}}\right)^{|B|}\mE\left[ \int_{\Delta_k(s,t)} e^{\bar{\mathcal G}(T) -\tfrac12\langle\bar{\mathcal G} \rangle(T)+{\frac{\sigma}2} \V_k(\vec{t})} \prod_{i\in A} \phi_i(U_{t_i}^{2i-1})\phi_i(U_{t_i}^{2i}) dt_i\prod_{i\in B} \psi_i(U_{t_{i}}^{2i}) dL_0^{U^{2i-1}-U^{2i}}(t_i)\right] \\ & \hspace{1cm} = \left({\frac{\sigma}{2}}\right)^{|B|}\mE\left[ \int_{\Delta_k(s,t)} e^{{\frac{\sigma}2} \V_k(\vec{t})} \mE \left[ e^{\bar{\mathcal G}(T) -\tfrac12\langle\bar{\mathcal G} \rangle(T)}  \mid \mathcal{F}_T(\mathbf{U})\right]\right. \\ & \hspace{4cm} \cdot\left.\prod_{i\in A} \phi_i(U_{t_i}^{2i-1})\phi_i(U_{t_i}^{2i}) dt_i\prod_{i\in B} \psi_i(U_{t_{i}}^{2i}) dL_0^{U^{2i-1}-U^{2i}}(t_i)\right],
\end{align*}}
where the last equality follows from the tower property of the conditional expectation. By \eqref{e:fil2}, the above inner expectation is $1$. Since the limit is free of $\bar{\mathcal{G}}$, we have the same above limit along every subsequence. Thus we arrive at the \eqref{e:add2} formula. This completes the proof.
\end{proof}

We end this section by recording a couple of useful estimates for moments of the increments of the quadratic martingale field.

\begin{prop}[Estimates for moments of the increments of QMF]\label{tight1} Fix $k\in\mathbb N$ and $T>0$. Then there exists a constant $\Con=\Con(k,T)>0$ such that for all bounded measurable functions $\phi$ on $\mathbb R$ and all $0\leq s<t\leq T$ one has that 
		\begin{align}
			\label{e.tight1}
			\sup_{N\ge 1}\Ex\big[ (Q_t^N(\phi)-Q_s^N(\phi))^{k}\big] \leq \Con\|\phi\|^{k}_{L^\infty(\mathbb R)}(t-s)^{k/4}.
		\end{align}
		Furthermore fix $p>1$ and $\e>0$. Then there exists $\Con=\Con(p,\e,k,T)>0$ such that for all functions $\phi \in L^p(\R)$ and all $\e \leq s<t\leq T$ one has 
		\begin{align}
			\label{e.tight2}
			\lim_{N\to\infty}\Ex\big[ (Q_t^N(\phi)-Q_s^N(\phi))^{k}\big] \leq \Con\|\phi\|_{L^p(\mathbb R)}^k(t-s)^k.
		\end{align}
	\end{prop}
We remark that the first bound is crude and nowhere near optimality. However, the latter bound is more important, and it will be most powerful when $p$ is very close to $1,$ as this will allow us to obtain optimal tightness bounds for limit points in the next section.	

 \begin{proof} \textbf{Proof of \eqref{e.tight1}.} Using \eqref{e:Qincmom} together with the trivial bound $|\phi(N^{-1/2}X^j_{t_j}-N^{1/4}t_j)|\leq \|\phi\|_{L^\infty},$ we obtain that 
		\begin{align}
			 \notag &\Ex[\big(Q_t^N(\phi)-Q_s^N(\phi)\big)^k]  \\ & \notag \le  N^{k/2} k!\|\phi\|_{L^\infty}^k\int_{\Delta_k(s,t)} \mathbf E_{SB^{(2k)}_{N^{1/2}\nu}}\bigg[\prod_{j=1}^k e^{-t_j\sqrt{N} + N^{1/4}(X^{2j-1}_{t_j}+X^{2j}_{t_j})}  \ind_{\{X_{t_j}^{2j-1} =X_{t_j}^{2j}\}} \bigg]dt_1\cdots dt_k \\ & =  k!\|\phi\|_{L^\infty}^k \cdot N^{k/2}\int_{\Delta_k(s,t)}\mathbf E_{SB^{(2k)}_{N^{1/2}\nu}}\bigg[\mathbf{A}_N(\vec{t})\cdot \mathbf{B}_N(\vec{t}) \cdot \prod_{j=1}^k  \ind_{\{X_{t_j}^{2j-1} =X_{t_j}^{2j}\}} \bigg]dt_1\cdots dt_k, \label{eqg}
		\end{align}
  where in above we use the same notation from \eqref{e.ab}:
  \begin{align*}
     \mathbf{A}_N(\vec{t}):=e^{\bar{\mathcal G}_N -\tfrac12\langle\bar{\mathcal G}_N \rangle}, \quad \mathbf{B}_N(\vec{t}):=\exp\bigg(N^{\frac12} \sum_{r=1}^k\sum_{2r-1\le i<j\le 2k}\hspace{-0.1cm} \int_{t_{r-1}}^{t_r} \hspace{-0.1cm}\ind_{\{X_u^i=X_u^j\}}du\bigg),
 \end{align*}
and where  $\bar{\mathcal G}_N=\bar{\mathcal{G}}_N(t)$ defined in \eqref{anbargn}. The equality in \eqref{eqg} is due to \eqref{e.multifr2}. 
Note that by Cauchy-Schwarz inequality and using the fact that $\mathbf{A}^2\cdot \mathbf{B}^2 \le \mathbf{A}^4+\mathbf{B}^4$, we have 
		\begin{align}  \mbox{\eqref{eqg}} &\leq  k!\|\phi\|_{L^\infty}^k\sqrt{N^{\frac{k}2}\int_{\Delta_k(s,t)}\mathbf E_{SB^{(2k)}_{N^{1/2}\nu}}\bigg[(\mathbf{A}_N(\vec{t})^4+ \mathbf{B}_N(\vec{t})^4) \cdot\prod_{j=1}^k  \ind_{\{X_{t_j}^j =Y_{t_j}^j\}} \bigg]dt_1\cdots dt_k} \label{bbound} \\ & \hspace{5.5cm}\cdot\sqrt{N^{\frac{k}2}\int_{\Delta_k(s,t)} \mathbf E_{SB^{(2k)}_{N^{1/2}\nu}}\bigg[\prod_{j=1}^k  \ind_{\{X_{t_j}^j =Y_{t_j}^j\}} \bigg] dt_1\cdots dt_k}\label{cbound}
		\end{align}
  For the factor in the square root of \eqref{cbound}, note that
  \begin{equation}
      \label{ccbound}
      \begin{aligned}
      N^{\frac{k}2}\int_{\Delta_k(s,t)} \mathbf E_{SB^{(2k)}_{N^{1/2}\nu}}\bigg[\prod_{j=1}^k  \ind_{\{X_{t_j}^j =Y_{t_j}^j\}} \bigg] dt_1\cdots dt_k & \le  \mathbf E_{SB^{(2k)}_{N^{1/2}\nu}}\bigg[\prod_{j=1}^k N^{\frac12}\int_s^t \ind_{\{X_{t}^j =Y_{t}^j\}} dt \bigg] \\ & \le  \prod_{j=1}^k\mathbf E_{SB^{(2k)}_{N^{1/2}\nu}}\bigg[ \bigg(N^{\frac12}\int_s^t \ind_{\{X_{t}^j =Y_{t}^j\}} dt\bigg)^k \bigg]^{\frac1k} \\ & \le \Con |t-s|^{k/2}.
  \end{aligned}
  \end{equation}
 where the last inequality follows by applying Lemma \ref{lte} on each of the expectations. Here the constant $\Con>0$ depends only on $k$.   
We now focus on the factor in the square root of \eqref{bbound}. We shall show that this factor can be bounded uniformly in $N\ge 1$ and $s,t\in [0,T]$. We split it into two parts: one containing $\mathbf{A}(\vec{t})$ and $\mathbf{B}(\vec{t})$. Recall the $\gamma_N$ measure from \eqref{gamman}. Note that	
		\begin{align*}
  & \sup_{N\ge 1}N^{k/2} \int_{\Delta_k(s,t)}\mathbf E_{SB^{(2k)}_{N^{1/2}\nu}}\bigg[\mathbf{A}_N(\vec{t})^4 \prod_{j=1}^k  \ind_{\{X_{t_j}^j =Y_{t_j}^j\}} \bigg]dt_1\cdots dt_k \\ & \hspace{2cm}\le \sup_{N\ge 1}\mathbf E_{SB^{(2k)}_{N^{1/2}\nu}} \bigg[\sup_{\vec{t}\in \Delta_k(0,T)}\mathbf A_N(\vec{t})^4 \int_{\Delta_k(s,t)} d\gamma_N\bigg] \\ & \hspace{2cm}\leq \sup_{N\ge 1}\sqrt{\mathbf E_{SB^{(2k)}_{N^{1/2}\nu}}\bigg[\sup_{\vec{t}\in \Delta_k(0,T)}\mathbf A_N(\vec{t})^8\bigg]\cdot\mathbf E_{SB^{(2k)}_{N^{1/2}\nu}}\big[\gamma_N(\Delta_k(0,T))^2\big]}.
		\end{align*} where the last inequality is due to Cauchy-Schwarz inequality. By \eqref{gmass1} and by \eqref{supmnab}, we find that the above expression is finite. On the other hand, for the $\mathbf B(\vec{t})$ term it is clear from Theorem \ref{converge}\ref{conva} and \ref{convc} that $$N^{k/2} \int_{\Delta_k(0,T)}\mathbf E_{SB^{(2k)}_{N^{1/2}\nu}}\bigg[\mathbf{B}_N(\vec{t})^4 \prod_{j=1}^k  \ind_{\{X_{t_j}^j =Y_{t_j}^j\}} \bigg]dt_1\cdots dt_k$$ converges to some finite quantity as $N\to\infty$, which implies that it is bounded independently of $N$. Thus the last two math displays imply that the term in \eqref{bbound} is bounded uniformly in $N\ge 1$ and $s,t\in [0,T]$. From \eqref{ccbound}, we see that \eqref{cbound} $\le \Con|t-s|^{k/4}$. Inserting these two bounds back in \eqref{bbound} and \eqref{cbound}, we arrive at \eqref{e.tight1}.
		
\medskip

\noindent\textbf{Proof  of \eqref{e.tight2}.} Appealing to the moment formula for the increment of QMF from \eqref{e:Qincmom} and the convergence from \eqref{e:add2} we have
		\begin{align}
  \label{e.idef}
  \lim_{N\to \infty}\Ex\bigg[\big(Q_t^N(\phi)-Q_s^N(\phi)\big)^k\bigg] = k!\left({\frac\sigma2}\right)^k \mathbf E_{B^{\otimes 2k}} \bigg[\int_{\Delta_k(s,t)}e^{{\frac\sigma2} \V_k(\vec{t})}\prod_{j=1}^k \phi(U_{t_j}^{2j})dL_0^{U^{2j-1}-U^{2j}}(t_j) \bigg],
		\end{align}
  where $\V_k(\vec{t})$ is defined in \eqref{def:v}. 
Lemma \ref{6.1} formalizes the intuition that ${\frac12}dL_0^{U^{2j-1}-U^{2j}}(t_j) = \delta_0(U_{t_j}^{2j-1}-U^{2j}_{t_j})$. Thus by appealing to that lemma, we can write the above expectation in terms of a certain family of concatenated bridge processes whose law we will write as $\mathbf P_c$. Specifically, we have that 
		\begin{align}
  \label{e.idef2}
			\mbox{r.h.s.~of \eqref{e.idef}} = k!\sigma^k\int_{\Delta_k(s,t)}  \mathbf E_{c} \bigg[e^{{\frac\sigma2} \V_k(\vec{t})}\prod_{j=1}^k \phi(U^{2j}_{t_j})\bigg]\prod_{j=1}^k p_{2t_j}(0)dt_1\cdots dt_k,
		\end{align}
where the expectation is taken over a collection of paths $U^j$ such that
\begin{itemize}
\setlength\itemsep{0.5em}
    \item $(U^{2i-1}-U^{2i}, U^{2i-1}+U^{2i})_{i=1}^k$ are $2k$ many independent processes.
    \item $U^{2i-1}+U^{2i}$ is a Brownian motion of diffusion rate $2$ for $i=1,2,\ldots,k$.
    \item $U^{2i-1}-U^{2i}$ is a Brownian bridge (from $0$ to $0$) of diffusion rate $2$ from $[0,t_i]$ and an independent Brownian motion of diffusion rate $2$ from $[t_i,\infty)$ for $i=1,2,\ldots,k$.
\end{itemize}
 We have used a different notation for the expectation operator $\mE_c$ in \eqref{e.idef2} ($c$ for concatenation) just to stress that the law is different from the standard Brownian motion. We claim that for all $p>1$ we have
\begin{align}
    \label{e.idef3}
     \sup_{\vec{t}\in [0,T]} \mathbf E_{c} \bigg[e^{{\frac\sigma2} \V_k(\vec{t})}\prod_{j=1}^k |\phi(U^{2j}_{t_j})|\bigg] \le \Con\cdot \|\phi\|_{L^p}^k.
\end{align}
where the $\Con>0$ depends on $p,k,\e,T$.  Let us assume \eqref{e.idef3} for the moment. By hypothesis, $t_j\ge \e >0$. Thus, $p_{2t_j}(0) \le \Con$ for some constant $\Con>0$ depending on $\e$. Thus, in view of \eqref{e.idef3}, to get an upper bound for the right-hand side of \eqref{e.idef2}, we may take the supremum of the integrand in the right-hand side of \eqref{e.idef2} and pull it outside of the integration. As the Lebesgue measure of $\Delta_k(s,t)$ is $\frac{(t-s)^k}{k!}$, we thus have the desired estimate in \eqref{e.tight2}.

\medskip

Let us now establish \eqref{e.idef3}.  Fix $p>1$ and take $q>1$ so that $p^{-1}+q^{-1}=1$. Use H\"older's inequality to write
		\begin{align*}
			\mathbf E_{c} \bigg[e^{{\frac\sigma2} \V_k(\vec{t})}\prod_{j=1}^k |\phi(U^{2j}_{t_j})|\bigg]  &\leq \mathbf E_{c} \bigg[e^{q{\frac\sigma2} \V_k(\vec{t})}\bigg]^{1/q}\mathbf E_c \bigg[\prod_{j=1}^k |\phi(U^{2j}_{t_j})|^p\bigg]^{1/p}.
		\end{align*}
For the first expectation above, observe that by Lemma \ref{bcts}, $\mathbf E_{c} \bigg[e^{q{\frac\sigma2} \V_k(\vec{t})}\bigg]$ is uniformly bounded over $\vec{t}\in \Delta_k(0,T)$. For the second expectation above, note that under $\mE_c$, we have $$U_{t_j}^{2j}=\tfrac12(U_{t_j}^{2j-1}+U_{t_j}^{2j})-\tfrac12(U_{t_j}^{2j-1}-U_{t_j}^{2j})=\tfrac12(U_{t_j}^{2j-1}+U_{t_j}^{2j})-0.$$ 
Thus under $\mE_c$, $U_{t_j}^{2j}$ are independent Gaussian random variables with variance $t_j/2$. Hence,
  \begin{align*}
      \mathbf E_c \bigg[\prod_{j=1}^k |\phi(U^{2j}_{t_j})|^p\bigg]^{1/p} = \prod_{j=1}^k\mathbf E_c \bigg[ |\phi(U^{2j}_{t_j})|^p\bigg]^{1/p} = \prod_{j=1}^k \bigg(\int_{\R} p_{t_j/2}(y)|\phi(y)|^p\,dy\bigg)^{1/p} \le \Con\cdot \|\phi\|_{L^p}^k.
  \end{align*}
where the last inequality follows by again using the fact that $t_j\ge \e>0$, together with the uniform bound $\sup_y p_t(y)\leq (2\pi t)^{-1/2}$, and noting that the constant $\Con$ is allowed to depend on $\e$.	This establishes \eqref{e.idef3} completing the proof of \eqref{e.tight2}.
	\end{proof}
	\section{Solving the martingale problem for the SHE}\label{sec:smp}
	
	The goal of this section is to establish the tightness of the field $\mathscr{X}^N$ defined in \eqref{a} and identify its limit points as the solution of the stochastic heat equation \eqref{she}. As the field $\mathscr{X}^N$ does not exist in a functional sense, we first introduce in Section \ref{sec.space} appropriate spaces and topologies that work well with the bounds derived in the previous section. In Section \ref{sec:tight} and Section \ref{sec:iden} we deal with the tightness and identification of the limit points of $\mathscr{X}^N$ respectively.

 \subsection{Weighted H\"older spaces and Schauder estimates} \label{sec.space}

 In this subsection, we introduce various natural topologies for our field $\mathscr{X}^N$ and its limit points. We then discuss how the heat flow affects these topologies and record Kolmogorov-type lemmas that will be key in showing tightness under these topologies. We begin by recalling many familiar and useful spaces of continuous and differentiable functions that have natural metric structures. 
For $d\ge 1$, we denote by $C_c^{\infty}(\R^d)$ the space of all compactly supported smooth functions on $\R^d$. For a smooth function on $\R^d$, we define its $C^r$ norm as
\begin{align}\label{crnorm}
    \|f\|_{C^r}:=\sum_{\vec{\alpha}} \sup_{\mathbf x\in \R^d} |D^{\vec\alpha} f(\mathbf x)|
\end{align}
 where the sum is over all $\vec{\alpha}\in \mathbb{Z}_{\ge 0}^d$ with $\sum \alpha_i\le r$ and $D^{\vec{\alpha}}:=\partial_{x_1}^{\alpha_1}\cdots \partial_{x_d}^{\alpha_d}$ denotes the mixed partial derivative.

\smallskip

We now recall the definition of weighted H\"older spaces from \cite[Definitions 2.2 and 2.3]{HL16}.  For the remainder of this paper, we shall work with \textit{elliptic and parabolic} weighted H\"older spaces with polynomial weight function
 \begin{align}
 \label{defwt}
     w(x):=(1+x^2)^\sig
 \end{align}
  for some fixed $\sig>1$. We introduce these weights because weighted spaces will be more convenient to obtain tightness estimates. Since the solution of \eqref{she} started from Dirac initial condition is known to be globally bounded away from $(0,0)$, we expect that it is possible to remove the weights throughout this section, but this would require more precise moment estimates than the ones we derived in previous sections, which take into account spatial decay of the fields.

	\begin{defn}[Elliptic H\"older spaces]\label{ehs}
		For $\alpha\in(0,1)$ we define the space $C^{\alpha,\sig}(\mathbb R)$ to be the completion of $C_c^\infty(\mathbb R)$ with respect to the norm given by $$\|f\|_{C^{\alpha,\sig}(\mathbb R)}:= \sup_{x\in\mathbb R} \frac{|f(x)|}{w(x)} + \sup_{|x-y|\leq 1} \frac{|f(x)-f(y)|}{w(x)|x-y|^{\alpha}}.$$
		For $\alpha<0$ we let $r=-\lfloor \alpha\rfloor$ and we define $C^{\alpha,\sig}(\mathbb R)$ to be the closure of $C_c^\infty(\mathbb R)$ with respect to the norm given by $$\|f\|_{C^{\alpha,\sig}(\mathbb R)}:= \sup_{x\in\mathbb R} \sup_{\lambda\in (0,1]} \sup_{\phi \in B_r} \frac{(f,S^\lambda_{x}\phi)_{L^2(\mathbb R)}}{w(x)\lambda^\alpha}$$ where the scaling operators $S^\lambda_{x}$ are defined by 
  \begin{align}\label{escale}
  S^\lambda_{x}\phi (y) = \lambda^{-1}\phi(\lambda^{-1}(x-y)),\end{align} 
  and where $B_r$ is the set of all smooth functions of {$C^r$ norm} less than 1 with support contained in the unit ball of $\mathbb R$.
	\end{defn}

 \begin{defn}[Function spaces]  Let $C^{\alpha,\sig}(\mathbb R)$ be as in Definition \ref{ehs}. We define $C([0,T],C^{\alpha,\sig}(\mathbb R))$ to be the space of continuous maps $g:[0,T]\to C^{\alpha,\sig}(\mathbb R),$ equipped with the norm $$\|g\|_{C([0,T],C^{\alpha,\sig}(\mathbb R))} := \sup_{t\in[0,T]} \|g(t)\|_{C^{\alpha,\sig}(\mathbb R)}.$$ 
\end{defn}

	Here and henceforth we will define $\Lambda_{[a,b]}:=[a,b]\times\mathbb R$ and we will define $\Lambda_T:=\Lambda_{[0,T]}.$

\smallskip

 So far we have used $\phi,\psi$ for test functions on $\R$. To make the distinction clear between test functions on $\R$ and $\R^2$, we shall use variant Greek letters such as $\varphi, \vartheta, \varrho$ for test functions on $\R^2$. In some instances, we will explicitly write $(f,\varphi)_{\mathbb R^2}$ or $(g,\phi)_{\mathbb R}$ when we want to be clear about the space in which we are pairing. 
 
	\begin{defn}[Parabolic H\"older spaces]\label{phs}
		We define $C_c^\infty(\Lambda_T)$ to be the set of functions on $\Lambda_T$ that are restrictions to $\Lambda_T$ of some function in $C_c^\infty(\mathbb R^2)$ (in particular we do not impose that elements of $C_c^\infty(\Lambda_T)$ vanish at the boundaries of $\Lambda_T).$ 
  
    For $\alpha\in(0,1)$ we define the space $C^{\alpha,\sig}_\mathfrak s(\Lambda_T)$ to be the completion of $C_c^\infty(\Lambda_T)$ with respect to the norm given by $$\|f\|_{C^{\alpha,\sig}_\mathfrak s(\Lambda_T)}:= \sup_{(t,x)\in \Lambda_T} \frac{|f(t,x)|}{w(x)} + \sup_{|s-t|^{1/2}+|x-y|\leq 1} \frac{|f(t,x)-f(s,y)|}{w(x)(|t-s|^{1/2}+|x-y|)^{\alpha}}.$$
		For $\alpha<0$ we let $r=-\lfloor \alpha\rfloor$ and we define $C^{\alpha,\sig}_\mathfrak s(\Lambda_T)$ to be the closure of $C_c^\infty(\Lambda_T)$ with respect to the norm given by $$\|f\|_{C^{\alpha,\sig}_\mathfrak s(\Lambda_T)}:= \sup_{(t,x)\in\Lambda_T} \sup_{\lambda\in (0,1]} \sup_{\varphi \in B_r} \frac{(f,S^\lambda_{(t,x)}\varphi)_{L^2(\Lambda_T)}}{w(x)\lambda^\alpha}$$ where the scaling operators are defined by 
		\begin{align}
			\label{scale}
			S^\lambda_{(t,x)}\varphi (s,y) = \lambda^{-3}\varphi(\lambda^{-2}(t-s),\lambda^{-1}(x-y)),
		\end{align}
		and where $B_r$ is the set of all smooth functions of $C^r$ norm less than 1 with support contained in the unit ball of $\mathbb R^2$.
	\end{defn}

 \begin{rk}[Derivatives of distributions] \label{d/dx}  By definition, any element $f\in C_\mathfrak s^{\alpha,\sig}(\Lambda_T)$ admits an $L^2$-pairing with any smooth function $\varphi: \Lambda_T\to \mathbb R$ of rapid decay, which we can write as $(f,\varphi)_{\Lambda_T}.$ Consequently there is a natural embedding $C_\mathfrak s^{\alpha,\sig}(\Lambda_T)\hookrightarrow \mathcal S'(\mathbb R^2)$ which is defined by formally setting $f$ to be zero outside of $[0,T]\times \mathbb R$. More rigorously, this means that the $L^2(\mathbb R^2)$-pairing of $f$ with any $\varphi \in \mathcal S(\mathbb R^2)$ is defined to be equal to $(f,\varphi|_{\Lambda_T})_{\Lambda_T}$.

 \smallskip
 
 The image of this embedding consists of some specific collection of tempered distributions that are necessarily supported on $[0,T]\times \mathbb R.$ Consequently we can sensibly define $\partial_tf$ and $\partial_xf$ as elements of $\mathcal S'(\mathbb R^2)$ whenever $f\in C_\mathfrak s^{\alpha,\sig}(\Lambda_T)$. Explicitly these derivatives are defined by the formulas$$(\partial_tf,\varphi)_{\Lambda_T} := -(f,\partial_t\varphi)_{\Lambda_T},\;\;\;\;\;\;\;\;\;\;(\partial_xf,\varphi)_{\Lambda_T} := -(f,\partial_x\varphi)_{\Lambda_T}$$for all smooth $\varphi:\Lambda_T\to\mathbb R$ of rapid decay. This convention on derivatives will be useful for certain computations later. From the definitions, when $\alpha<0$ one can check that for such $f$ one necessarily has $\partial_tf \in C_\mathfrak s^{\alpha-2,\sig}(\Lambda_T)$ and $\partial_x f \in C_\mathfrak s^{\alpha-1,\sig}(\Lambda_T).$

\smallskip

We remark that this fails for $\alpha>0$. Indeed, by our convention of derivatives, $\partial_tf$ may no longer be a smooth function (or even a function) even if $f\in C_c^\infty(\Lambda_T)$. This is because such an $f$ gets extended to all of $\mathbb R^2$ by setting it to be zero outside $\Lambda_T$. In particular, if $f$ does not vanish on the boundary of $\Lambda_T$, then it may become a discontinuous function under our convention of extension to $\mathbb R^2$. Due to these discontinuities, the distributional derivative $\partial_tf$ can be a tempered distribution with singular parts along the boundaries (one may verify that $\partial_t f$ can be at best an element of $C^{-2,\sig}_\mathfrak s(\Lambda_T)$ for generic $f\in C_c^\infty(\Lambda_T)$). In our later computations, we never take derivatives of functions in $C_\mathfrak s^{\alpha,\sig}(\Lambda_T)$ with $\alpha>0$. 
\end{rk}

 We now discuss the smoothing effect of heat flow on these weighted H\"older spaces.
 
	\begin{prop}[Smoothing effect on elliptic spaces]\label{p65} For $f\in C_c^\infty(\mathbb R)$ and $t>0$ define $$P_tf(x):= \int_\mathbb R p_{t}(x-y)f(y)dy.$$ Then for all $\alpha \le \beta <1$, there exists $\Con=\Con(\alpha,\beta,T)>0$ such that $$\|P_tf\|_{C^{\beta,\sig}(\mathbb R)} \leq \Con t^{-(\beta-\alpha)/2}\|f\|_{C^{\alpha,\sig}(\mathbb R)} $$ uniformly over $f\in C_c^\infty$ and $t\in [0,T]$. In particular, $P_t$ extends to a globally defined linear operator on $C^{\alpha,\sig}(\mathbb R)$ which maps boundedly into $C^{\beta,\sig}(\mathbb R).$
	\end{prop}
	
	A proof may be found in \cite[Lemma 2.8]{HL16} in the case of an exponential weight. {The proof for polynomial weights is identical}.

\smallskip

 For the parabolic H\"older spaces, the following lemma states that the heat flow improves the regularity by a factor of $2$ and provides a Schauder-type estimate.
	\begin{prop}[Schauder estimate]\label{sch}
		For $f\in C_c^\infty(\Lambda_T)$ let us define 
		\begin{align}
  \label{e:kf}
			Kf(t,x):= \int_{\Lambda_T} p_{t-s}(x-y)f(s,y)dsdy,
		\end{align} where $p_t$ is the standard heat kernel for $t>0,x\in\mathbb R$ and $p_t(x):=0$ for $t<0.$ Then for all $\alpha<-1$ with $\alpha\notin\mathbb Z$ there exists $\Con=\Con(\alpha)>0$ independent of $f$ such that $$\|Kf\|_{C^{\alpha+2,\sig}_\mathfrak s(\Lambda_T)}\leq \Con\cdot\|f\|_{C^{\alpha,\sig}_\mathfrak s(\Lambda_T)}.$$ In particular $K$ extends to a globally defined linear operator on $C^{\alpha,\sig}_\mathfrak s(\Lambda_T)$ that maps boundedly into $C^{\alpha+2,\sig}_\mathfrak s(\Lambda_T).$ Furthermore, if $f\in C^{\alpha,\sig}_\mathfrak s(\Lambda_{T})$, then $K(\partial_t-\frac12\partial_x^2)f= f$.
	\end{prop}

We remark that the last statement is only true because of our convention on distributional derivatives that we have explained in Remark \ref{d/dx}. Without that convention, that statement would be false even for a smooth function $f$ that does not vanish on the line $\{t=0\}$.

	\begin{proof} Let $\varphi$ be any smooth function that equals $1$ in a ball of radius $1$ about the origin in $\mathbb R^2$ and is supported on a ball of radius $2$ about the origin in $\mathbb R^2$.
  For $(t,x)\in\mathbb R^2$ let us define $P^+(t,x) := p_t(x)\varphi(t,x)$ and $P^-(t,x):=p_t(x)-P^+(t,x)$ so that $$p_t(x) = P^+(t,x)+P^-(t,x).$$ We have that $P^-$ is a globally smooth function on $\mathbb R^2$, and without any loss of generality, we may assume that it is supported on $\Lambda_{T+1}$ since values of the heat kernel outside $\Lambda_T$ do not matter when convolving with distributions supported on $\Lambda_T$. Then we get a corresponding decomposition $K=K^++K^-$ of the operator defined in the proposition statement.
		
		{Since $P^+$ is supported on a ball of radius 2, the proof of the weighted Schauder estimate for $K^+$ is sketched in \cite[Page 8, proof of Corollary 1.2]{HL16}, and we do not repeat it here. That proof is given in an elliptic setting, but the proof of the parabolic version is no different because the singularity of the 1+1 dimensional heat satisfies a bound of the form $$|\partial_t^{k_1}\partial_x^{k_2} p_t(x)| \leq \Con (|t|^{1/2}+|x|)^{-1-2k_1-k_2},$$} where $\Con$ may depend on $k_1,k_2 \geq 0$ but not on $t,x\in\mathbb R$.
		
		Thus it only remains to prove the weighted Schauder estimate for $K^-$. As $P^-$ is globally smooth on $\Lambda_T$ and has a Gaussian decay at infinity, we may write 
  $$P^-(t,x)=\sum_{n=1}^{\infty} \psi_n(t,x)$$
  where $\psi_n$ can be taken to be smooth functions supported on $[0,T+1]\times\big([n-1,n+1]\cup[-n-1,1-n]\big)$ and $\|\psi_n\|_{L^\infty}\leq \Con e^{-\frac1{\Con} n^2}$, for some constant $\Con>0$ independent of $n$. Consequently, if we write $K^-=\sum_n K_n^-$ where $K_n^-$ corresponds to convolution with $\psi_n$, then the infinite sum of the operator norms $\sum_n \|K_n\|_{C^{\alpha,\sig}_\mathfrak s\to C^{\alpha+2,\sig}_\mathfrak s}$ may be shown to be finite using the fact that Gaussian decay of the $\psi_n$ dominates the growth of the polynomial weights.

\smallskip

  Finally, it remains to prove $K(\partial_t-\frac12\partial_x^2)f= f$ for $f\in C^{\alpha,\sig}_\mathfrak s(\Lambda_{T})$ with $\alpha<0$. To prove this, we remark that $Kf$ may actually be defined for \textit{any} $f\in \mathcal S'(\mathbb R^2)$ with support contained in $[0,T]\times \mathbb R$, via the formula $$(Kf,\varphi) := (f,K^\dagger * \varphi),$$ where $K^\dagger(t,x) = \frac1{\sqrt{2\pi |t|}}{e^{-x^2/2|t|}}\ind_{\{t\le 0\}} $ and $*$ denotes convolution in both space and time. The right side is well-defined because $f$ is supported on $[0,T]\times \mathbb R$, and because one can ensure that if $\varphi\in\mathcal S(\mathbb R^2)$ then $K^\dagger * \varphi$ agrees on that strip with an element of $\mathcal S(\mathbb R^2)$ thanks to the Gaussian decay properties of $K^\dagger$ in the spatial variable (this is our reason for working on $\Lambda_T$ rather than all of $[0,\infty)\times\mathbb R$). 
  
  Then $Kf$ as defined above will be an element of $\mathcal S'(\mathbb R^2)$ that is supported on $[0,\infty)\times \mathbb R$, and this definition of $Kf$ agrees with the one in the proposition statement because it agrees on smooth test functions using integration by parts. Now we claim (more generally than stated in the proposition) that $K(\partial_t-\frac12\partial_x^2)f = f$ for \textit{any} $f\in \mathcal S'(\mathbb R^2)$ that is supported on $[0,T]\times \mathbb R$. But this just amounts to showing that $((\partial_t-\frac12\partial_x^2)f, K^\dagger * \varphi) = (f,\varphi)$ for all such $f$ and all $\varphi\in\mathcal S(\mathbb R^2)$. By definition of $(\partial_t-\frac12\partial_x^2)f$, this is equivalent to showing $(f,(-\partial_t-\frac12\partial_x^2)K^\dagger*\varphi) = (f,\varphi)$ for all such $f$ and all $\varphi\in\mathcal S(\mathbb R^2)$. As $f$ is supported on $[0,T]\times \mathbb R$, this reduces to showing that $(-\partial_t-\frac12\partial_x^2)K^\dagger*\varphi=\varphi$ for all $\varphi\in\mathcal S(\mathbb R^2)$ that are supported on $[-1,T+1]\times \mathbb R$ (for instance), which is clear. 
  \end{proof}

\begin{cor}\label{sme} Define $J: C^{\alpha,\sig}(\mathbb R) \to C_\mathfrak s^{\alpha,\sig}(\Lambda_T)$ by $$Jf(t,x):= \big(f,p_t(x-\cdot)\big)_\mathbb R,$$  $J$ is a bounded linear operator for any $\alpha <1$ and any $\sig>0.$  Consider the operator $\hat P_t:= JP_t$ with $P_t$ defined in Proposition \ref{p65}. By the semigroup property of the heat kernel, we have that 
 \begin{align}\label{e.hatjrel}
     \hat P_{s}f(t,x)=(Jf)(t+s,x).
 \end{align}
 Furthermore, the operators satisfy 
 \begin{align}
     \label{e.schhat}
     \|\hat P_tf\|_{C^{\beta,\sig}_\mathfrak s(\Lambda_T)} \leq \Con \cdot t^{-(\beta-\alpha)/2}\|f\|_{C^{\alpha,\sig}(\mathbb R)},
 \end{align}
where $\Con=\Con(\alpha,\beta,T)>0$ is independent of $t$ and $f$.
	\end{cor}
 \begin{proof} The first part follows by using the Schauder estimate in Proposition \ref{sch}, noting that $Jf = K(\delta_0\otimes f)$ where for $f \in C^{\alpha,\sig}(\mathbb R)$ the latter distribution is defined by $(\delta_0\otimes f, \varphi)_{\mathbb R^2}:= (f,\varphi(0,\cdot))_{\mathbb R}.$ Directly from the definitions one can check that $f\mapsto \delta_0\otimes f$ is bounded from $C^{\alpha,\sig}(\mathbb R)\to C_\mathfrak s^{\alpha-2,\sig}(\Lambda_T).$ Finally, \eqref{e.schhat} follows from Proposition \ref{p65}.     
 \end{proof}
 
	\begin{rk} 
 Although we have only defined it for $0<\alpha<1$, one may actually define $C_\mathfrak s^{\alpha,\sig}$ for all $\alpha >0$, as the closure of $C_c^\infty(\Lambda_T)$ with respect to $$\|f\|_{C^{\alpha,\sig}_\mathfrak s(\Lambda_T)}:= \sup_{(t,x)\in \Lambda_T} \frac{|f(t,x)|}{w(x)}+\sup_{(t,x)\in\Lambda_T} \sup_{\lambda\in (0,1]} \sup_{\varphi \in B_0} \frac{(f,S^\lambda_{(t,x)}\varphi)}{w(x)\lambda^\alpha}$$ where the scaling operators are defined in \eqref{scale} and where $B_0$ is the set of all smooth functions of $C^0$ norm less than 1, supported on the unit ball of $\mathbb R^2$ that are \textit{orthogonal in $L^2(\mathbb R^2)$ to all polynomials of parabolic degree less than or equal to} $\lceil \alpha \rceil$, see for instance equation (3.8) and (3.9) in \cite{CW17}. One may verify that this norm is equivalent to the one in Definition \ref{phs} for $\alpha \in (0,1)$. Stated in this way, the Schauder estimate actually holds for all $\alpha \in \mathbb R\backslash \mathbb Z,$ though we will not need this.
	\end{rk}

We next define a space-time distribution that is supported on a single temporal cross-section:

\begin{defn}\label{otimes} Let $\alpha<0$. Given some $f\in C^{\alpha,\sig}(\mathbb R)$ and $b \in [0,T]$ we define $\delta_b \otimes f\in C_\mathfrak s^{\alpha-2,\sig}(\Lambda_T)$ by the formula
$(\delta_b\otimes f, \varphi)_{\mathbb R^2}:= (f,\varphi(b,\cdot))_{\mathbb R}.$
 \end{defn}

 Directly from the definitions of the scaling operators in \eqref{escale} and \eqref{scale}, it is clear that for fixed $b\in [0,T]$, the linear map $f\mapsto \delta_b\otimes f$ is bounded from $C^{\alpha,\sig}(\mathbb R)\to C_\mathfrak s^{\alpha-2,\sig}(\Lambda_T)$ as long as $\alpha<0.$ It is also clear that $\delta_b \otimes f$ is necessarily supported on the line $\{b\}\times \mathbb R$. The following lemma shows that under mild conditions, $\delta_T\otimes f$ vanishes upon the action of the heat flow $K$ defined in \eqref{e:kf}.
 
 \begin{lem}\label{irrel} Let $f\in C^{\alpha,\tau}(\mathbb R),$ where $\alpha>-2.$ Then viewing $K$ as an operator from $C_\mathfrak s^{\alpha-2,\tau}(\Lambda_T)\to C_\mathfrak s^{\alpha,\tau}(\Lambda_T)$, we have that
$K(\delta_T \otimes f) = 0.$
 \end{lem}

 \begin{proof}Note that $\delta_T \otimes f \in C_\mathfrak s^{\alpha,\tau}(\Lambda_T)$ where $\alpha>-4.$ Thus by Proposition \ref{sch}, we know that $K(\delta_T \otimes f) \in C_\mathfrak s^{\beta,\tau}(\Lambda_T)$ for some $\beta>-2$. Now since $\delta_T \otimes f$ is supported on $\{T\}\times \mathbb R$, it immediately follows that $K(\delta_T \otimes f)$ is also supported on $\{T\}\times \mathbb R.$

Thus it suffices to show that if $g\in C_\mathfrak s^{\beta,\tau}(\Lambda_T)$ for some $\beta\in(-2,0),$ and if $g$ is supported on $\{T\}\times \mathbb R$ then $g=0.$ We can prove this directly: fix some smooth nonnegative compactly supported function $\varphi$ on $\mathbb R^2$, and define $g^\lambda(t,x):= (g,S^\lambda_{(t,x)}\varphi),$ where the scaling operators are given by \eqref{scale}. Using $g\in C_\mathfrak s^{\beta,\tau}(\Lambda_T)$ for some $\beta\in(-2,0),$ one has that $g^\lambda(t,x) \leq Cw(x) |T-t|^{\frac{\beta}2}$ where $C$ is independent of $\lambda,x,t,$ and where $w$ is the weight function given by \eqref{defwt}. Now if $\psi$ is any compactly supported function on $\mathbb R^2$, then by integrability of $t\mapsto t^{\beta/2}$ on $[0,1]$ and the dominated convergence theorem, we find that $$(g,\psi) = \lim_{\lambda\to 0} \int_{\mathbb R^2} g^\lambda(t,x) \psi(t,x) dtdx = \int_{\mathbb R^2}  \lim_{\lambda\to 0}  g^\lambda(t,x) \psi(t,x) dtdx=0,$$ where we used the fact that $\lim_{\lambda\to 0}g^\lambda(t,x) = 0$ at all points $(t,x)$ outside of the measure-zero subset given by $\{T\}\times\mathbb R.$
 \end{proof}

 We remark that the last paragraph of the above proof has an intuitive analogue in the elliptic setting as well: if $f\in C^{\alpha,\sig}(\mathbb R)$ where $\alpha>-1$, and if $f$ is supported on a single point, then $f=0$. The proof proceeds in a similar manner.

 \medskip

 We end this subsection by recording a Kolmogorov-type lemma for the three spaces introduced at the beginning of this subsection. It will be crucial in proving tightness in those respective spaces.
	
	\begin{lem}[Kolmogorov lemma] \label{l:KC} Let $L^2(\Omega,\mathcal{F},\Pr)$ be the space of all random variables defined on a probability space  $(\Omega,\mathcal{F},\Pr)$ with finite second moment. We have  the following:
 \begin{enumerate}[label=(\alph*),leftmargin=15pt]
 \setlength\itemsep{0.5em}

 \item \label{KC1} (Elliptic H\"older Space)  Let $\phi \mapsto V(\phi)$ be a bounded linear map from $\mathcal S(\mathbb R)$ into $L^2(\Omega,\mathcal F,\mathbb P)$. Recall $S^\lambda_{x}$ from \eqref{escale}. Assume there exists some $p>1$ and $\alpha<0$ and $\Con=\Con(\alpha,p)>0$ such that one has $$\mE[|V(S^\lambda_{x}\phi)|^p]\leq \Con\lambda^{\alpha p},$$ uniformly over all smooth functions $\phi$ on $\mathbb R$ supported on the unit ball of $\mathbb R$ with $\|\phi\|_{L^\infty}\leq 1$, and uniformly over $\lambda\in(0,1]$ and $x\in\mathbb R$. Then for any $\sig>1$ and any $\beta<\alpha-1/p$ there exists a random variable $\mathscr V$ taking values in $C^{\beta,\sig}(\mathbb R)$ such that $(\mathscr V,\phi)=V(\phi)$ almost surely for all $\phi.$ Furthermore one has that $$\mE[\|\mathscr V\|^p_{C^{\beta,\sig}(\mathbb R)}]\leq \Con',$$ where $\Con'$ depends on the choice of $\alpha,p,$ and the constant $\Con$ appearing in the moment bound above but not on $V,\Omega,\mathcal F,\mathbb P$.

 \item \label{KC2} (Function space) 
		Let $(t,\phi) \mapsto V(t,\phi)$ be a map from $[0,T]\times \mathcal S(\mathbb R)$ into $L^2(\Omega,\mathcal F,\mathbb P)$ which is linear and continuous in $\phi$. Fix a non-negative integer $r$. Assume there exists some $\kappa>0, p>1/\kappa$ and $\alpha<-r$ and $\Con=\Con(\kappa,\alpha,p,{T})>0$ such that one has 
		\begin{align*}\mE[|V(t,S^\lambda_{x}\phi)|^p]&\leq \Con\lambda^{\alpha p},\\  \mE[|V(t,S^\lambda_{x}\phi)-V(s,S^\lambda_{x}\phi)|^p]&\leq \Con\lambda^{(\alpha -\kappa) p} |t-s|^{\kappa p},
		\end{align*}uniformly over all smooth functions $\phi$ on $\mathbb R$ supported on the unit ball of $\mathbb R$ with {$\|\phi\|_{C^r}\leq 1$}, and uniformly over $\lambda\in(0,1]$ and $0\leq s,t\leq T$. Then for any $\sig>1$ and any $\beta<\alpha-\kappa$ there exists a random variable $\big(\mathscr V(t)\big)_{t\in[0,T]}$ taking values in $C([0,T],C^{\beta,\sig}(\mathbb R))$ such that $(\mathscr V(t),\phi)=V(t,\phi)$ almost surely for all $\phi$ and $t$. Furthermore, one has that $$\mE[\|\mathscr V\|^p_{C([0,T],C^{\beta,\sig}(\mathbb R))}]\leq \Con',$$ where $\Con'$ depends on the choice of $\alpha,\beta,p,\kappa,$ and the constant $\Con$ appearing in the moment bound above but not on $V,\Omega,\mathcal F,\mathbb P$.
 
     \item \label{KC}
		(Parabolic H\"older Space) Let $\varphi \mapsto V(\varphi)$ be a bounded linear map from  $\mathcal S(\mathbb R^2)$ to $L^2(\Omega,\mathcal F,\mathbb P)$. Assume $V(\varphi)=0$ for all $\varphi$ with support contained in the complement of $\Lambda_T$. Recall $S^\lambda_{(t,x)}$ from  \eqref{scale}. Fix a non-negative integer $r$. Assume there exists some $p>1$ and $\alpha <0$ and $\Con=\Con(\alpha,p)>0$ such that one has $$\mE[|V(S^\lambda_{(t,x)}\varphi)|^p]\leq \Con\lambda^{\alpha p},$$ uniformly over all smooth functions $\varphi$ on $\mathbb R^2$ supported on the unit ball of $\mathbb R^2$ with {$\|\varphi\|_{C^r}\leq 1$}, and uniformly over $\lambda\in(0,1]$ and $(t,x)\in\Lambda_T$. Then for any $\sig>1$ and any $\beta<\alpha-3/p$ there exists a random variable $\mathscr V$ taking values in $C^{\beta,\sig}_\mathfrak s(\Lambda_T)$ such that $(\mathscr V,\varphi)=V(\varphi)$ almost surely for all $\varphi.$ Furthermore one has that $$\mE[\|\mathscr V\|^p_{C^{\beta,\sig}_\mathfrak s(\Lambda_T)}]\leq \Con',$$ where $\Con'$ depends on the choice of $\alpha,p,$ and the constant $\Con$ appearing in the moment bound above but not on $V,\Omega,\mathcal F,\mathbb P$.

 \end{enumerate}

	\end{lem}
	{A proof of the above results may be adapted from the proof of Lemma 9 in Section 5 of \cite{WM}.} We remark that we do not actually need uniformity over a large class of test functions as we have written above, just a single well-chosen test function would suffice (e.g. the Littlewood-Paley blocks as used in \cite{WM} or the Daubechies wavelets in \cite{HL16}). Part (a) of the above lemma can be stated with $C^r$ functions just like part (c), but we will not need this general version.
 
	\subsection{Tightness}\label{sec:tight} In this section, we prove tightness of the field $\mathscr{X}^N$ and establish regularity estimates for the limit points. {At this stage, readers may find it helpful to revisit Section \ref{sec1.4.2} for a proof sketch of the tightness argument.}

\begin{prop}\label{xtight}
		For each $t\ge 0$, the fields $\{\mathscr X_t^N\}_{N\ge 1}$ defined in \eqref{a} may be realized as random elements of the elliptic H\"older space $C^{\alpha,\sig}(\mathbb R)$ for any $\sig>1$ and any $\alpha<-1.$ Moreover they are tight with respect to that topology as $N\to\infty$, in fact, we have that $\sup_N \Ex[ \|\mathscr X_t^N\|_{C^{-\alpha,\sig}(\mathbb R)}^p]<\infty$ for all $p\ge 1$. Furthermore, the $\mathscr X^N$ defined in \eqref{b} themselves may be viewed as elements of the parabolic H\"older space $C_\mathfrak s^{\alpha,\sig}(\Lambda_T)$ for $\alpha<-1$ and $\tau>1.$ Moreover they are tight with respect to that topology as $N\to\infty$.
	\end{prop}
	
	\begin{proof} Using Lemma \ref{rnd} and Lemma \ref{dn} we obtain that 
		\begin{align*}
			\Ex[\mathscr X_t^N(\phi)^{2k}] &= \mathbf E_{D_{N^{1/4}}SB_{N^{1/2}\nu}^{(k)}}\bigg[\prod_{j=1}^{2k} \phi(X^j_t) e^{{\frac12}N^{1/2}\sum_{i<j} V_t^{ij}(\mathbf{X})}\bigg] \\ &=   \mathbf E_{SB_{N^{1/2}\nu}^{(k)}}\bigg[\mathcal M_N(t)\prod_{j=1}^{2k}\phi(X^j_t) e^{{\frac12}N^{1/2}\sum_{i<j} V_t^{ij}(\mathbf{X})}\bigg] \\ &\leq \|\phi\|^{2k}_{L^\infty(\mathbb R)} \mathbf E_{SB_{N^{1/2}\nu}^{(k)}}[\mathcal M_N(t)^2]^{1/2}\mathbf E_{SB_{N^{1/2}\nu}^{(k)}}\bigg[e^{N^{1/2}\sum_{i<j} V_t^{ij}(\mathbf{X})}\bigg]^{1/2},
		\end{align*}
	where $\mathcal{M}_N(t)$ is defined in \eqref{mnt} and the last inequality above is due to the Cauchy-Schwarz inequality. From Remark \ref{rk:mn} and Lemma \ref{traps} (with $\nu\mapsto N^{1/2}\nu$), we know that the above two expectations are uniformly bounded in $N$. Therefore the last expression is bounded by some universal constant times $\|\phi\|_{L^\infty}^{2k}.$ In particular we get that $$ \Ex[\mathscr X_t^N(S_x^\lambda\phi)^{2k}]\leq \Con \|S_x^\lambda\phi\|_{L^\infty}^{2k}\leq C\lambda^{-2k},$$ uniformly over all $\phi\in C_c^\infty(\mathbb R)$ with $\|\phi\|_{L^\infty}\leq 1,$ where $S_x^\lambda$ is defined in \eqref{escale}. Appealing to Lemma \ref{l:KC} \ref{KC1} we get the desired result for fixed $t$.

    To prove the second statement about tightness in the parabolic H\"older space, we will as in \eqref{b} denote space-time pairings as $(\mathscr X^N,\varphi)_{L^2(\mathbb R^2)}$. Note by Minkowski's inequality that 
    \begin{align*}\mathbb E[ (\mathscr X^N,\varphi)_{L^2(\mathbb R^2)}^{2k} ]^{1/2k} &= \mathbb E\bigg[ \int_\mathbb R \bigg(\mathscr X_s^N(\varphi(s,\cdot))ds\bigg)^{2k}\bigg]^{1/2k} \leq \int_\mathbb R \mathbb E[\mathscr X_s^N(\varphi(s,\cdot))^{2k}]^{1/2k} ds.
    \end{align*}
    Now with $S^\lambda_{(t,x)}$ defined in \eqref{scale}, we have 
    \begin{align*}\mathbb E[ (\mathscr X^N,S^\lambda_{(t,x)}\varphi)_{L^2(\mathbb R^2)}^{2k} ]^{1/2k} &\leq \int_\mathbb R \mathbb E\big[\mathscr X_s^N\big( (S^\lambda_{(t,x)}\varphi)(s,\cdot)\big)^{2k}\big]^{1/2k} ds  \leq \Con \int_{t-\lambda^2}^{t+\lambda^2} \lambda^{-3} ds \leq \Con \lambda^{-1},
    \end{align*}
    uniformly over all $\varphi\in C_c^\infty(\mathbb R^2)$ supported on the unit ball, with $\|\varphi\|_{L^\infty}\leq 1.$ Here the $\lambda^{-3}$ comes from the fact that the parabolic scaling operator differs from the elliptic one by exactly a factor of $\lambda^{-2}$. In the second bound, we used the same bounds appearing in the proof for the elliptic case above, noting that the constants there are independent of $t\in [0,T]$. Appealing to Lemma \ref{l:KC} \ref{KC} we get the desired result.
	\end{proof}

 \begin{prop}\label{mcts}
		The fields $M^N$ defined in \eqref{mart_ob} may be realized as an element of $C([0,T],C^{\alpha,\sig}(\mathbb R))$ for any $\alpha<-2$ and $\sig>1$. Moreover, they are tight with respect to that topology.  $Q^N$ defined in \eqref{QVfield} may be realized as an element of $C([0,T],C^{\gamma,\sig}(\mathbb R))$ for any $\gamma<-1$ and $\sig>1$. Moreover, they are tight with respect to that topology. Let $(M^\infty,Q^\infty)$ be a joint limit point of $(M^N,Q^N)$ in $C([0,T],C^{\alpha,\sig}(\mathbb R)\times C^{\gamma,\sig}(\mathbb R)).$ For all $\phi\in C_c^\infty(\mathbb R)$ the process $(M_t^\infty(\phi))_{t\in[0,T]}$ is a martingale with respect to the canonical filtration on that space, and moreover one has 
  \begin{equation}
      \label{e:mcts}
      \langle M^\infty(\phi)\rangle_t = Q_t^\infty(\phi^2).
  \end{equation}
	\end{prop}

\medskip
 
We remark that the canonical filtration $\mathscr F_t$ on $C([0,T],\Xi)$ for any Polish space $\Xi$ is the one generated by the random variables $X(s)$ for $s\leq t$ as $X$ varies through all elements of $C([0,T],\Xi)$.
	
	\begin{proof}Take any $\phi\in C_c^\infty(\mathbb R)$ with $\|\phi\|_{L^\infty}\leq 1.$ Recall $S_{(t,x)}^{\lambda}$ from \eqref{scale}. Using the first bound in Proposition \ref{tight1}, we have $$\Ex[ (Q_t^N(S^\lambda_{x}\phi)-Q_s^N(S^\lambda_{x}\phi))^{k}]\leq \Con |t-s|^{k/4}\lambda^{-k},$$ uniformly over $x\in\mathbb R, \lambda\in (0,1], 0\leq s,t\leq T$. Since $Q_0^N(\phi)=0$ by definition, therefore the assumptions of Lemma \ref{l:KC} \ref{KC2} are satisfied for any $\kappa\leq 1/4$, any $p>1/\kappa$, and any $\alpha\leq -1$, and we conclude the desired tightness for $Q^N$.

  \smallskip
  
		Now we address the tightness of the $M^N$. By Lemma \ref{qv} and the Burkholder-Davis-Gundy inequality, we have that $$\Ex[ (M_t^N(\phi)-M_s^N(\phi))^{2k}] \leq \Con \cdot \Ex\bigg[ \bigg((Q_t^N\big((\phi+N^{-1/4}\phi')^2\big)-Q_s^N\big((\phi+N^{-1/4}\phi')^2\big)\bigg)^{k}\bigg],$$ where $\Con=\Con(k)>0$ is free of $\phi,s,t,N.$ Using a crude bound $$(\phi+N^{-1/4}\phi')^2 \leq 2\phi^2 +2N^{-1/2} (\phi')^2 \leq 2\phi^2 + 2 (\phi')^2,$$ we find from the first bound in Proposition \ref{tight1} that 
		\begin{equation}\label{m1}\Ex[ (M_t^N(\phi)-M_s^N(\phi))^{2k}] \leq \Con(t-s)^{k/4} (\|\phi\|_{L^\infty}^2 + \|\phi'\|_{L^\infty}^2)^k \leq \Con(t-s)^{k/4}\|\phi\|_{C^1}^{2k}.
		\end{equation}
		This gives $$\Ex[ (M_t^N(S^\lambda_{x}\phi)-M_s^N(S^\lambda_{x}\phi))^{2k}] \leq \Con(t-s)^{k/4} \lambda^{-4k} $$
		uniformly over $x\in\mathbb R, \lambda\in (0,1], 0\leq s,t\leq T$, and $\phi\in C_c^\infty(\mathbb R)$ with $\|\phi\|_{C^1}\leq 1.$ Moreover $M_0(\phi)=0$ by definition, therefore the assumptions of Lemma \ref{l:KC} \ref{KC2} are satisfied for any $\kappa\leq 1/8$, any $p>1/\kappa$, and any $\alpha\leq -2$. Hence, we conclude the desired tightness for $M^N$.

  \smallskip
  
	We next show that the limit point $M^{\infty}(\phi)$ is a martingale.	Since $M_0^N(\phi)=0$, from \eqref{m1}, we see that $\sup_N \Ex[M_t^N(\phi)^{2k}]<\infty.$ Thus by uniform integrability, $M^{\infty}(\phi)$ is a martingale. In the prelimit we know that $$(M_t^N(\phi))^2 - Q_t^N\big((\phi+N^{-1/4}\phi')^2\big)=(M_t^N(\phi))^2-Q_t^N(\phi^2) - \mathcal E_t^N(\phi)$$ is a martingale, where the error term $\mathcal E_t^N$ is defined in \eqref{etn}. By Proposition \ref{4.2} we know that $\mathcal E_t^N$ vanishes in probability, so we conclude (again by uniform $L^p$ boundedness guaranteed by Proposition \ref{tight1}) that $(M_t^\infty(\phi))^2-Q_t^\infty(\phi^2)$
	is a martingale. This verifies \eqref{e:mcts} completing the proof.
	\end{proof}

 \begin{prop}\label{ds}
For $\alpha<0$, the derivative operator $\partial_s : C([0,T],C^{\alpha,\sig}(\mathbb R)) \to C_\mathfrak s^{\alpha-2,\sig}(\Lambda_T)$ which is defined by sending $(v_t)_{t\in[0,T]}$ to the distribution $$(\partial_s v,\varphi)_{L^2(\Lambda_T)}:=v_T(\varphi(T,\cdot)) -\int_0^T v_t(\partial_t\varphi(t,\cdot))dt -v_0(\varphi(0,\cdot)),$$ whenever $\varphi\in C_c^\infty(\Lambda_T),$ is a bounded linear map. Let $M^N$ be as in Proposition \ref{mcts}, viewed as elements of $C([0,T],C^{\alpha,\sig}(\mathbb R))$ for some $\alpha<-2.$ Also let $\mathscr X^N_t$ and $\mathscr X^N$ and be as in Proposition \ref{xtight}, viewed as elements of $ C^{\gamma,\sig}(\mathbb R)$ and $ C_\mathfrak s^{\gamma,\sig}(\Lambda_T)$ respectively for some $\gamma<-1$. Then one has that 
\begin{align}
    \label{e:delMrel}
    \partial_s M^N = -\delta_{(0,0)}+(\partial_t - \tfrac12\partial_x^2)\mathscr X^N+\delta_T\otimes \mathscr X^N_T
\end{align}
in the sense of distributions.
 \end{prop}

 \begin{proof}Recall $S^\lambda_{(t,x)}$ from \eqref{scale}. From the definition of the space $C^{\alpha,\sig}(\mathbb R)$ one verifies directly that if $v=(v_t)_{t\in [0,T]}\in C([0,T],C^{\alpha,\sig}(\mathbb R))$ then $$\int_0^T |v_t\big(\partial_t (S^\lambda_{(s,y)} \varphi)(t,\cdot)\big)|dt \leq w(y) \int_{s-\lambda^2}^{s+\lambda^2} \lambda^{\alpha-4}dt \leq w(y)\lambda^{\alpha-2}$$ uniformly over $(s,y)\in \Lambda_T$ and $\varphi\in C_c^\infty(\mathbb R^2)$ supported on the unit ball of $\mathbb R^2$ with $\|\varphi\|_{C^r}\leq 1$ where $r=\lceil -\alpha \rceil.$ This proves the first part of the proposition that $\partial_s$ is bounded from $C([0,T],C^{\alpha,\sig}(\mathbb R)) \to C_\mathfrak s^{\alpha-2,\sig}(\Lambda_T).$ 

 \smallskip

For the second part, let us abbreviate $\mathscr M^N:=\partial_s M^N$. We first note by definition of $\partial_s$ that for any smooth $\varphi$ of compact support on $\mathbb R^2$, $\mathscr M^N(\varphi)$ is explicitly given by 
\begin{align*}
			\mathscr M^N(\varphi)=M_T^N(\varphi(T,\cdot))-\int_0^{T} \hspace{-0.1cm}M^N_s\big(\partial_s\varphi(s,\cdot)\big)ds,
\end{align*}
where we used the fact that $M_0^N\equiv 0$. Note that the above expression is the same as $R_T^N(\varphi)$ defined in \eqref{tmart}.  Now observe that using \eqref{mart_ob} we have
 \begin{equation*}
      \begin{aligned}
			 \int_0^{T} M^N_s\big(\partial_s\varphi(s,\cdot)\big)ds & = \int_0^{T} \hspace{-0.2cm}\mathscr X^N_s\big(\partial_s\varphi(s,\cdot)\big)ds -\int_0^{T} \big(\partial_s\varphi(s,0)\big)ds \\ & \hspace{3cm}-\int_0^{T} \frac12\int_0^s \hspace{-0.2cm} \mathscr{X}_u^N\big(\partial_{xx}\partial_s\varphi(s,\cdot)\big)du ds \\ & = \int_0^{T} \mathscr X^N_s\big(\partial_s\varphi(s,\cdot)\big)ds +\varphi(0,0)-\varphi(T,0) \\ & \hspace{3cm} -\int_0^{T} \frac12\int_0^s \mathscr{X}_u^N\big(\partial_{xx}\partial_s\varphi(s,\cdot)\big)du ds. 
		\end{aligned}
 \end{equation*}
 The last term on the right-hand side of the above equation can be simplified as follows.
	\begin{align*}
	\int_0^{T} \frac12\int_0^s \mathscr{X}_u^N\big(\partial_{xx}\partial_s\varphi(s,\cdot)\big)du ds&=\int_0^{T} \frac12\int_u^{T} \mathscr{X}_u^N\big(\partial_{xx}\partial_s\varphi(s,\cdot)\big)ds du\\&=\frac12\int_0^{T} \big[\mathscr{X}_u^N\big(\partial_{xx}\varphi(T,\cdot)\big)-\mathscr{X}_u^N\big(\partial_{xx}\varphi(u,\cdot)\big) \big]du\\ &= -M_T^N(\varphi(T,\cdot)) + \mathscr X_T^N(\varphi(T,\cdot)) -  \varphi(T,0) \\ & \qquad -\frac12\int_0^T \mathscr X_u^N(\partial_{xx}\varphi(u,\cdot)\big)du.
	\end{align*}	 
 We applied \eqref{mart_ob} with $\phi = \varphi(T,\cdot)$ in the last line. Combining these expressions, we derive that
 \begin{align*}
     \mathscr{M}^N(\varphi)=-\varphi(0,0)-\int_0^T \mc{X}_s^N(\partial_s\varphi(s,\cdot)ds-\frac12\int_0^T \mathscr X_s^N(\partial_{xx}\varphi(s,\cdot)\big)ds + \mathscr X_T^N(\varphi(T,\cdot)). 
 \end{align*}
 This verifies \eqref{e:delMrel} completing the proof.
 \end{proof}

Since we expect to obtain a Dirac initial data in the limiting SPDE, there is an additional singularity at the origin that we have not yet taken into account. To fix this issue, we now formulate a tightness result taking into account this singularity, by starting our martingale-forced heat equation from some positive time $\e$ (which should be thought of as being close to $0$).

	\begin{prop}[Tightness]\label{tightm2}Fix $\e>0,$ and consider the fields $\mathscr X^{N,\e}(t,\cdot):=\mathscr X_{t+\e}^N(\cdot),$ for $t\ge 0$. Set $\mathscr X^{N,\e}(t,\cdot):=0$ for $t<0$. The fields $\mathscr X^{N,\e}$ may be realized as random variables taking values in $C_\mathfrak s^{\alpha,\sig}(\Lambda_T)$ for any $\sig>1$ and $\alpha<-1$. Furthermore, they are tight with respect to that topology, and any limit point necessarily lies in $C_\mathfrak s^{-\kappa+1/2,\sig}(\Lambda_T)$ for any $\kappa>0$.
	\end{prop}

	\begin{proof}
		
  The proof that the fields $\mathscr X^{N,\e}$ are tight in $C_\mathfrak s^{\alpha,\sig}(\Lambda_T)$ for any $\sig>1$ and $\alpha<-1$ follows by using the exact same strategy followed in the proof of Proposition \ref{xtight}. Let us set $$\mathscr M^{N,\e}:= \big(\partial_t-\frac12\partial_x^2\big)\mathscr X^{N,\e}+(\delta_T \otimes \mathscr X^N_{T+\e} - \delta_0 \otimes \mathscr X^N_\e).$$ 
  We first claim that for fixed $\e>0$, the fields $\{\mathscr M^{N,\e}\}_{N>1}$ are tight in $C_\mathfrak s^{-3-\kappa,\sig}(\Lambda_T)$ for any $\sig>1$ and $\kappa>0$. As $\mathscr{X}^{N,\e}$ is tight in $C_\mathfrak s^{-1-\kappa,\sig}(\Lambda_T)$, the tightness of the $(\partial_t - \frac12\partial_x^2)\mathscr X^{N,\e}$ in $C_\mathfrak s^{-3-\kappa,\sig}(\Lambda_T)$ is immediate since $\partial_t - \frac12\partial_x^2$ is a bounded linear map from $C_\mathfrak s^{\alpha,\sig}(\Lambda_T)$ into $C_\mathfrak s^{\alpha-2,\sig}(\Lambda_T)$ as explained in Remark \ref{d/dx}. From Proposition \ref{xtight} we know that both $\mathscr X^N_{T+\e} $ and $\mathscr X^N_{\e} $ are tight in $C^{-1-\kappa,\sig}(\mathbb R)$. As we remarked after Definition \ref{otimes}, the map $f\mapsto \delta_a\otimes f$ is continuous from $C^{-1-\kappa,\sig}(\mathbb R) \to C_\mathfrak s^{-3-\kappa,\sig}(\Lambda_T),$ therefore we can immediately conclude that $(\delta_T \otimes \mathscr X^N_{T+\e} - \delta_0 \otimes \mathscr X^N_\e)$ are tight in $C_\mathfrak s^{-3-\kappa,\sig}(\Lambda_T)$. Hence, the sum given by $\mathscr M^{N,\e}$ is also tight in the latter space as $N\to \infty$ (for fixed $\e$).

  Next, we address the regularity of the limit points of $\mathscr X^{N,\e}.$ By the last statement in Proposition \ref{sch}, we can apply the linear operator $K$ to both sides to obtain $$\mathscr X^{N,\e} = K(\delta_0\otimes \mathscr X^N_\e) + K\mathscr M^{N,\e} - K(\delta_T \otimes \mathscr X^N_{T+\e}).$$ By Proposition \ref{xtight} we know that $\mathscr X^N_{T+\e}$ is tight in $C^{\gamma,\sig}(\mathbb R)$ for $\gamma<-1$, and from Lemma \ref{irrel} we can conclude that $K(\delta_T \otimes \mathscr X^N_{T+\e})=0$. On the other hand, $K(\delta_0\otimes \mathscr X^N_\e) = J\mathscr X^N_\e$, where $J$ is defined in Corollary \ref{sme}. We thus have the Duhamel equation $$\mathscr X^{N,\e} = J \mathscr X^N_\e + K\mathscr M^{N,\e}.$$ 
For technical reasons that will be made clear below, we now replace $\e$ by $\e/2$ and $T$ by $T+1$. If we take any joint limit point $(\mathscr X^{\infty,\e/2},\mathscr M^{\infty,\e/2},h)$ in $C_\mathfrak s^{\alpha,\sig}(\Lambda_{T+1})\times C_\mathfrak s^{\alpha+2,\sig}(\Lambda_{T+1})\times C^{\gamma,\sig}(\mathbb R)$ as $N\to\infty$ of $(\mathscr X^{N,\e/2},\mathscr M^{N,\e/2},\mathscr X^N_{\e/2})$, then (by continuity of $J$ and $K$ on the relevant spaces) the Duhamel relation still holds in the limit, i.e., \begin{align}
    \label{e.xie2}
    \mathscr X^{\infty,\e/2}= Jh + K\mathscr M^{\infty,\e/2}.
\end{align}
We now claim that 
\begin{align}
\label{e.tm2}
    \Ex[\|\mathscr M^{\infty,\e/2}\|_{C_\mathfrak s^{-\kappa-3/2,\sig}(\Lambda_{T+1})}^p]<\infty.
\end{align} 

Let us now complete the proof of regularity assuming \eqref{e.tm2}.
\begin{itemize}[leftmargin=20pt]
\itemsep\setlength{0.5em}
    \item Using \eqref{e.tm2} and Proposition \ref{sch}, it follows that 
$$\mathbb E\bigg[\big\|K\mathscr M^{\infty,\e/2}\big\|_{C_\mathfrak s^{-\kappa+1/2,\sig}(\Lambda_{T+1})}^p\bigg]<\infty.$$ This implies that the restriction of $K\mathscr M^{\infty,\e/2}$ to $[\e/2,T+\e/2]\times\mathbb R$ lies in $C_\mathfrak s^{-\kappa+1/2,\sig}(\Lambda_{[\e/2,T+\e/2]})$.

\item By Proposition \ref{xtight} we have $\Ex[\|h\|_{C^{-1-\kappa,\sig}}^p]<\infty$. So from \eqref{e.schhat} with $\alpha:=-1-\kappa$ and $\beta := \frac12-\kappa$. It follows that $$\Ex[\|\hat P_{\e/2}h\|_{C_\mathfrak s^{-\kappa+1/2,\sig}(\Lambda_{T+1})}^p]<\infty.$$ 
Since from \eqref{e.hatjrel}, we know $Jh(t+\e/2,x) = \hat P_{\e/2} h(t,x)$, we thus have that the restriction of $Jh$ to $[\e/2,T+\e/2]\times\mathbb R$ lies in $C_\mathfrak s^{-\kappa+1/2,\sig}(\Lambda_{[\e/2,T+\e/2]})$.
\end{itemize}
Thanks to the above two bullet points and the relation \eqref{e.xie2}, we have showed that the restriction of $\mathscr X^{\infty,\e/2}$ to $[\e/2,T+\e/2]\times\mathbb R$ lies in $C_\mathfrak s^{-\kappa+1/2,\sig}(\Lambda_{[\e/2,T+\e/2]})$. This is equivalent to the fact that $\mathscr X^{\infty,\e}$ lies in $C_\mathfrak s^{-\kappa+1/2,\sig}(\Lambda_T)$. This completes the proof modulo \eqref{e.tm2}. 

\medskip

\noindent\textbf{Proof of \eqref{e.tm2}}. We shall show \eqref{e.tm2} with $\e/2$ and $T+1$ replaced by $\e$ and $T$ respectively. Following the same proof as in Lemma \ref{ds}, one can check that $$\mathscr M^{N,\e}(\varphi) = R^N_{T+\e}(\varphi) - R^N_\e(\varphi),$$ where $R^N$ are the same martingales defined in \eqref{tmart}. We will denote the right side as $R^{N,\e}(\varphi)$ and the quadratic variation up to time $T$ as $\langle R^{N,\e}(\varphi)\rangle_T.$ Let us fix any $\alpha<-3$ and take any joint limit point $(\mathscr M^{\infty,\e},\mathscr X^{\infty,\e})$ in $C_\mathfrak s^{\alpha, \sig}(\Lambda_T) \times C_\mathfrak s^{\alpha+2,\sig}(\Lambda_T)$ of $(\mathscr M^{N,\e},\mathscr X^{N,\e})$ as $N\to\infty$. Let us consider any smooth function $\varphi$ supported on the unit ball of $\mathbb R^2$ such that $\|\varphi\|_{C^1}\leq 1.$ Recall $S_{(t,x)}^{\lambda}$ from \eqref{scale}. Using Skorohod's lemma, Fatou's lemma, and the Burkholder-Davis-Gundy inequality, one has that
		\begin{align*}
			\Ex[ |\mathscr M^{\infty,\e}(S_{(t,x)}^{\lambda}\varphi)|^p] & \leq \liminf_{N\to\infty} \Ex[ |\mathscr M^{N,\e}(\varphi)|^p] \leq \Con \cdot \liminf_{N\to\infty}\Ex \big[  \langle R^{N,\e}(S_{(t,x)}^{\lambda}\varphi)\rangle_{T}^{p/2}\big],
		\end{align*}
where $\Con=\Con(p)>0$. From Lemma \ref{qvt}, the quadratic variation is given by
\begin{align*}
    \langle R^{N,\e}(S_{(t,x)}^{\lambda}\varphi)\rangle_{T}= N^{1/2}\int_0^{T} \big( (\mathscr X_{s+\e}^N)^{Sq},((1+N^{-1/4}\partial_x)S_{(t,x)}^{\lambda}\varphi(s,\cdot))^2\big)ds.
\end{align*}
Using $(a+b)^2\leq 2(a^2+b^2),$ it is clear that $\big((1+N^{-1/4}\partial_x)\varphi\big)^2 \leq 2\varphi^2 + 2N^{-1/2}(\partial_x\varphi)^2$. Furthermore, the term with a factor $N^{-1/2}$ will not contribute when taking the $\liminf$ as $N\to \infty$ (by e.g. Proposition \ref{4.2}). 
Therefore, we can obtain that $$\liminf_{N\to\infty} \mathbb E [ \langle R^{N,\e}(S_{(t,x)}^{\lambda}\varphi)\rangle_{T}] \leq \liminf_{N\to\infty} N^{1/2} \mathbb E\bigg[ \int_0^T \big( (\mathscr X_{s+\e}^N)^{Sq} , (S_{(t,x)}^{\lambda}\varphi)^2\big)ds\bigg].$$
Recall the fields $Q^N$ defined in \eqref{QVfield}. Notice that $\big(S_{(t,x)}^\lambda \varphi\big)^2 \leq \lambda^{-6} \ind_{[t-\lambda^2,t+\lambda^2]\times [x-\lambda,x+\lambda]}$. Therefore, taking any $\delta>0$, by the above bound and by the second bound in Proposition \ref{tight1} we find that 
		\begin{align*}\mathbb E[ |\mathscr M^{\infty,\e}(S_{(t,x)}^\lambda\varphi)|^p ]&\leq 2\Con  \lambda^{-3p}\liminf_{N\to\infty}  \mathbb E \big[ \big(Q^N_{t+\e+\lambda^2}(\ind_{[x-\lambda,x+\lambda]}) - Q^N_{t+\e-\lambda^2}(\ind_{[x-\lambda,x+\lambda]})\big)^{p/2}\big] \\& \leq 2\Con \lambda^{-3p}  (\lambda^2)^{p/2} \| \ind_{[x-\lambda,x+\lambda]}\|_{L^{1+\delta}(\mathbb R)}^{p/2}= \Con \lambda^{-\frac32 p - \big[\frac{\delta}{1+\delta}\big]p }.
		\end{align*}
		Given any $\kappa>0$, by making $\delta=\delta(\kappa)$ close to $0$ and making $p=p(\kappa)$ large enough, we may then apply Lemma \ref{l:KC}\ref{KC} to conclude \eqref{e.tm2}. 
	\end{proof}

	\subsection{Identification of the limit points}\label{sec:iden}

	In this subsection, we identify the limit points as the solution of the stochastic heat equation \eqref{she} and thereby prove our weak convergence theorem: Theorem \ref{t.weakConv}.

\begin{thm}[Solving the martingale problem]\label{solving_mp}
		Consider the triplet of processes $(\mathscr X^N,M^N,Q^N)_{t\ge 0}$, where $\mathscr X^N$, $M^N$, and $Q^N$ are defined in \eqref{a}, \eqref{mart_ob}, and \eqref{QVfield} respectively. Fix $\alpha<-1, \beta<-2, \gamma<-1,$ and $\sig>1.$ These triples are jointly tight in the space $$C^{\alpha,\sig}_\mathfrak s(\Lambda_T)\times C([0,T],C^{\beta,\sig}(\mathbb R))\times C([0,T],C^{\gamma,\sig}(\mathbb R)).$$
		Consider any joint limit point $(X^\infty, M^\infty, Q^\infty)$. Then for any $s>0$, the process $(t,x)\mapsto X_{s+t}^\infty(x)$ is necessarily supported on the space $C_\mathfrak s^{-\kappa+1/2,\sig}( \Lambda_T)$. Furthermore, $(M_t^\infty(\phi))_{t\ge 0}$ is a continuous martingale for all $\phi\in C_c^\infty(\mathbb R)$, and moreover for all $0< s\leq t< T$ one has the almost sure identities 
		\begin{align}\label{mp1}
			M_t^\infty(\phi)-M_s^\infty(\phi) &= \int_\mathbb R (X_t^\infty(x)-X_s^\infty(x))\phi(x)dx -\frac12 \int_s^t \int_\mathbb R X_u^\infty(x)\phi''(x)dx du\\\label{mp2}
			\langle M^\infty (\phi) \rangle_t &= Q_t^\infty(\phi^2) \\\label{mp3}
			Q_t^\infty(\phi)-Q_s^\infty(\phi) &= \sigma \int_\mathbb R\int_s^t (X_u^\infty(x))^2 \phi(x)\,du\,dx.
		\end{align}
	\end{thm}

Before going into the proof, we remark that with some inspection it may be verified that all quantities make sense given the spaces they lie in, as long as we choose $\kappa<1/2$ to ensure that $X^\infty$ is a continuous function in space-time away from $t=0$. 

	\begin{proof} Most parts of the following theorem are already established in previous propositions and lemmas. Note that the tightness of $\mathscr X^N$ was proved in Proposition \ref{xtight}.  The tightness of $M^N$ and $Q^N$ was shown in Proposition \ref{mcts}. Thus the joint tightness follows from individual tightness. Consider any limit point $(\mathscr X^\infty,M^\infty,Q^\infty).$ The fact that for any $s>0$, the process $(t,x)\mapsto X^\infty(s+t,x)$ is necessarily supported on the space $C_\mathfrak s^{-\kappa+1/2,\sig}( \Lambda_T)$ was proved in Proposition \ref{tightm2}. From Proposition \ref{mcts} we have that for all $\phi\in C_c^\infty(\mathbb R)$ the process $(M_t^\infty(\phi))_{t\ge 0}$ is a martingale. \eqref{mp2} is already proven in Proposition \ref{mcts} as \eqref{e:mcts}. All we are left to show is \eqref{mp1} and \eqref{mp3}.

 \bigskip

\noindent\textbf{Proof of \eqref{mp1}.} 
  In Lemma \ref{ds} we obtained that $\partial_s M^N =-\delta_{(0,0)}+(\partial_t - \tfrac12\partial_x^2)\mathscr X^N+\delta_T \otimes \mathscr X^N_T$ in the sense of distributions, which by disregarding the boundary terms implies that $$(\partial_s M^N, \varphi) =((\partial_t - \tfrac12\partial_x^2)\mathscr X^N, \varphi)$$ for all $\varphi$ of compact support contained in $\Lambda_{[\e,T-\e]}$ for some $\e>0$. Note that this relation is still respected by any limit points and hence remains true when $N=\infty$. We thus have
  $$(\partial_s M^\infty, \varphi) =((\partial_t - \tfrac12\partial_x^2)\mathscr X^\infty, \varphi)$$
  Since both $\mathscr X^\infty$ and $M^\infty$ lie in spaces with strong enough topologies {($C_\mathfrak s^{-\kappa+1/2,\sig}( \Lambda_T)$ and $C([0,T],C^{\beta,\sig}(\mathbb R))$ respectively)}, taking $\varphi(u,x)$ to approach the function $(u,x) \mapsto \ind_{[s,t]}(u) \phi(x)$ in the above equation leads to \eqref{mp1}.
  
\bigskip

\noindent\textbf{Proof of \eqref{mp3}.} Fix any $0\le t\le T$ and let $\mu$ be a $\mathcal{S}'(\R)$ valued random variable defined as
\begin{align*}
    (\mu,\phi):= Q_t^{\infty}(\phi)-\sigma\int_{\R}\int_0^t (\mathscr{X}_s^{\infty}(x))^2\phi(x)\,ds\,dx.
\end{align*}
We claim that $ (\mu,\phi)=0$ a.s. for all $\phi\in \mathcal{S}(\R).$ This will validate \eqref{mp3}. To verify this, let us define a sequence of function-valued random variables:
\begin{align*}
    \mu^{\e}(a):= Q_t^{\infty}(\xi_{\e}^a)-\sigma\int_0^t (\mathscr{X}_s^{\infty}(\xi_{\e\sqrt{2}}^a))^2ds
\end{align*}
 where $\xi_{\e}^a(x):=\e^{-1}\xi(\e^{-1}(x-a))$ with $\xi(x):=\frac1{\sqrt\pi}e^{-x^2}$. Observe that for any $\phi\in C_c^{\infty}(\R)$ supported on $[-S,S]$,
by Cauchy-Schwarz and Jensen's inequality we have that $$\Ex\big[|(\mu^\e, \phi)|\big] \leq \|\phi\|_{L^2(\mathbb R)}\cdot \Ex\bigg[ \sqrt{\int_{[-S,S]} \mu^\e(a)^2 da}\bigg] \le \|\phi\|_{L^2} \bigg[\int_{[-S,S]} \Ex[\mu^\e(a)^2]da \bigg]^{1/2}.$$ 

We now make use of Proposition \ref{4.1}. From \eqref{e:QXpolylog}, we see that the integral is uniformly bounded in $\e>0$. On the other hand, \eqref{Qllim} tells us that the integrand goes to zero as $\e\to 0$. Thus, by the dominated convergence theorem $(\mu^\e , \phi) \to 0$ in probability. But $(\mu^\e , \phi) \to (\mu,\phi)$ almost surely as $\e \to 0$. Hence, $(\mu,\phi)=0$ almost surely. This proves the claim for compactly supported $\phi$. For general $\phi \in \mathcal S(\mathbb R)$, we may find a sequence of functions  $\phi_n\in C_{c}^{\infty}(\R)$ such that $\phi_n \to \phi$ in the topology of $\mathcal S(\mathbb R)$. Then $0=(\mu,\phi_n)\to (\mu,\phi)$ almost surely as $n\to\infty$. Hence,  $(\mu,\phi)=0$ almost surely for all $\phi\in \mathcal{S}(\R)$. This verifies the claim, completing the proof.
 \end{proof}			
	We now complete the proof of our main theorem, Theorem \ref{t.weakConv}.

	\begin{proof}[Proof of Theorem \ref{t.weakConv}\ref{t.main}] We continue with the notation and setup of Theorem \ref{solving_mp}. We have already established the tightness of $\mathscr X^N$ in the topology of $C_\mathfrak s^{\alpha,\sig}(\Lambda_T)$ in Proposition \ref{xtight}. Consider any limit point $ X^\infty$ of $\mathscr X^N$. From the previous theorem, we already know that $(t,x) \mapsto X_{t+\e}^\infty(x)$ is a continuous function in space and time. From the three equations \eqref{mp1}, \eqref{mp2}, and \eqref{mp3} in Theorem \ref{solving_mp} it follows that the martingale problem for \eqref{she} is satisfied by any limit point $X^\infty$. We refer the reader to \cite[Proposition 4.11]{BG97} for the characterization of the law of \eqref{she} as the solution to this martingale problem. 
		
		The result there is only stated for continuous initial conditions, so what this really shows is that for any $\e>0$ the law of the continuous field $(t,x) \mapsto X_{t+\e}^\infty(x)$ is that of the solution of \eqref{she} with initial condition $X_{\e}^\infty(\cdot).$ Thus we still need to pin down the initial data as $\delta_0$, by showing that we can let $\e\to 0$ and see that the limit of $X_{\e}^\infty(\cdot)$ is equal to $\delta_0$ in some sense. 
		
		In \cite[Section 6]{Par19} there is a general approach to do exactly this. Specifically, it suffices to show as in Lemma 6.6 of that paper the two bounds 
		\begin{align}\label{delta} & \Ex[|X_t^\infty(x)|^r]^{2/r}\leq \Con \cdot t^{-1/2}p_t(x),\\ \label{delta1}
			& \Ex[ |X_t^\infty(x)-p_t(x)|^r]^{2/r} \leq \Con \cdot p_t(x),
		\end{align}where $r>1$ is arbitrary, $p_t$ is the standard heat kernel, and $\Con$ is independent of $t>0$ and $x\in\mathbb R$. Clearly, it suffices to show this when $r$ is an even integer.
 {Using $L^2$ bounds for each term in the chaos expansion of \eqref{she} with $\delta_0$ initial condition (see Lemma 2.4 in \cite{corwin2018exactly}), it follows that the  solution of \eqref{she} with $\delta_0$ initial condition satisfies the above two bounds for $r=2$. The bounds for general $r$ then follow from the $r=2$ case by the hypercontractivity inequality \cite[Theorem 1.4.1]{nualart2006malliavin}.} By the moment convergence result (Proposition \ref{p:mcov}), we know any limit point $X^\infty$ must satisfy $\Ex[(\int_\mathbb R X_t^\infty(x)\phi(x)dx)^k] = \mE[(\int_\mathbb R \mathcal{Z}_t(x)\phi(x)dx)^k]$ for all $k\in\mathbb N$ and $\phi\in C_c^\infty(\mathbb R)$, where $(t,x) \mapsto \mathcal{Z}_t(x)$ solves \eqref{she} with $\delta_0$ initial data and $\sigma=\frac1{{2}\nu([0,1])}$. From here we can conclude by letting $\phi \to \delta_x$ that $\Ex[X_t^\infty(x)^k] = \mE[ \mathcal{Z}_t(x)^k]$ for all $k\in \mathbb N$ and all $x\in\mathbb R$. Consequently, we may immediately deduce \eqref{delta} and \eqref{delta1} by the corresponding bounds for $\mathcal{Z}_t$. This completes the proof. 
	\end{proof}

Finally, we conclude this section by proving Theorem \ref{t.weakConv}\ref{t.main2} which follows from the existence of a certain bounded linear operator.

 \begin{proof}[Proof of Theorem \ref{t.weakConv}\ref{t.main2}] Fix $\alpha<0$. We claim that there exists a bounded linear map $\overline{K\partial_s}: C([0,T],C^{\alpha,\sig}(\mathbb R))\to C([0,T],C^{\alpha,\sig}(\mathbb R))$ such that for all $v = (v_t)_{t\in[0,T]} \in C([0,T],C^{\alpha,\sig}(\mathbb R))$ one has that $\overline{K\partial_s} v  = K(\partial_s(v))$ where $K$ and $\partial_s$ are the same operators defined in Propositions \ref{sch} and \ref{ds} respectively (where the equality may be interpreted in the sense that both sides are viewed as elements of $\mathcal S'(\mathbb R^2)$).
 
 Note that the existence of $\overline{K\partial_s}$ would immediately imply the desired result. Indeed by applying $K$ to both sides of \eqref{e:delMrel} and using Proposition \ref{sch} and Lemma \ref{irrel} we get that $$\mathscr X^N = p+ K(\partial_s(M^N)),$$ where $p$ is the standard heat kernel. We also know from Proposition \ref{mcts} that the $M^N$ are tight in $C([0,T],C^{\gamma,\sig}(\mathbb R))$ for $\gamma<-2$, therefore tightness of $\mathscr X^N$ in $C([0,T],C^{\gamma,\sig}(\mathbb R))$ for $\gamma<-2$ would be immediate by continuity of $\overline{K\partial_s}$. The fact that any limit point coincides with the law of \eqref{she} follows from Theorem \ref{solving_mp}, \eqref{delta} and \eqref{delta1} by taking a joint limit point with the triple appearing there, then using the fact that $K(\partial_s(M^N)) = \overline{K\partial_s} M^N$. 

\smallskip

We now turn toward the existence of $\bar{K\partial_s}$. To construct $\overline{K\partial_s}$ we first remark that $C_c^\infty(\Lambda_T)$ embeds into $C([0,T],C^{\alpha,\tau}(\mathbb R))$ in the natural way: a function $v \in C_c^\infty(\Lambda_T)$ may be identified with $(v(t))_{t\in [0,T]}$ given by $(v(t),\phi)_{L^2(\mathbb R)} := \int_\mathbb R v(t,x)\phi(x)dx.$ Then $\partial_s v=:v'$ is given by $(v'(t),\phi) := \int_\mathbb R (\partial_tv)(t,x)\phi(x)dx.$ In this case we may integrate by parts to obtain 
\begin{align*}
    (K(\partial_sv)(t),\phi) &= \bigg( \int_0^t P_{t-s}v'(s)ds\;,\; \phi \bigg) \\ &= \bigg( v(t) - P_tv(0) + \int_0^t \partial_s P_{t-s} v(s) ds\;, \;\phi\bigg) \\ &= (v(t),\phi) - (v(0), P_t\phi) - \int_0^t (v(s), P_{t-s}\partial_x^2\phi) ds,
\end{align*}
 where all pairings are in $L^2(\mathbb R)$ and we use that $\partial_s P_{t-s} = -\partial_x^2 P_{t-s}, $ and $\partial_x^2$ and $P_t$ are self-adjoint operators on $L^2(\mathbb R)$.

Based on this calculation, we consider a general path $v\in C([0,T],C^{\alpha,\sig}(\mathbb R))$ and we finally define $\overline{K\partial_s}v$ suggestively by the formula \begin{equation}\label{kddef}\big(\big[\;\overline{K\partial_s} v\;\big](t), \phi):= (v(t),\phi) - (v(0), P_t\phi) - \int_0^t (v(s), P_{t-s}\partial_x^2\phi) ds. 
\end{equation}
First note that the pairings on the right-hand side are indeed well-defined for all $v\in C([0,T],C^{\alpha,\sig}(\mathbb R))$. What remains to be seen is that the integral is convergent and that the right-hand side is indeed an element of $C([0,T],C^{\alpha,\sig}(\mathbb R))$ whose norm may be controlled by $\|v\|_{C([0,T],C^{\alpha,\sig}(\mathbb R))}.$ By the definition of these spaces, we must replace $\phi$ by $S^\lambda_x\phi$ (where the scaling operators are given in Definition \ref{ehs}) and then study the growth as $\lambda$ becomes close to 0. 

The first two terms on the right side of \eqref{kddef} are completely straightforward to deal with: the growth is at worst $\lambda^\alpha (1+x^2)^\tau$ uniformly over $\lambda, \phi, x, T$, since we assumed $v\in C([0,T],C^{\alpha,\sig}(\mathbb R)).$    To deal with the third term, we claim that one has 
\begin{equation}\label{bd5}|(v(s), P_{t-s}\partial_x^2S^\lambda_x\phi) |\lesssim \big(\lambda^{\alpha-2} \wedge (t-s)^{\frac{\alpha}2-1}\big)(1+x^2)^\sig\end{equation} uniformly over $s<t\in [0,T]$, as well as $\lambda,\phi,x$. Indeed the bound of the form $\lambda^{\alpha-2}$ follows by noting that when $t-s$ is much smaller than $\lambda$, $P_{t-s}$ is essentially the identity operator, so we can effectively disregard the heat kernel and note that $\partial_x^2 S^\lambda_x\phi$ paired with $v(s)$ satisfies a bound of order $\lambda^{\alpha-2}$. Likewise the bound of the form $(t-s)^{\frac{\alpha}2-1}$ is obtained by noting that when $\lambda$ is very small compared to $t-s$, $P_{t-s}S_x^\lambda\phi$ behaves like $S_x^{\sqrt{t-s}}\phi$, giving a bound of order $(\sqrt{t-s})^{\alpha-2}$ after applying $\partial_x^2$ and pairing with $v(s)$. This proves the bound \eqref{bd5}.

Now the fact that $\|\overline{K\partial_s}v\|$ can be controlled by $\|v\|$ follows simply by noting that uniformly over $0<\lambda^2 \leq t\leq T$ one has $$\int_0^t \big(\lambda^{\alpha-2} \wedge (t-s)^{\frac{\alpha}2-1}\big) ds = \int_0^{t-\lambda^2}(t-s)^{\frac{\alpha}2-1}ds + \int_{t-\lambda^2}^t \lambda^{\alpha-2}ds \leq \Con(\alpha)\cdot  \lambda^\alpha. $$
It remains to verify that $\overline{K\partial_s} v$ is indeed a \textit{continuous} path. So far, what we have effectively shown is that $\overline{K\partial_s}$ is a bounded operator from $L^\infty([0,T], C^{\alpha,\sig}(\mathbb R))$ to itself. Note that $C([0,T], C^{\alpha,\sig}(\mathbb R))$ is precisely the closure of $\mathcal S(\Lambda_T)$ inside of $L^\infty([0,T], C^{\alpha,\sig}(\mathbb R)),$ where $\mathcal S(\Lambda_T)$ denotes restrictions to $\Lambda_T$ of those functions in $\mathcal S(\mathbb R^2)$. From here, path continuity of $\overline{K\partial_s}v$ for all $v\in C([0,T],C^{\alpha,\sig}(\mathbb R))$ is immediate because the bounded operator $\overline{K\partial_s}: L^\infty([0,T], C^{\alpha,\sig}(\mathbb R))\to L^\infty([0,T], C^{\alpha,\sig}(\mathbb R))$ maps the linear subspace $\mathcal S(\Lambda_T)$ into itself.

Since $\overline{K\partial_s}$ is indeed a bounded operator, and since $\overline{K\partial_s} v  = K(\partial_s(v))$ for smooth paths $v\in C_c^\infty(\Lambda_T)$ as shown above, this equality remains true for general $v$ by the denseness of the linear subspace $C_c^\infty(\Lambda_T)$ in $C([0,T],C^{\alpha,\sig}(\mathbb R)).$
 \end{proof}

 We remark that the above proof is essentially sharp, in the sense that $K$ improves parabolic regularity by $2$, while $\partial_s$ destroys parabolic regularity by $2$. Hence, we expect their composition to preserve regularity, which is what we showed. Nonetheless,  the H\"older exponent appearing in Theorem \ref{t.weakConv}\ref{t.main2} may not be sharp, because we do not know if we obtained tightness of $M^N$ in the best possible space in Proposition \ref{mcts}.

\section{Quenched tail field} \label{sec:qtailfield} The goal of this section is to prove the results from Section \ref{sec:qtfresults}, namely Theorem \ref{t:qtfconv} and Theorem \ref{t:max}. Recall the quenched tail field $F_N(t,x)$ from Definition \ref{qtf}. We begin with a preliminary estimate for the first moment of $F_N(t,x)$.

\begin{lem}\label{A1} Let $B$ be a standard Brownian motion. For all $t>0, x\in \R$, we have
    $$\mathbb E[F_N(t,x)] = N^{1/4}\mathbf E[ e^{-N^{1/4}(B_t-x)}\ind_{\{B_t\geq x\}}] \leq (2\pi t)^{-1/2}.$$ 
\end{lem}

\begin{proof}
    The proof proceeds by writing $F_N$ as a quenched expectation, exactly as we did in Section \ref{pfidea1} or Section \ref{sec:girs}. Then one takes the annealed expectation over the quenched expectation and uses the scale invariance of Brownian motion to obtain $$\mathbb E[F_N(t,x)] = N^{1/4}e^{\frac{t}{2}\sqrt{N}+N^{1/4}x}\mathbf P(B_t-N^{1/4}t\geq x).$$ Now applying Girsanov with the stochastic exponential of $-N^{1/4}B$ will give the claimed equality.  To prove the inequality,  we note that \begin{align*}N^{1/4}\mathbf E[ e^{-N^{1/4}(B_t-x)}\ind_{\{B_t\geq x\}}] &= N^{1/4} \int_x^\infty e^{-N^{1/4}(u-x)}p_t(u)du \\ &\leq (2\pi t)^{-1/2}\int_x^\infty N^{1/4}e^{-N^{1/4}(u-x)}du =(2\pi t)^{-1/2}.
    \end{align*}
\end{proof}
We first prove a variant of Theorem \ref{t:qtfconv} where the quenched tail field is integrated against test functions.

 \begin{prop}\label{hi} Fix $m\in \mathbb{N}$. For all $\phi_1,\ldots,\phi_m\in C_c^{\infty}(\R)$ and $t_1,\ldots,t_m>0$ we have 
\begin{align}\label{qqt}
	\bigg( \int_\mathbb R \phi_i(x) F_N(t_i,x) dx\bigg)_{i=1}^m \stackrel{d}{\to} \bigg( \int_\mathbb R \phi_i(x) \mathcal Z_{t_i}(x) dx\bigg)_{i=1}^m.
\end{align}
  \end{prop}

\begin{proof}Let us formally define $$\mu_t^N(dx):= N^{1/2}e^{\frac{t}2 \sqrt{N} } K_{0,Nt}(N^{3/4}t+N^{1/2}x)dx.$$
The above definition can be interpreted rigorously as a random Borel measure by pushing forward $K_{0,Nt}$ under a linear change of coordinates,  just as we did in \eqref{a}. Then we notice that Definition \ref{qtf} is equivalent to $$F_N(t,x):= N^{1/4} e^{N^{1/4}x} \mu_t^N[x,\infty).$$
{Integrating by parts we have} $$\int_{-\infty}^x N^{1/4}e^{N^{1/4}u} \mu_t^N[u,\infty)du = \int_{-\infty}^x e^{N^{1/4}u}\mu_t^N(du) + e^{N^{1/4}x}\mu_t^N[x,\infty).$$  
In the sense of distributions, it is clear that 
        \begin{align*}
           & F_N(t,x) = \partial_x \bigg[\int_{-\infty}^x N^{1/4}e^{N^{1/4}u} \mu_t^N[u,\infty)du\bigg], \quad e^{N^{1/4}x}\mu_t^N(dx)=\partial_x\bigg[\int_{-\infty}^x e^{N^{1/4}u}\mu_t^N(du)\bigg].
        \end{align*}
        Consequently, for each $1\le i\le m$, we can repeatedly integrate by parts to obtain that 
        \begin{align*}
            \int_{\mathbb R}\phi_i(x) F_N(t_i,x)dx &= -\int_{\mathbb R} \phi_i'(x) \bigg[\int_{-\infty}^x N^{1/4}e^{N^{1/4}u} \mu_{t_i}^N[u,\infty)du\bigg]dx \\ &= -\int_{\mathbb R} \phi_i'(x)\bigg[\int_{-\infty}^x e^{N^{1/4}u}\mu_{t_i}^N(du)+ e^{N^{1/4}x}\mu_{t_i}^N[x,\infty)\bigg]dx \\ &= \int_{\mathbb R}\phi_i(x) e^{N^{1/4}x}\mu_{t_i}^N(dx) - \int_{\mathbb R} \phi_i'(x)e^{N^{1/4}x}\mu_{t_i}^N[x,\infty)dx.
        \end{align*}
Thanks to Theorem \ref{t.weakConv}\ref{t.main2} (see Remark \ref{rk.main}\ref{finiteconv2}) we know that as $N\to \infty$ $$\bigg(\int_{\mathbb R}\phi_i(x) e^{N^{1/4}x}\mu_{t_i}^N(dx)\bigg)_{i=1}^m \stackrel{d}{\to} \bigg(\int_{ \R}\phi_i(x) \mathcal Z_{t_i}(x)dx\bigg)_{i=1}^m.$$ 
Thus, to show \eqref{qqt} it suffices to show that as $N\to \infty$, $\int_{\mathbb R} \phi_i'(x)e^{N^{1/4}x}\mu_{t_i}^N[x,\infty)dx$ goes to zero in $L^1(\mathbb P)$  for each $1\le i\le m$. Indeed, observe that
    \begin{align*}
       \mathbb E\bigg|\int_{\mathbb R} \phi_i'(x)e^{N^{1/4}x}\mu_{t_i}^N[x,\infty)dx\bigg| & \leq \int_{\mathbb R} |\phi_i'(x)|e^{N^{1/4}x}\mathbb E[\mu_{t_i}^N[x,\infty)]dx. 
    \end{align*}
  By Lemma \ref{A1} we get $$e^{N^{1/4}x}\mathbb E[\mu_{t_i}^N[x,\infty)]=\mathbf E[ e^{-N^{1/4}(B_{t_i}-x)}\ind_{\{B_{t_i}\geq x\}}]\leq N^{-1/4}(2\pi t_i)^{-1/2}$$ for all $x\in\mathbb R$, which implies the desired $L^1$ convergence. This completes the proof.
  \end{proof}

Note that composition with the logarithm function is not a continuous operation in the topology of integration against test functions, hence the convergence result stated above cannot be directly translated to any convergence result for the KPZ equation.
 Our next lemma {demonstrates uniform convergence for the two-point correlation function} of the quenched tail field. This is the crucial result that allows us to improve Proposition \ref{hi} to obtain multipoint convergence in law of $F_N(t,x)$ to $\mathcal Z_t(x)$ for \textit{individual} values of $(t,x)\in (0,\infty)\times \mathbb R$, after which taking a logarithm becomes sensible.
  
  \begin{lem}\label{7.5}
  As $N\to \infty$, we have that
  $$\mathbb E[F_N(t,x)F_N(t,y)] \to \mE[\mathcal Z_t(x)\mathcal Z_t(y)]$$
 where $\mathcal{Z}$ as in Proposition \ref{p:mcov}. Furthermore, the above convergence is uniform over compact sets of $(t,x,y) \in (0,\infty)\times \R^2$.
  \end{lem}

  \begin{proof} By adapting the methods in Section \ref{sec:girs}, the two-point correlation function given by $(t,x,y)\mapsto \mathbb E[F_N(t,x)F_N(t,y)]$ can be written as an expression involving only the 2-point motion associated to the kernels $K_{0,t}$. More specifically, in the notation of Section \ref{sec:girs}, it equals $$Ne^{t\sqrt{N}}\mathbf P_{SB^{(2)}_{N^{1/2}\nu}} (X_t-N^{1/4}t \geq x, Y_t-N^{1/4}t \geq y).$$ Consequently the two-point correlation function only depends on the total mass of the characteristic measure. We shall thus assume $\nu$ is a multiple of uniform measure. Without loss of generality, we will assume $\sigma=1$ (i.e., $\nu$ is precisely half times the uniform measure).
  
  We shall now invoke results from \cite{mark} which are for the {case where the characteristic measure is a multiple of the uniform measure on $[0,1]$}. In \cite{mark}, the authors consider a slightly different field $\hat F_N(t,x)$ which is ``dual" to ours in a certain sense. More specifically, they vary the starting point $x$ and fix the tail probability as $[0,\infty)$, whereas we fix the starting point $0$ and vary the tail probability as $[x,\infty)$. One may show that the distributions of both fields are the same as space-time processes, modulo a reflection of the characteristic measure. More precisely, we have $F_N^\nu \stackrel{d}{=} \hat F_N^{\mu}$, where $\mu$ is the pushforward of $\nu$ by $x\mapsto 1-x$, and the superscript highlights the dependence of each field on the underlying characteristic measure.\footnote{In particular, our results imply multipoint convergence of their field $\hat F_N$ to \eqref{she} as well.} This is proved using an explicit coupling of the two fields via the Brownian web. See \cite[Equation (1.8)]{yu} and the subsequent discussion for a short proof, which in turn is based on a construction of \cite[Section 3]{sss}. 
  
  Note that for the case where {the characteristic measure is a multiple of the uniform measure}, we have $\mu = \nu$. Proposition 1.22 in \cite{mark} provides exact moment formulas for the unscaled version of $F_N(t,x)$ in this case. Taking the scalings into account we have that
\begin{align*}
    \mathbb E[ F_N(t,x)F_N(t,y)] & =\oint_{r_1+i\mathbb R} \oint_{r_2 + i \mathbb R} \frac{N^{-1/2}(z_2-z_1)}{N^{-1/2}(z_2-z_1) - N^{-1/2} - N^{-3/4}(z_1+z_2) 
 - N^{-1} z_1z_2} \\ & \hspace{1cm}\cdot e^{\frac{t}{2} (z_1^2+z_2^2)+(xz_1+yz_2)}\frac{N^{-1/2}}{(N^{-1/4}+N^{-1/2}z_1)(N^{-1/4}+N^{-1/2}z_2)} \frac{dz_1}{2\pi i}\frac{dz_2}{2\pi i},
\end{align*}
$$ $$
 where $r_1,r_2$ are fixed numbers not depending on $t,x,y$ and $N$ (large enough), and such that $r_2>r_1 + 1.$ 

Now assume that we have a sequence $(t^N,x^N,y^N)$ converging to $(t,x,y)\in (0,\infty)\times \mathbb R^2$. Note that for $z_1 \in r_1+i\mathbb R$ and $z_2 \in r_2+i\mathbb R$ we have that $|e^{\frac{t^N}{2} (z_1^2+z_2^2)+(x^Nz_1+y^Nz_2)}| \leq \Con e^{-\frac1\Con(|z_1|^2+|z_2|^2)}$ for some constant $\Con>0$ which is independent of $N,z_1,z_2$. Also, note that the ratios appearing in the integral may be bounded independently of $N$ thanks to the fact that $z_1 \in r_1+i\mathbb R$ and $z_2 \in r_2+i\mathbb R$ with $r_2>r_1+1$. Therefore, by the dominated convergence theorem (with dominating function given by $e^{-\frac1\Con(|z_1|^2+|z_2|^2)}$) we may conclude that the integral expression for $\mathbb E[ F_N(t^N,x^N)F_N(t^N,y^N)]$ converges as $N\to \infty$ to $$\oint_{r_1+i\mathbb R} \oint_{r_2 + i \mathbb R} \frac{z_2-z_1}{z_2-z_1 - 1} e^{\frac{t}{2} (z_1^2+z_2^2)+(xz_1+yz_2)}\frac{dz_1}{2\pi i}\frac{dz_2}{2\pi i}.$$
 This is known to agree with $\mathbf E[\mathcal Z_t(x)\mathcal Z_t(y)]$, see \cite[Section 6.2]{bigmac}, thus proving the uniform convergence.
  \end{proof}
With the above lemma in place, we now prove a stronger version of Theorem \ref{t:qtfconv}.

\begin{prop}\label{unifconv} Suppose that the deterministic sequence of vectors $(x^N_1,\ldots,x^N_m) \in \mathbb R^m$ converges as $N\to \infty$ to $(x_1,\ldots,x_m)\in \R^m$. Fix $t_1,\ldots,t_m>0$. Then we have the joint convergence $$\big( F_N(t_i,x^N_i))_{i=1}^m \stackrel{d}{\to} \big( \mathcal Z_{t_i}(x_i)\big)_{i=1}^m.$$
\end{prop}

  \begin{proof} We first prove the case when $x_i^N=x_i$ for $1\le i\le m$.   We give a proof for $m=1$ to simplify the notation, but the generalization to larger $m$ is straightforward. We will simply write $(t_1,x_1)$ as $(t,x)$, which will be fixed throughout the proof. By Lemma \ref{A1} we get that $\mathbb E[F_N(t,x)]\leq (2\pi t)^{-1/2}$ which is independent of $N$, thus it follows that the $\{F_N\}_{N\ge 1}$ are tight.  Consider any limit point $\mu_0$ of the laws of $\{F_N(t,x)\}_{N\ge 1}$. Fix a smooth compactly supported nonnegative function $\phi:\mathbb R\to \mathbb R$ which integrates to $1$, and define $\phi^\lambda_x(y) = \lambda^{-1}\phi(\lambda^{-1}(y-x)).$ By Proposition \ref{hi} we know that for each $\lambda,x\in \R$, $\big\{\int_\mathbb R \phi^{\lambda}_x(y) F_N(t,y) dy\big\}_{N\ge 1}$ is a tight sequence. For $r\ge 1$, consider any measure $\mu_r$ on $\mathbb R^{r+1}$ which is a joint limit point as $N\to \infty$ of $$\bigg(F_N(t,x) \ , \  \big(\int_\mathbb R \phi^{2^{-k}}_x(y) F_N(t,y) dy\big)_{k=1}^r\bigg),$$ such that the first marginal of $\mu_r$ is $\mu_0$. Using a diagonalization argument, these measures $\mu_r$ may be chosen so that they form a projective family, and therefore by the Kolmogorov extension theorem, we may consider any projective limit $\mu$ which will be a measure on the space of sequences $(a_k)_{k\ge 0}\in \mathbb R^{\mathbb N_0}$, equipped with the $\sigma$-algebra generated by the projection maps. By Proposition \ref{hi}, we find that for such a measure $\mu$ the marginal distribution of $(a_k)_{k\ge 1}$ is simply equal to the law of $\big( \int_\mathbb R \mathcal Z_t(y)\phi^{2^{-k}}_x(y)dy\big)_{k\ge 1}$.  We now claim that  
  \begin{align}\label{e:fbds}
          \limsup_{\lambda\to 0}\limsup_{N\to\infty} \mathbb E \bigg| F_N(t,x) - \int_\mathbb R F_N(t,y) \phi^\lambda_x(y)dy\bigg|  = 0.
      \end{align}
  Assuming this fact, one finds that such a measure $\mu$ is necessarily supported on those sequences $(a_k)_{k\ge 0}$ which satisfy $a_0 = \lim_{k\to \infty} a_k$ in $L^1(\mu)$, which means that $a_0$ must have the law of $\mathcal Z_t(x)$ under $\mu$. We are thus left to check \eqref{e:fbds}. 
   By Lemma \ref{7.5}, we have that
      $$\limsup_{N\to\infty} \mathbb E[ (F_N(t,x)-F_N(t,y))^2] = \mE[(\mathcal Z_t(x)-\mathcal Z_t(y))^2] \le \Con|x-y|$$
uniformly over compacts sets of $(t,x,y)\in (0,\infty)\times\R^2$ where the above inequality is a known estimate for \eqref{she} (see \cite[Proposition 2.4-nw]{das} for example). Thus, by Fatou's lemma, we note that 
      \begin{align*}
          \limsup_{N\to\infty} \mathbb E \bigg| F_N(t,x) - \int_\mathbb R F_N(t,y) \phi^\lambda_x(y)dy\bigg|  &\leq \limsup_{N\to\infty} \int_\mathbb R\mathbb E| F_N(t,x)-F_N(t,y)| \phi^\lambda_x(y)dy \\ &\leq  \int_\mathbb R \limsup_{N\to\infty} \mathbb E| F_N(t,x)-F_N(t,y)| \phi^\lambda_x(y)dy\\&\leq\int_\mathbb R \limsup_{N\to\infty}\mathbb E[ (F_N(t,x)-F_N(t,y))^2]^{1/2} \phi^\lambda_x(y)dy \\ &\leq \Con\int_\mathbb R |x-y|^{1/2} \phi^\lambda_x(y)dy =\Con\lambda^{1/2} \int_\mathbb R |u|^{1/2} \phi(u)du.
      \end{align*}
Taking $\lambda\to 0$ leads to \eqref{e:fbds}. To justify the above application of Fatou's lemma with $\limsup$, one needs a bound on the integrand which is independent of $N$ and is in $L^1(\phi^\lambda_x(y)dy)$, for which we may use Lemma \ref{A1} to get $$\mathbb E| F_N(t,x)-F_N(t,y)| \leq \mathbb E[ F_N(t,x)+F_N(t,y)] \leq 2(2\pi t)^{-1/2}.$$
Finally, the case when $x_i^N$ varies with $N$ follows from the fact that as $N\to \infty$ we have $$\mathbb E\big[\big(F_N(t_i,x_i^N) - F_N(t_i,x_i)\big)^2 \big]\to 0$$ due to the uniform convergence in Lemma \ref{7.5}. This completes the proof.
      \end{proof}

\begin{rk}\label{7.7}  {We remark that one can allow $t_1,\ldots,t_m$ to depend on $N$ as well, and the multi-point convergence result analogous to Proposition \ref{unifconv} would still be true as long as the $t_i^N$ converge to some positive real numbers $t_i$ as $N\to \infty$. These are the precise modifications needed: first one modifies the proof of Proposition \ref{hi} to allow the $t_i$ to depend on $N$. This is fairly immediate since Theorem \ref{t.weakConv}\ref{t.main2} gives convergence in a topology that is uniform in time. Then notice that in Lemma \ref{7.5} we have already allowed the $t_i$ to depend on $N$. Thus we can replace $t_i$ by $t_i^N$ in Proposition \ref{unifconv} as well (the estimate using Fatou's Lemma will still work), which is enough to establish the claim. After the claim is established, some interesting effects can be observed by setting $t^N_i:= t_i+ \alpha N^{-1/4}$ for $\alpha\in \mathbb R$. Using Definition \ref{qtf} one sees in this case that $\alpha$ gets reinterpreted as a shift of the $x$ variable, but with different scaling values.}
\end{rk}

\begin{proof}[Proof of Theorem \ref{t:max}] We continue with the same notations as in the statement of the Theorem \ref{t:max}. Fix any $a\in \R$.
 Set $v:=\sqrt{ct}$ and $x:=d\sqrt{\frac{t}{c}}$ so that $c=\frac{v^2}{t}$ and $d=vx/t$.  It suffices to show that 
 \begin{equation}
 \label{gbel}
     \begin{aligned}
     \lim_{N\to\infty}\mathbf P_{SB^{(k)}_\nu} \bigg(\max_{1\le i \le k(N)} \big\{X^i(Nt)\big\} - vN^{3/4} - x\sqrt{N}& -\tfrac{tN^{1/4}}{v} \big(r_N-\tfrac14 \log N\big) \leq a N^{1/4}\bigg) \\ &= {\mathbb P(t/v[G+\log \mathcal Z_{v^2/t}(vx/t)] \leq a)},
 \end{aligned}
 \end{equation}
  where $\mathcal Z$ solves \eqref{she} with $\sigma = \frac{v}{{2}t \nu([0,1])}$ and $G$ is  an independent Gumbel random variable.  

     \smallskip
     
     For now, fix any $k\in \mathbb N$ which may be arbitrary. Fix a realization of the kernels $\omega = \{K_{s,t}:-\infty<s<t<\infty\}.$ As before, we let $\mathbb P_\omega^{(k)}$ be the quenched probability of $k$ random motions $(X^i(t))_{i=1}^k$ sampled from these kernels. By the scale invariance of SBM kernels, the random variable $$\mathbb P_\omega^{(k)} \bigg(\max_{1\le i \le k} \big\{X^i(Nt)\big\} \leq v N^{3/4}+x\sqrt{N}+\frac{t}{v} N^{1/4}(r_N-\frac14 \log N) + aN^{1/4}\bigg) $$ has the same law for any $y>0$ as the random variable $$\mathbb P_\omega^{(k)} \bigg(\max_{1\le i \le k} \big\{X^i(Ny^2t)\big\} \leq y\big[v N^{3/4}+x\sqrt{N}+\frac{t}{v} N^{1/4}(r_N-\frac14 \log N) +aN^{1/4}\big]\bigg)$$ except that the characteristic measure in the latter expression has changed to $y\cdot \nu$. Now we make a specific choice of $y:=v/t$ and we get by Definition \ref{qtf} that it is equal (pathwise, not just in distribution) to the quantity 
     \begin{align}\label{disp}
        \bigg( 1-\exp\bigg[-\frac{v^2}{2t}\sqrt{N} - N^{\frac14}\frac{vx}{t} - \frac{av}{t} -r_N\bigg] \cdot F_N\big(v^2/t, \;\tfrac{vx}{t} + N^{-\frac14} (\tfrac{av}{t} +r_N -\tfrac14\log N)\big)\bigg)^k.
     \end{align}
     where it is implicit that the characteristic measure associated to $F_N$ is $(v/t)\cdot \nu$. Then we may choose $$k=k(N)= \bigg\lfloor \exp\bigg[\frac{v^2}{2t}\sqrt{N} + \frac{vx}{t} N^{1/4} + r_N\bigg]\bigg\rfloor$$ as stated in the theorem to obtain that \eqref{disp} is a quantity of the form $(1-\frac{u(N)}{k(N)+O_N(1)})^{k(N)}$, where the $O_N(1)$ term is deterministic and bounded between $0$ and $1$, and $u(N)$ is a sequence of random variables which by Proposition \ref{unifconv} converges in law to the strictly positive random variable $e^{-\frac{av}t}\mathcal Z_{v^2/t}(vx/t)$. Since the functions $u\mapsto (1-\frac{u}{k(N)+O_N(1)})^{k(N)}$ converge uniformly to $u\mapsto e^{-u}$ on compact subsets of $[0,\infty)$, we can conclude that the expression in \eqref{disp}, therefore converges in law to the random variable $$\exp \big[ -e^{-\frac{av}t}\mathcal Z_{v^2/t}(vx/t)\big],$$ where $\mathcal Z$ solves \eqref{she} with $\sigma=\frac{v}{{2}t\nu([0,1])}$. Now this convergence in law also implies convergence of the associated annealed expectations thanks to the boundedness of all quantities involved, which implies that the limit of the left side of \eqref{gbel} equals $$\mathbb E\bigg[\exp \big[ -e^{-\frac{av}t}\mathcal Z_{v^2/t}(vx/t)\big]\bigg],$$ which agrees with the right-hand side of \eqref{gbel} by the explicit form of the Gumbel distribution.
 \end{proof}

Using Remark \ref{7.7}, one could also consider replacing time $Nt$ by time $Nt$ plus a correction of order $N^\alpha$ for some $\alpha<1$, and this would change the constants appearing in the proof in an explicit manner and introduce even more terms in the recentering if $\alpha \in [1/2,1)$. For brevity, we will not explore this.
\appendix

\section{Technical lemmas on Brownian bridges} \label{app}

\begin{defn}[Brownian concatenated process] \label{d.bcp} Fix any $x\in \R$, and $s\in [0,1]$. Let $B_{\operatorname{br}}$ be a Brownian bridge on $[0,s]$ from $0$ to $x$. Let $W$ be a standard Brownian motion independent of $B_{\operatorname{br}}$. We define the Brownian concatenated process $B_{s,x}$ with anchor at $(s,x)$ as
\begin{align*}
    B_{s,x}(y):=\begin{cases}
    B_{\operatorname{br}}(y) & y\in [0,s] \\  
    W(y-s)+x & y\in [s,1].
    \end{cases}
\end{align*}
    
\end{defn}

	\begin{lem}\label{6.1}
		 Let $W=(W_s)_{s\in [0,1]}$ be a Brownian motion on $[0,1]$. Let $F=(F_s)_{s\in[0,1]}:C[0,1]\to C[0,1]$ be a Borel-measurable functional. Suppose that $\mE[\sup_{s\leq 1}F_s(W)^2]<\infty$. Suppose further that the function $g(s,x):=\mE\left[F_s(B_{s,x})\right]$ is bounded on $[0,1]\times\mathbb R$ and continuous in $x$ for each $s>0$. Then for all compactly supported smooth functions $\phi$ we have the identity 
		$$\mE\bigg[ \int_0^1 F_s(W)\phi(W_s)ds \bigg] = \int_0^1 \int_{\mathbb R} g(s,x)p_s(x) \phi(x)\,dx\,ds,$$ and moreover for all $x\in\mathbb R$ we have
		$$\mE\bigg[ \int_0^1 F_s(W) dL_x^W(s)\bigg]= \int_0^1 g(s,x) p_s(x) ds,$$ 
        where $dL^W_x(s)$ is the Lebesgue-Stiltjes measure associated to the nondecreasing function $s\mapsto L_x^W(s).$
	\end{lem}
We remark that all expectations appearing above are finite. Let $Z:=\sup_{s\leq 1}|F_s(W)|$. Note that $\big|\int_0^1 F_s(W)\phi(W_s)ds\big| \leq Z\cdot \int_0^1 |\phi(W_s)|ds$ and $\big|\int_0^1 F_s(W) dL_x^W(s)\big| \leq Z \cdot L^W_x(1).$ Therefore a simple application of Cauchy-Schwarz (noting that $\mE[Z^2]<\infty$ by assumption) shows that all expectations above are finite.
 
	\begin{proof}  Note that for all $\phi$ continuous and compactly supported, by the tower property of conditional expectation we have $$\mE\bigg[ \int_0^1 F_s(W)\phi(W_s)ds \bigg] = \int_0^1 \mE[F_s(W)\phi(W_s)]ds= \int_0^1 \mE\big[\mE[F_s(W)|W_s]\phi(W_s)\big]ds.$$ 
 Now the law of $W$ conditioned on $W_s$ is precisely given by the Brownian concatenated process with anchor at $(s, W_s).$ Thus, $\mE[F_s(W)|W_s] = g(s, W_s)$. Therefore, the above expectation equals $$\int_0^1 \mE[g(s,W_s)\phi(W_s)]ds = \int_0^1 \int_\mathbb R g(s,x)p_s(x)\phi(x)dxds,$$ proving the first identity. To prove the second identity, we will let $\phi$ converge weakly to a Dirac mass and use the continuity in $x$ of $g$. Let $\phi$ be smooth, nonnegative, even, supported in $[-1,1]$, and suppose that it integrates to 1. Let $\phi^\e(u):=\e^{-1}\phi(\e^{-1}u)$. It is standard that the random measure on $[0,1]$ given by $\phi^\e(x-W_s)ds$ converges almost surely as $\e\to 0$ to $L^W_x(ds),$ with respect to the topology of weak convergence on the space of finite measures on $[0,1].$ Consequently, since we assume continuity of $s\mapsto F_s(W)$ we have the almost sure convergence 
		\begin{equation}\label{app1}\int_0^1 F_s(W)\phi^\e(x-W_s)ds \stackrel{\e\to 0}{\longrightarrow} \int_0^1 F_s(W) L^W_x(ds).
		\end{equation}
		Now by the first identity, the expectation of the left-hand side equals $$\int_0^1 H_x^\e(s) ds, \ \ \ \mbox{ where }H_x^\e(s):=\int_\mathbb R g(s,u)p_s(u)\phi^\e(x-u)du.$$ It is clear from the continuity of $g$ in $x$ that $H_x^\e(s) \to g(s,x)p_s(x)$ as $\e\to 0$, and moreover since $g$ is assumed to be globally bounded and $\phi$ is supported in $[-1,1]$, we have that $$|H_x^\e(s)| \leq \max_{u\in[x-1,x+1]}|g(s,u)| p_s(u) \leq \|g\|_{L^\infty} s^{-1/2}.$$ Hence by the dominated convergence theorem we find that $$\mE\bigg[\int_0^1 F_s(W)\phi^\e(x-W_s)ds\bigg]=\int_0^1 H^\e(s)ds \to \int_0^1 g(s,x)p_s(x)ds.$$ Thus it suffices to prove that \eqref{app1} also holds in $L^1$. To prove that, we note that $$\bigg|\int_0^1 F_s(W)\phi^\e(x-W_s)ds\bigg|\leq \int_0^1 |F_s(W)| \phi^\e (x-W_s)ds \leq Z \int_0^1 \phi^\e(x-W_s)ds,$$ where $Z:=\sup_{s\in [0,1]} |F_s(W)|$. By the occupation time formula for the local time, we have 
		$$\sup_{\e\in (0,1]} \int_0^1 \phi^\e(x-W_s)ds = \sup_{\e\in(0,1]} \int_\mathbb R \phi^\e(x-u) L^W_u(1)du \leq \sup_{u\in\mathbb R} L^W_u(1).$$ Summarizing the last two expressions and applying the Cauchy-Schwarz inequality, we have that $$\mE_{BM} \bigg[ \sup_{\e\in (0,1]} \bigg|\int_0^1 F_s(W)\phi^\e(x-W_s)ds\bigg| \bigg] \leq \mE[Z^2]^{1/2} \mE\big[\sup_{u\in\mathbb R} L^W_u(1)^2\big]^{1/2}.$$ The expected square of the sup over local times is finite, e.g., by \cite[Chapter XI, Theorem (2.4)]{RY99}. Therefore, we have uniform integrability and thus $L^1$ convergence in \eqref{app1} as desired.
	\end{proof}

	\begin{lem}\label{coupled}
		Let $B^{s,x}(t)$ be a Brownian bridge on $[0,s]$ from $0$ to $x$. Denote by $L^{B^{s,x}}_0(t)$ the local time at 0 of $B^{s,x}$ at time $t$. Then the map $(s,x) \mapsto L^{B^{s,x}}_0(s)$ has a modification which is almost surely continuous in both variables on $(0,1]\times \mathbb R$.
	\end{lem}
	
	\begin{proof}
		Let $G(s,x):= L^{B^{s,x}}_0(s)$. Note that the zero set of $B^{s,xs^{1/2}}$ is just a reparametrization of that of $B^{x,1}$, in fact $B^{x,xs^{1/2}}$ is just $s^{1/2}$ multiplied by a reparametrization of $B^{x,1}$, and so we have that $G(s,xs^{1/2})=G(1,x)$. Thus $G(s,x)=G(1,xs^{-1/2})$ and therefore it suffices to show a.s. H\"older continuity of $x\mapsto G(1,x)$. 

  \medskip
  
		To prove continuity of $x\mapsto G(1,x)$, let us simplify notation and henceforth write $B^{1,x} = B^x$. Now write $L_0^{B^x}(1) = L_0^{B^x}(1/2) + L_0^{W^x}(1/2)$, where $W^x(t) = B^x(1-t).$ Note that on the time interval $[0,1/2]$, $B^x$ is absolutely continuous with respect to a Brownian motion of drift $x$ started from 0, and moreover the Radon-Nikodym derivative is \textit{independent} of $x$. Likewise, $W^x$ is absolutely continuous with respect to a Brownian motion of drift $-x$ started from $x$, and similarly, the Radon-Nikodym derivative is \textit{independent} of $x$. Consequently, for a standard Brownian motion $Z$ if we define $Z^{a,\lambda}(t):=a+\lambda t+Z(t),$ then it suffices to show that $(a,\lambda) \mapsto L_0^{Z^{a,\lambda}}(1/2)$ is a.s. jointly H\"older continuous as this would imply the desired result for both $B^x$ and $W^x$. Note that $L_0^{Z^{a,\lambda}}(1/2) = L_{-a}^{Z^{0,\lambda}}(1/2).$ For Brownian motion, joint continuity of local times in the drift and the spatial variable is now a standard result.
	\end{proof}

 \begin{lem}\label{tl} Fix $a,b,c,d,x\in \R$ with $c^2+d^2=1$. Let $B$ be a Brownian bridge from $0$ to $x$ on $[0,1]$ and let $W$ be an independent standard Brownian motion. Consider the process $X:=a+cB+dW$. For all $\theta>0$, we have
 \begin{equation}
     \label{tlo}
     \mE[e^{\theta L_0^X(1)}]\leq \theta\sqrt{2 \pi}e^{\theta^2/2}.
 \end{equation}
 \end{lem}
 \begin{rk}\label{rklast}
     Since the above bound is uniform in $a,b,c,d,x$, the same bound continues to hold when $a,b,c,d,x$ are random variables (with $c^2+d^2=1$) independent of $B$ and $W$.
 \end{rk}
 \begin{proof} Note that conditioned on $W(1)$, $X$ is  a standard Brownian bridge from $a$ to $a+cx+dW(1)$. In \cite[(3)]{pitman} it is shown that for a standard Brownian bridge $Y$ from $a$ to $a+b$ one has for all $v>0$ that $$\mathbb P(L_0^Y(1)>v) = e^{b^2/2}e^{-(|a|+|a+b|+v)^2/2}$$ and therefore if $\theta>0$ then 
		$$\mE[e^{\theta L_0^Y(1)}] = \int_0^\infty \theta e^{\theta y} \mathbb P(L_0^Y(1)>y)dy\leq \theta e^{b^2/2}\int_{\mathbb R} e^{\theta y} e^{-(|a|+|a+b|+y)^2/2}dy= \theta \sqrt{2\pi} e^{\theta^2/2}e^{-(|a|+|a+b|)\theta}.$$
	Thus returning to the process $X$ we simply take $b=cx+dW(1)$, and we find that $$\mE[e^{\theta L_0^X(1)}|W(1)] \leq \theta\sqrt{2\pi}e^{\theta^2/2} e^{-(|a|+|a+cx+dW(1)|)\theta}.$$ Taking the expectation again and noting that $e^{-(|a|+|a+cx+dW(1)|)\theta}\leq 1$ we arrive at \eqref{tlo}, proving the claim.  
 \end{proof}
	\begin{lem}\label{bcts} Fix $r,m\in \mathbb{Z}_{>0}$. Fix $\mathbf x=(x_1,x_2,\ldots,x_r)\in \R^r$, $\mathbf t=(t_1,t_2,\ldots,t_m) \in [0,1]^m$, and $\mathbf s=(s_1,s_2,\ldots,s_r) \in [0,1]^r$. Let $\mathbf B^{\mathbf s, \mathbf x}=(B^i)_{i=1}^r$ be $r$-many independent processes with each $B^i$ being distributed as a Brownian concatenated process with the anchor at $(s_i,x_i)$ (recall Definition \ref{d.bcp}). Let $\ell_k: \mathbb R^n \to \mathbb R$ be linear functions for $1\leq k \leq m$. Define $$f(\mathbf{x},\mathbf{s},\mathbf{t}):= \mathbf E \bigg[ \exp\bigg(\sum_{k=1}^mL_0^{\ell_k(\mathbf B^{\mathbf s, \mathbf x})}(t_k)\bigg)\bigg].$$ Then $f$ is globally bounded and continuous on $\mathbb R^n\times (0,1]^r\times (0,1]^m.$
	\end{lem}
	
	\begin{proof}
		We use a coupling argument based on Lemma \ref{coupled}. Take $W^1,\ldots,W^n$ to be i.i.d.~Brownian motions on $[0,1]$ and then define $$B^{j,s,x}(t):= W^j(t) -(\tfrac{t}{s} \wedge 1) (W^j(s)-x).$$ 
		Note that for all $r$-tuples of pairs $(s_j,x_j)$ the processes $(B^{j,s_j,x_j})_{j=1}^r$ are distributed as $(B^i)_{i=1}^r.$ Applying Lemma \ref{coupled} we see that if $(\mathbf x_n,\mathbf s_n,\mathbf t_n)_{n\in\mathbb N}$ converges to $(\mathbf x,\mathbf s, \mathbf t) \in \mathbb R^n\times (0,1]^r\times (0,1]^m,$ we have $$L_0^{\ell_k(\mathbf B^{\mathbf s_n, \mathbf x_n})}(t_{n,k})\to L_0^{\ell_k(\mathbf B^{\mathbf s, \mathbf x})}(t_k)$$ for each $1\le k \le n$ almost surely. Consequently, the exponential of the sum over $k$ also converges almost surely. To show convergence of the expectations of those exponentials we need uniform integrability of the exponentials. We will prove this in a way that also gives global boundedness of the function $f.$ Since local time is an increasing process we have the automatic bound $f(\mathbf{x},\mathbf{s},\mathbf{t}) \leq f(\mathbf{x},\mathbf{s},\ind).$ To prove uniform integrability we will upper-bound $f(\mathbf{x},\mathbf{s},\ind)$ in $L^2$ independently of $s_j,x_j$. First, apply H\"older's inequality to obtain 
		$$\mathbf E \bigg[ \exp\bigg(\sum_{k=1}^mL_0^{\ell_k(\mathbf B^{\mathbf s, \mathbf x})}(1)\bigg)\bigg] \leq \prod_{k=1}^m \mathbf E \bigg[\exp\bigg( mL_0^{\ell_k(\mathbf B^{\mathbf s, \mathbf x})}(1)\bigg)\bigg]^{1/m}.$$ 
	Thus, it suffices to show each term in the product above is finite. Let us take any $\ell(\mathbf{v}):=a_1v_1+\cdots+a_rv_r$.	Assume without loss of generality that $s_1\leq \cdots \leq s_r.$ Then on each subinterval, $[s_j,s_{j+1}]$, we have $$\ell(\mathbf B^{\mathbf s, \mathbf x})(t)= \sum_{i=1}^j a_i[W^i(t)-(W^i(s_i)-x_i)]+\sum_{i=j+1}^r a_i[W^i(t)-\tfrac{t}{s_i}(W^i(s_i)-x_i)].$$ 
{Let us set
\begin{align}\label{ddefg}
    D_{1,j}:=\sum_{i=1}^j a_i^2, \qquad D_{2,j}:=\sum_{i=j+1}^r a_i^2, \qquad \gamma:=D_{2,j} \cdot \big[\sum_{i=j+1}^r a_i^2/s_i\big]^{-1}.
\end{align}
Note that $\ell(\mathbf B^{\mathbf s, \mathbf x})(t+s_j)-\ell(\mathbf B^{\mathbf s, \mathbf x})(s_j)$ has mean $t\sum_{i=j+1}^r \frac{a_ix_i}{s_i}$, and for $s,t\in [0,s_{j+1}-s_j]$ its covariance function is given by
\begin{align*}
    C(s,t) & :=(s\wedge t) \cdot D_{1,j}+\sum_{i=j+1}^r a_i^2\big[(s\wedge t)-\tfrac{ts}{s_i}\big]  =(s\wedge t) \cdot D_{1,j}+\big[(s\wedge t)-\tfrac{ts}{\gamma}\big] \cdot D_{2,j}.
\end{align*}
As $s_1\le \cdots \le s_r$, one can check that $\gamma \ge (s_{j+1}-s_j)$. Thus, $$\ell(\mathbf B^{\mathbf s, \mathbf x})(t+s_j)-\ell(\mathbf B^{\mathbf s, \mathbf x})(s_j) \stackrel{d}{=} U_1(t)+U_2(t)$$ as processes on $[0,s_{j+1}-s_j]$ where $U_1$ is a Brownian bridge with diffusion rate $D_{2,j}$  from $0$ to $\gamma\sum_{i=j+1}^r a_ix_i/s_i$ on $[0,\gamma]$ and $U_2$ is an independent Brownian motion with diffusion rate $D_{1,j}$.}

\smallskip
 
 On $[0,s_{j+1}-s_j],$ we have shown that $\ell(\mathbf B^{\mathbf s, \mathbf x})-\ell(\mathbf B^{\mathbf s, \mathbf x})(s_j)$ is distributed as the independent sum of a Brownian bridge (not necessarily ending at zero) and a Brownian motion whose diffusivities sum to $\|\ell\|_2^2=a_1^2+\cdots+a_r^2.$ Consequently, by breaking up the expectation on each of these subintervals (by again applying H\"older's inequality) and using the fact that each of these subintervals has size at most 1, it suffices to bound $\mE[e^{\theta L_0^X(1)}]$ for all $\theta>0$ where $X:=a+cB+dW$, where $B$ is a standard Brownian bridge from $0$ to $x$ on $[0,1]$ and $W$ is an independent standard Brownian motion on $[0,1]$. More precisely, we need a bound which is uniform over all $a\in\mathbb R$ and $c^2+d^2=1$. This is precisely done in Lemma \ref{tl}. This completes the proof.
		\end{proof}

We remark that Definition \ref{d.bcp} can be naturally extended to Brownian concatenated processes defined on $[0, T]$ for any fixed $T>0$. Both Lemma \ref{6.1} and Lemma \ref{bcts} continue to hold under this setting.
 
		\bibliographystyle{alpha}
		\bibliography{refs.bib}

\begin{thebibliography}{HCGCC23}

\bibitem[AC22]{ac22}
Arka Adhikari and Sourav Chatterjee.
\newblock {An invariance principle for the 1D KPZ equation}.
\newblock {\em arXiv preprint arXiv:2208.02492}, 2022.

\bibitem[ACQ11]{ACQ}
Gideon Amir, Ivan Corwin, and Jeremy Quastel.
\newblock Probability distribution of the free energy of the continuum directed
  random polymer in {$1+1$} dimensions.
\newblock {\em Comm. Pure Appl. Math.}, 64(4):466--537, 2011.

\bibitem[AKQ14]{akq}
Tom Alberts, Konstantin Khanin, and Jeremy Quastel.
\newblock {The intermediate disorder regime for directed polymers in dimension
  $1+1$}.
\newblock {\em The Annals of Probability}, 42(3):1212 -- 1256, 2014.

\bibitem[Ami91]{ami91}
Madjid Amir.
\newblock Sticky {B}rownian motion as the strong limit of a sequence of random
  walks.
\newblock {\em Stochastic processes and their applications}, 39(2):221--237,
  1991.

\bibitem[BC95]{BC95}
Lorenzo Bertini and Nicoletta Cancrini.
\newblock The stochastic heat equation: {F}eynman-{K}ac formula and
  intermittence.
\newblock {\em J. Statist. Phys.}, 78(5-6):1377--1401, 1995.

\bibitem[BC14]{bigmac}
Alexei Borodin and Ivan Corwin.
\newblock Macdonald processes.
\newblock {\em Probability Theory and Related Fields}, 158(1-2):225--400, 2014.

\bibitem[BC17]{bc}
Guillaume Barraquand and Ivan Corwin.
\newblock Random-walk in beta-distributed random environment.
\newblock {\em Probability Theory and Related Fields}, 167(3-4):1057--1116,
  2017.

\bibitem[BG97]{BG97}
Lorenzo Bertini and Giambattista Giacomin.
\newblock Stochastic {B}urgers and {KPZ} equations from particle systems.
\newblock {\em Comm. Math. Phys.}, 183(3):571--607, 1997.

\bibitem[BGK98]{BGK98}
Denis Bernard, Krzysztof Gawedzki, and Antti Kupiainen.
\newblock Slow modes in passive advection.
\newblock {\em Journal of Statistical Physics}, 90:519--569, 1998.

\bibitem[BLD20]{bld}
Guillaume Barraquand and Pierre Le~Doussal.
\newblock Moderate deviations for diffusion in time dependent random media.
\newblock {\em Journal of Physics A: Mathematical and Theoretical},
  53(21):215002, 2020.

\bibitem[BMP04]{ew4}
Carlo Boldrighini, Robert~A Minlos, and Alessandro Pellegrinotti.
\newblock Random walks in quenched iid space-time random environment are always
  as diffusive.
\newblock {\em Probability theory and related fields}, 129(1):133, 2004.

\bibitem[BR20]{mark}
Guillaume Barraquand and Mark Rychnovsky.
\newblock Large deviations for sticky {B}rownian motions.
\newblock {\em Electronic Journal of Probability}, 25, 2020.

\bibitem[BRAS06]{balazs2006random}
M{\'a}rton Bal{\'a}zs, Firas Rassoul-Agha, and Timo Sepp{\"a}l{\"a}inen.
\newblock The random average process and random walk in a space-time random
  environment in one dimension.
\newblock {\em Communications in mathematical physics}, 266(2):499--545, 2006.

\bibitem[BW21]{dom2}
Dom Brockington and Jon Warren.
\newblock The {B}ethe ansatz for sticky {B}rownian motions.
\newblock {\em arXiv preprint arXiv:2104.06482}, 2021.

\bibitem[BW22]{dom}
Dom Brockington and Jon Warren.
\newblock At the edge of a cloud of {B}rownian particles.
\newblock {\em arXiv preprint arXiv:2208.11952}, 2022.

\bibitem[CFKL95]{CFKL95}
Misha Chertkov, Gregory Falkovich, Igor Kolokolov, and Vladmir Lebedev.
\newblock Normal and anomalous scaling of the fourth-order correlation function
  of a randomly advected passive scalar.
\newblock {\em Physical Review E}, 52(5):4924, 1995.

\bibitem[CG17]{gu}
Ivan Corwin and Yu~Gu.
\newblock {Kardar--Parisi--Zhang equation and large deviations for random walks
  in weak random environments}.
\newblock {\em Journal of Statistical Physics}, 166:150--168, 2017.

\bibitem[CGST20]{cgst20}
Ivan Corwin, Promit Ghosal, Hao Shen, and Li-Cheng Tsai.
\newblock {Stochastic PDE limit of the six vertex model}.
\newblock {\em Communications in Mathematical Physics}, 375(3):1945--2038,
  2020.

\bibitem[Cha22]{cha22}
Sourav Chatterjee.
\newblock Local {KPZ} behavior under arbitrary scaling limits.
\newblock {\em Communications in Mathematical Physics}, pages 1--28, 2022.

\bibitem[CLDR10]{cldr}
Pasquale Calabrese, Pierre Le~Doussal, and Alberto Rosso.
\newblock Free-energy distribution of the directed polymer at high temperature.
\newblock {\em Europhysics Letters}, 90(2):20002, 2010.

\bibitem[Cor12]{Cor12}
Ivan Corwin.
\newblock The {K}ardar{--P}arisi{--Z}hang equation and universality class.
\newblock {\em Random Matrices: Theory Appl.}, 1(01):1130001, 2012.

\bibitem[Cor18]{corwin2018exactly}
Ivan Corwin.
\newblock Exactly solving the kpz equation.
\newblock {\em arXiv preprint arXiv:1804.05721}, 2018.

\bibitem[CS20]{CS20}
Ivan Corwin and Hao Shen.
\newblock Some recent progress in singular stochastic partial differential
  equations.
\newblock {\em Bulletin of the American Mathematical Society}, 57(3):409--454,
  2020.

\bibitem[CST18]{cst18}
Ivan Corwin, Hao Shen, and Li-Cheng Tsai.
\newblock {$\operatorname{ASEP}(q,j)$ converges to the KPZ equation}.
\newblock {\em Annales de l'Institut Henri Poincaré, Probabilités et
  Statistiques}, 54(2):995 -- 1012, 2018.

\bibitem[CSZ16]{poly}
Francesco Caravenna, Rongfeng Sun, and Nikos Zygouras.
\newblock Polynomial chaos and scaling limits of disordered systems.
\newblock {\em Journal of the European Mathematical Society}, 19(1):1--65,
  2016.

\bibitem[CT17]{ct17}
Ivan Corwin and Li-Cheng Tsai.
\newblock {KPZ equation limit of higher-spin exclusion processes}.
\newblock {\em The Annals of Probability}, 45(3):1771 -- 1798, 2017.

\bibitem[CW17]{CW17}
Ajay Chandra and Hendrik Weber.
\newblock Stochastic {PDE}s, regularity structures, and interacting particle
  systems.
\newblock In {\em Annales de la facult{\'e} des sciences de Toulouse
  Math{\'e}matiques}, volume~26, pages 847--909, 2017.

\bibitem[Das24]{das}
Sayan Das.
\newblock {Temporal increments of the KPZ equation with general initial data}.
\newblock {\em Electronic Journal of Probability}, 29:1--28, 2024.

\bibitem[DDP24]{DDP24}
Sayan Das, Hindy Drillick, and Shalin Parekh.
\newblock {Multiplicative SHE limit of random walks in space--time random
  environments}.
\newblock {\em Probability Theory and Related Fields}, pages 1--83, 2024.

\bibitem[DG22]{ew6}
Alexander Dunlap and Yu~Gu.
\newblock A quenched local limit theorem for stochastic flows.
\newblock {\em Journal of Functional Analysis}, 282(6):109372, 2022.

\bibitem[Dot10]{dot}
Victor Dotsenko.
\newblock Bethe ansatz derivation of the {T}racy-{W}idom distribution for
  one-dimensional directed polymers.
\newblock {\em Europhysics Letters}, 90(2):20003, 2010.

\bibitem[DT16]{dembo}
Amir Dembo and Li-Cheng Tsai.
\newblock {Weakly asymmetric non-simple exclusion process and the
  Kardar--Parisi--Zhang equation}.
\newblock {\em Communications in Mathematical Physics}, 341:219--261, 2016.

\bibitem[EF16]{ellis2016brownian}
Tom Ellis and Ohad~N Feldheim.
\newblock {The Brownian web is a two-dimensional black noise}.
\newblock In {\em Annales de l’Institut Henri Poincar{\'e}-Probabilit{\'e}s
  et Statistiques}, volume~52, pages 162--172, 2016.

\bibitem[Fel52]{feller}
William Feller.
\newblock The parabolic differential equations and the associated semi-groups
  of transformations.
\newblock {\em Annals of Mathematics}, pages 468--519, 1952.

\bibitem[FK99]{ew1}
Albert Fannjiang and Tomasz Komorowski.
\newblock {Turbulent diffusion in Markovian flows}.
\newblock {\em Annals of Applied Probability}, pages 591--610, 1999.

\bibitem[Flo14]{flo}
Gregorio R~Moreno Flores.
\newblock On the (strict) positivity of solutions of the stochastic heat
  equation.
\newblock {\em The Annals of Probability}, pages 1635--1643, 2014.

\bibitem[FS10]{FS10}
Patrik~L Ferrari and Herbert Spohn.
\newblock Random growth models.
\newblock {\em arXiv:1003.0881}, 2010.

\bibitem[GH04]{gaw}
Krzysztof Gawedzki and P{\'e}ter Horvai.
\newblock Sticky behavior of fluid particles in the compressible {K}raichnan
  model.
\newblock {\em Journal of statistical physics}, 116:1247--1300, 2004.

\bibitem[GIP15]{GIP15}
Massimiliano Gubinelli, Peter Imkeller, and Nicolas Perkowski.
\newblock Paracontrolled distributions and singular {PDE}s.
\newblock In {\em Forum of Mathematics, Pi}, volume~3. Cambridge University
  Press, 2015.

\bibitem[GJ14]{GJ14}
Patr{\'\i}cia Gon{\c{c}}alves and Milton Jara.
\newblock Nonlinear fluctuations of weakly asymmetric interacting particle
  systems.
\newblock {\em Arch. Ration. Mech. Anal.}, 212(2):597--644, 2014.

\bibitem[GJ17]{gj17}
Patr{\'\i}cia Gon{\c{c}}alves and Milton Jara.
\newblock Stochastic {B}urgers equation from long range exclusion interactions.
\newblock {\em Stochastic Processes and their Applications},
  127(12):4029--4052, 2017.

\bibitem[GK95]{GK95}
Krzysztof Gawedzki and Antti Kupiainen.
\newblock Anomalous scaling of the passive scalar.
\newblock {\em Physical review letters}, 75(21):3834, 1995.

\bibitem[GK06]{GK96}
Krzysztof Gawedzki and Antti Kupiainen.
\newblock {University in turbulence: An exactly solvable model}.
\newblock In {\em Low-Dimensional Models in Statistical Physics and Quantum
  Field Theory: Proceedings of the 34. Internationale Universit{\"a}tswochen
  f{\"u}r Kern-und Teilchenphysik Schladming, Austria, March 4--11, 1995},
  pages 71--105. Springer, 2006.

\bibitem[GP17]{GP17}
Massimiliano Gubinelli and Nicolas Perkowski.
\newblock {KPZ} reloaded.
\newblock {\em Commun. Math. Phys.}, 349(1):165--269, 2017.

\bibitem[GP18]{GP18}
Massimiliano Gubinelli and Nicolas Perkowski.
\newblock Energy solutions of {KPZ} are unique.
\newblock {\em J. Amer. Math. Soc.}, 31(2):427--471, 2018.

\bibitem[GV00]{GV00}
Krzysztof Gawedzki and Massimo Vergassola.
\newblock Phase transition in the passive scalar advection.
\newblock {\em Physica D: Nonlinear Phenomena}, 138(1-2):63--90, 2000.

\bibitem[Hai13]{Hai13}
Martin Hairer.
\newblock Solving the {KPZ} equation.
\newblock {\em Annals of Mathematics}, pages 559--664, 2013.

\bibitem[Hai14]{Hai14}
Martin Hairer.
\newblock A theory of regularity structures.
\newblock {\em Invent Math.}, 198(2):269--504, 2014.

\bibitem[HCGCC23]{hass23}
Jacob~B Hass, Aileen~N Carroll-Godfrey, Ivan Corwin, and Eric~I Corwin.
\newblock Anomalous fluctuations of extremes in many-particle diffusion.
\newblock {\em Physical Review E}, 107(2):L022101, 2023.

\bibitem[HL81]{hl81}
John~Michael Harrison and Austin~J Lemoine.
\newblock Sticky {B}rownian motion as the limit of storage processes.
\newblock {\em Journal of Applied Probability}, 18(1):216--226, 1981.

\bibitem[HL15]{HL16}
Martin Hairer and Cyril Labb{\'e}.
\newblock A simple construction of the continuum parabolic {A}nderson model on
  $\mathbb{R}^2$.
\newblock {\em Electronic Communications in Probability}, 20:1--11, 2015.

\bibitem[HL18]{hairer2018multiplicative}
Martin Hairer and Cyril Labb{\'e}.
\newblock Multiplicative stochastic heat equations on the whole space.
\newblock {\em Journal of the European Mathematical Society}, 20(4):1005--1054,
  2018.

\bibitem[HQ18]{hq18}
Martin Hairer and Jeremy Quastel.
\newblock {A class of growth models rescaling to KPZ}.
\newblock In {\em Forum of Mathematics, Pi}, volume~6, page~e3. Cambridge
  University Press, 2018.

\bibitem[HW09a]{HW09}
Chris Howitt and Jon Warren.
\newblock Consistent families of {B}rownian motions and stochastic flows of
  kernels.
\newblock {\em The Annals of Probability}, 37(4), jul 2009.

\bibitem[HW09b]{hw09b}
Chris Howitt and Jon Warren.
\newblock Dynamics for the {B}rownian web and the erosion flow.
\newblock {\em Stochastic Processes and their Applications}, 119(6):2028--2051,
  2009.

\bibitem[IM63]{ito1963brownian}
Kiyoshi It{\^o} and H.P. McKean.
\newblock Brownian motions on a half line.
\newblock {\em Illinois journal of mathematics}, 7(2):181--231, 1963.

\bibitem[JRAS19]{joseph2019independent}
Mathew Joseph, Firas Rassoul-Agha, and Timo Sepp{\"a}l{\"a}inen.
\newblock Independent particles in a dynamical random environment.
\newblock In {\em Probability and Analysis in Interacting Physical Systems: In
  Honor of SRS Varadhan, Berlin, August, 2016}, pages 75--121. Springer, 2019.

\bibitem[KLD23]{alex}
Alexandre Krajenbrink and Pierre Le~Doussal.
\newblock {Crossover from the macroscopic fluctuation theory to the
  Kardar-Parisi-Zhang equation~controls the large deviations beyond
  Einstein{\textquotesingle}s diffusion}.
\newblock {\em Physical Review E}, 107(1), jan 2023.

\bibitem[KLO12]{ew5}
Tomasz Komorowski, Claudio Landim, and Stefano Olla.
\newblock {\em {Fluctuations in Markov processes: time symmetry and martingale
  approximation}}, volume 345.
\newblock Springer Science \& Business Media, 2012.

\bibitem[KO01]{ew2}
Tomasz Komorowski and Stefano Olla.
\newblock On homogenization of time-dependent random flows.
\newblock {\em Probability theory and related fields}, 121:98--116, 2001.

\bibitem[KPZ86]{kpz}
Mehran Kardar, Giorgio Parisi, and Yi-Cheng Zhang.
\newblock Dynamic scaling of growing interfaces.
\newblock {\em Physical Review Letters}, 56(9):889, 1986.

\bibitem[Kra68]{kar68}
Robert~H Kraichnan.
\newblock Small-scale structure of a scalar field convected by turbulence.
\newblock {\em The Physics of Fluids}, 11(5):945--953, 1968.

\bibitem[KS88]{konno}
Nuri Konno and Tokuzo Shiga.
\newblock Stochastic partial differential equations for some measure-valued
  diffusions.
\newblock {\em Probability theory and related fields}, 79(2):201--225, 1988.

\bibitem[KS14]{karatzas}
Ioannis Karatzas and Steven Shreve.
\newblock {\em Brownian motion and stochastic calculus}, volume 113.
\newblock springer, 2014.

\bibitem[Kup10]{kup10}
Antti Kupiainen.
\newblock {\em Lessons for turbulence}.
\newblock Springer, 2010.

\bibitem[LDT17]{ldt}
Pierre Le~Doussal and Thimoth{\'e}e Thiery.
\newblock {Diffusion in time-dependent random media and the Kardar-Parisi-Zhang
  equation}.
\newblock {\em Physical Review E}, 96(1):010102, 2017.

\bibitem[LJR04a]{lejan}
Yves Le~Jan and Olivier Raimond.
\newblock Flows, coalescence and noise.
\newblock {\em Ann. Probab.}, 32(2):1247--1315, 2004.

\bibitem[LJR04b]{jan2004sticky}
Yves Le~Jan and Olivier Raimond.
\newblock Sticky flows on the circle and their noises.
\newblock {\em Probability Theory and Related Fields}, 129(1):63--82, 2004.

\bibitem[Mue91]{mue91}
Carl Mueller.
\newblock On the support of solutions to the heat equation with noise.
\newblock {\em Stochastics: An International Journal of Probability and
  Stochastic Processes}, 37(4):225--245, 1991.

\bibitem[MW17]{WM}
Jean-Christophe Mourrat and Hendrik Weber.
\newblock Global well-posedness of the dynamic $\phi^{4}$ model in the plane.
\newblock {\em The Annals of Probability}, 45(4):2398--2476, 2017.

\bibitem[Nua06]{nualart2006malliavin}
David Nualart.
\newblock {\em The Malliavin calculus and related topics}, volume 1995.
\newblock Springer, 2006.

\bibitem[Par18]{Par19}
Shalin Parekh.
\newblock The {KPZ} limit of {ASEP} with boundary.
\newblock {\em Communications in Mathematical Physics}, 365(2):569--649, sep
  2018.

\bibitem[Pit99]{pitman}
Jim Pitman.
\newblock The distribution of local times of a {B}rownian bridge.
\newblock In {\em S\'{e}minaire de {P}robabilit\'{e}s, {XXXIII}}, volume 1709
  of {\em Lecture Notes in Math.}, pages 388--394. Springer, Berlin, 1999.

\bibitem[QS15]{QS15}
Jeremy Quastel and Herbert Spohn.
\newblock The one-dimensional {KPZ} equation and its universality class.
\newblock {\em J. Stat. Phys.}, 160(4):965--984, 2015.

\bibitem[Qua11]{Qua11}
Jeremy Quastel.
\newblock Introduction to {KPZ}.
\newblock {\em Current developments in mathematics}, 2011(1), 2011.

\bibitem[RAS05]{ew3}
Firas Rassoul-Agha and Timo Sepp{\"a}l{\"a}inen.
\newblock An almost sure invariance principle for random walks in a space-time
  random environment.
\newblock {\em Probability theory and related fields}, 133(3):299--314, 2005.

\bibitem[RS15]{rs15}
Mikl{\'o}s~Z R{\'a}cz and Mykhaylo Shkolnikov.
\newblock Multidimensional sticky {B}rownian motions as limits of exclusion
  processes.
\newblock {\em Annals of Applied Probability}, 25(3):1155--1188, 2015.

\bibitem[RY99]{RY99}
Daniel Revuz and Marc Yor.
\newblock {\em Continuous martingales and {B}rownian motion}, volume 293 of
  {\em Grundlehren der mathematischen Wissenschaften [Fundamental Principles of
  Mathematical Sciences]}.
\newblock Springer-Verlag, Berlin, third edition, 1999.

\bibitem[SS00]{ss00}
Boris~I Shraiman and Eric~D Siggia.
\newblock Scalar turbulence.
\newblock {\em Nature}, 405(6787):639--646, 2000.

\bibitem[SS10]{ss}
Tomohiro Sasamoto and Herbert Spohn.
\newblock Exact height distributions for the {KPZ} equation with narrow wedge
  initial condition.
\newblock {\em Nuclear Physics B}, 834(3):523--542, 2010.

\bibitem[SSG11]{scsm}
Oded Schramm, Stanislav Smirnov, and Christophe Garban.
\newblock {On the scaling limits of planar percolation}.
\newblock {\em The Annals of Probability}, 39(5):1768 -- 1814, 2011.

\bibitem[SSS09]{sss0}
Emmanuel Schertzer, Rongfeng Sun, and Jan Swart.
\newblock {Special points of the Brownian net}.
\newblock {\em Electronic Journal of Probability}, 14(none):805 -- 864, 2009.

\bibitem[SSS14]{sss}
Emmanuel Schertzer, Rongfeng Sun, and Jan Swart.
\newblock Stochastic flows in the {B}rownian web and net.
\newblock {\em Mem. Amer. Math. Soc.}, 227(1065):vi+160, 2014.

\bibitem[SSS17]{sss2}
Emmanuel Schertzer, Rongfeng Sun, and Jan Swart.
\newblock {The Brownian web, the Brownian net, and their universality}.
\newblock {\em Advances in disordered systems, random processes and some
  applications}, pages 270--368, 2017.

\bibitem[Tsi04a]{tsirelson2004boris}
Boris Tsirelson.
\newblock Boris tsirelson: Scaling limit, noise, stability.
\newblock {\em {Lectures on Probability Theory and Statistics: Ecole d'Et{\'e}
  de Probabilit{\'e}s de Saint-Flour XXXII-2002}}, pages 1--106, 2004.

\bibitem[Tsi04b]{tsir}
Boris Tsirelson.
\newblock Nonclassical stochastic flows and continuous products.
\newblock {\em Probability Surveys}, 1:173--298, 2004.

\bibitem[Wal86]{Wal86}
John~B Walsh.
\newblock An introduction to stochastic partial differential equations.
\newblock In {\em {\'E}cole d'{\'E}t{\'e} de Probabilit{\'e}s de Saint Flour
  XIV-1984}, pages 265--439. Springer, 1986.

\bibitem[War15]{war}
Jon Warren.
\newblock Sticky particles and stochastic flows.
\newblock {\em In Memoriam Marc Yor-S{\'e}minaire de Probabilit{\'e}s XLVII},
  pages 17--35, 2015.

\bibitem[Yan22]{yang22}
Kevin Yang.
\newblock {KPZ equation from non-simple variations on open ASEP}.
\newblock {\em Probability Theory and Related Fields}, 183(1-2):415--545, 2022.

\bibitem[Yan23a]{yang23b}
Kevin Yang.
\newblock Hairer-{Q}uastel universality in non-stationarity via energy solution
  theory.
\newblock {\em Electronic Journal of Probability}, 28:1--26, 2023.

\bibitem[Yan23b]{yang23}
Kevin Yang.
\newblock {Kardar--Parisi--Zhang equation from long-range exclusion processes}.
\newblock {\em Communications in Mathematical Physics}, pages 1--129, 2023.

\bibitem[Yu16]{yu}
Jinjiong Yu.
\newblock {Edwards-Wilkinson fluctuations in the Howitt-Warren flows}.
\newblock {\em Stochastic Processes and their Applications}, 126(3):948--982,
  2016.

\end{thebibliography}
		
	\end{document}